\documentclass[12pt]{article}

\usepackage[T1]{fontenc}
\usepackage{amsmath}
\usepackage{amssymb}
\usepackage{amsfonts}
\usepackage{fullpage}
\usepackage{latexsym,amssymb,amsmath}
\usepackage{euscript}

\usepackage{tikz}
\usetikzlibrary{patterns}

\usepackage[colors]{optsys}

\allowdisplaybreaks

\newcommand{\NATURE}{{\cal K}}
\renewcommand{\nature}{k}
\newcommand{\BELIEF}{\Delta\np{\NATURE}}
\newcommand{\belief}{p}
\newcommand{\beliefbis}{q}

\newcommand{\Entropy}{H}
\renewcommand{\borel}[1]{\tribu{B}({#1})}
\newcommand{\bpborel}[1]{\tribu{B}\bp{#1}}
\newcommand{\InformationStructures}{\mathfrak{I}}
\newcommand{\RestrictedSupportFunction}[1]{\widetilde{\sigma}_{#1}}
\newcommand{\InverseRestrictedSupportFunction}[1]{\theta_{#1}}
\newcommand{\FromACTIONPAYOFFSETtoConvBELIEF}{\Sigma}
\newcommand{\FromConvBELIEFtoACTIONPAYOFFSET}{\Theta}

\newcommand{\Conv}[1]{\Gamma\np{#1}}

\newcommand{\RegularConv}[1]{\Gamma_{c}\bp{#1}}
\newcommand{\MaximalExpectedUtility}{maximal expected utility}
\newcommand{\ExpectedUtilityMaximizer}{expected utility maximizer}
\newcommand{\ActionPayoff}{utility act}
\newcommand{\ActionPayoffSet}{utility act set}

\newcommand{\ACTIONPAYOFFSET}{\mathcal{A}\np{\NATURE}}
\newcommand{\RegularActionPayoffSet}{c-utility act set}
\newcommand{\REGULARACTIONPAYOFFSET}{\mathcal{A}_{c}\np{\NATURE}}
\newcommand{\RegularValueFunction}{c-value function}
\newcommand{\ACTION}{{G}}
\newcommand{\action}{g}
\newcommand{\ACTIONbis}{{H}}
\newcommand{\actionbis}{h}
\newcommand{\actionter}{i}
\renewcommand{\fonctionun}{g}
\renewcommand{\fonctiondeux}{h}
\newcommand{\NEUTRAL}{\mathfrak{A}}
\newcommand{\neutral}{a}
\newcommand{\NEUTRALbis}{\mathfrak{B}}
\newcommand{\neutralbis}{b}
\newcommand{\NEUTRALpayoff}{U}

\newcommand{\MORE}{{M}}

\newcommand{\more}{m}
\newcommand{\LESS}{{L}}

\newcommand{\less}{l}
\newcommand{\MEDIUM}{{T}}

\newcommand{\medium}{t}
\renewcommand{\ValueFunction}[1]{V_{#1}}
\newcommand{\VoI}{\mathrm{V\!o\!I}}
\renewcommand{\Cone}{X}
\renewcommand{\Primal}{X}
\renewcommand{\Primalbis}{X'}
\renewcommand{\primalbis}{x'}

\newcommand{\signed}{y}
\renewcommand{\sequence}[2]{\np{#1}_{#2}}           
\newcommand{\bpsequence}[2]{\bp{#1}_{#2}}           
\newcommand{\scalpro}[2]{\left\langle#2 \mid \:#1\right\rangle}  
\newcommand{\stardifference}{\!\sim\!}

\title{Increasing Value of Information\\
Implies Separable Utility}

\author{Michel \textsc{De Lara}\thanks{%
CERMICS, \'Ecole nationale des ponts et chauss\'ees, IP Paris, France
E-mail: \texttt{michel.delara@enpc.fr}}}

\begin{document}

\maketitle

    \begin{abstract}%
We consider decision-making under incomplete information about an unknown state of nature.
Utility acts (that is, utility vectors indexed by states of nature) 
and beliefs (probability distributions over the states of nature)
are naturally paired by bilinear duality, giving the expected utility.
With this pairing, an expected utility maximizer (DM) is characterized
by a continuous closed convex comprehensive set of utility acts (c-utility act set).
We show that DM~M values information more than DM~L if and only if
the c-utility act set of DM~M is obtained by Minkowski addition from the c-utility act set of DM~L.
In the classic setting of decision theory, this is interpreted as
the equivalence between more valuable information, on the one hand,
and multiplying decisions and adding utility, on the other hand
(additively separable utility).
We also introduce the algebraic structure of dioid to describe two operations between DMs:
union (``adding'' options) and fusion (multiplying options and adding utilities).
We say that DM~M is more flexible by union (resp. by fusion) than DM~L
if DM~M is obtained by union (resp. by fusion) from DM~L. 
Our main result is that DM~M values information more than DM~L if and only if
DM~M is more flexible by fusion than DM~L. We also study when flexibility by union can lead
to more valuable information.
    \end{abstract}%






 \section{Introduction}
 \label{Introduction}

 An individual goes to an unknown
restaurant, and selects starters (and main course, but no dessert) in the menu;
we denote this decision-maker (DM) by~S.
As the restaurant is unknown, DM~S might be interested in gathering information
to make better choices.
Information has value (see the seminal paper \cite{Marschak-Miyasawa:1968}), and
we ask the question: what changes in the problem of DM~S make information more valuable?
 More generally, in this paper, we focus on the question:
 for a given economic agent, what changes in decision variables and utility function
 make information more valuable?

 The value of information (VoI) is a well-known concept in economics: it is the (nonnegative) difference
between the maximal expected utility with and without knowledge about an unknown state of nature.
It is not the place here to review the large literature on VoI, as our focus is on
the little studied subject of interpersonal comparison of VoIs. 
Interpersonal comparison of the value of information was at the heart of our 2009 working paper~\cite{DeLara:2009}.
More recently, the question has been studied in \cite{whitmeyer2024makinginformationvaluable}, and
 \cite[Theorem~3.1]{whitmeyer2024makinginformationvaluable} obtains that more valuable information
 is equivalent to convexity of the difference of two (convex) value functions
 (a result that can also be found in \cite[Proposition~0]{Jones-Ostroy:1984}).
 However, \cite{whitmeyer2024makinginformationvaluable} only provides sufficient conditions
on decision variables and utility function that leads to more valuable information.
 In this paper, we build upon the convex analysis framework developed in~\cite{DeLara-Gossner:2020}
 --- exploiting duality between utility~acts and beliefs
\cite{Gilboa-Schmeidler:1989,MacMarRusAmbiguity2006,cerreia2011complete} --- 
and we obtain a necessary and sufficient condition, that bears on the decision problem,
for more valuable information.

We illustrate our main result on the unknown restaurant problem.
We consider two new DMs: DM~SorD selects (aside main course)
either starters or desserts (but not both);
DM~SandD selects (aside main course) both starters and desserts.
We show that DM~SandD values information more than DM~S (whatever belief about
the restaurant), whereas there is no systematic comparison between DM~SorD and
DM~S.

The paper is organized as follows.
In Sect.~\ref{Abstract_decision-maker}, we 
recall the duality between utility acts and beliefs in classic decision models,
then move to define abstract decision problems.
In Sect.~\ref{Fusion_makes_information_more_valuable},
we define information structures (and their comparison),
and we define the value of information for a decision-maker.
Then, we state the main result of the paper:
the equivalence between more valuable information, on the one hand,
and multiplying decisions and adding utility, on the other hand.
In Sect.~\ref{Flexibility_and_more_valuable_information},
we introduce dioids of expected utility maximizers, union and fusion of decisions makers,
and two kinds of flexibility (by union or by fusion).
With this, the main result of the paper can be reformulated as 
the equivalence between more valuable information and more flexibility by fusion.
As flexibilities by union or fusion are not exclusive, 
we also study the relation between flexibility by union and more valuable information.
Finally, we relate our results to those in three recent papers
\cite{whitmeyer2024makinginformationvaluable},
\cite{Denti:2022} and \cite{Yoder:2022}.
Sect.~\ref{Conclusion} concludes.
The Appendix contains some material in convex analysis, and technical Propositions and proofs.

\section{Abstract decision-maker}
\label{Abstract_decision-maker}

In~\S\ref{Duality_between_payoffs_and_beliefs},
we introduce a duality between utility~acts and beliefs like in~\cite{DeLara-Gossner:2020}.
In~\S\ref{Classic_decision_problems_under_imperfect_information},
we revisit classic decision problems under imperfect information with this duality.
In~\S\ref{Abstract_decision_problems_and_value_function},
we define abstract decision problems and decision makers.

\subsection{Duality between utility~acts and beliefs}
\label{Duality_between_payoffs_and_beliefs}

We denote by $\RR$ the set of real numbers,
$\barRR = \ClosedIntervalClosed{-\infty}{+\infty} = \RR \cup \na{-\infty,+\infty} $,
$\RR_{+} = \ClosedIntervalOpen{0}{+\infty} $,
$\RR_{++}=\OpenIntervalOpen{0}{+\infty}$.
We consider a nonempty finite set~$\NATURE$, that represents states of nature.
 
\subsubsubsubsection{Functions}\footnote{%
  Adopting usage in mathematics, we follow Serge Lang and use ``function'' only to
  refer to mappings in which the codomain is a set of numbers (i.e. a subset
  of~$\RR$ or $\CC$, or their possible extensions with $\pm \infty$),
  and reserve the term mapping for more general
  codomains.}

Let \( \Primal \subset \RR^{\NATURE} \) be a nonempty set.
For any function \( \fonctionprimal \colon \Primal \to \barRR \),
its \emph{epigraph} is
\( \epigraph\fonctionprimal= \defset{
  \np{\primal,t}\in\Primal\times\RR}%
{\fonctionprimal\np{\primal} \leq t} \subset \RR^{\NATURE}\times\RR \),
its \emph{effective domain} is
\( \dom\fonctionprimal= \defset{\primal\in\Primal}{%
  \fonctionprimal\np{\primal} <+\infty} \).
A function
\( \fonctionprimal \colon \Primal \to \barRR \) is said to be \emph{convex}
if its epigraph is a convex subset of \( \RR^{\NATURE}\times\RR \), 
\emph{proper} if it never takes the
value~$-\infty$ and that \( \dom\fonctionprimal \not = \emptyset \),
\emph{lower semi continuous (\lsc)} if its epigraph is
a closed subset of \( \RR^{\NATURE}\times\RR \).

A function
\( \fonctionprimal \colon \Primal \to \barRR \) is said to be \emph{closed} \cite[p.~15]{Rockafellar:1974}
if it is either \lsc\ and nowhere having the value $-\infty$,
or is the constant function~$-\infty$.
A function \( \fonctionprimal \colon \Primal \to \barRR \) is said to be \emph{closed convex}
if it is either one of the two constant functions~$-\infty$ and~$+\infty$,
or it is a proper convex \lsc\ function.

\subsubsubsubsection{Subsets}

For any subset \( \Primal\subset\RR^{\NATURE} \), we denote by
$\overline{\Primal}$ the \emph{topological closure} of~$\Primal$, 
by $\convexhull \Primal$ (or $\convexhull\np{\Primal}$) the \emph{convex hull} of~$\Primal$ --- that is, the
smallest convex set in~$\RR^{\NATURE}$ containing~$\Primal$ --- by
$\closedconvexhull \Primal$ (or $\closedconvexhull\np{\Primal}$) the
\emph{closed convex hull} of~$\Primal$ --- that is, the smallest closed convex set
in~$\RR^{\NATURE}$ containing~$\Primal$.
%
A subset $\Cone\subset\RR^{\NATURE}\times\RR$ is said to be a \emph{cone}\footnote{%
  Hence, a cone does not necessarily contain the origin~$0$.}  if
$\RR_{++}\Cone \subset \Cone$.
%
The \emph{Minkowski sum} of two subsets $\Primal, \Primalbis\subset\RR^{\NATURE}$ is
\begin{subequations}
  \begin{align}
    {\Primal+\Primalbis}
    &=
      \defset{\primal+\primalbis}{\primal\in \Primal, \primalbis\in \Primalbis}
     = \bigcup_{\primal\in \Primal} \np{\Primalbis +\primal}
      \eqfinv
      \label{eq:Minkowski_sum}
      \intertext{and the \emph{star-difference} is
\cite[Chapter~III, Example~1.2.6]{Hiriart-Urruty-Lemarechal-I:1993} } 
         \Primal\stardifference\Primalbis
        &=
        \defset{\primal\in\RR^{\NATURE}}%
        {\primal + \Primalbis \subset \Primal}
        = \bigcap_{\primal\in \Primal} \np{\Primalbis -\primal}
          \eqfinp
          \label{eq:stardifference}
  \end{align}
\end{subequations}

\subsubsubsubsection{Duality}

Now, we turn to duality.
When considered as a primal space, we identify the vectors in~$\RR^{\NATURE}$ with
functions from states of nature in~\( \NATURE \) to the real numbers in~\( \RR \), 
that we call \emph{\ActionPayoff s}\footnote{%
We follow a terminology found in~\cite{Denti:2022}. 
Indeed, an act ``\`a la Savage'' \cite{Savage:1972} is a mapping from states of nature to consequences
and, here, these latter are measured in utility.}.
When considered as a dual space, we identify the vectors in~$\RR^{\NATURE}$ with
\emph{signed measures} on~$\NATURE$, and the elements of the simplex
\( \BELIEF = \defset{\signed \in \RR_{+}^{\NATURE}}{\sum_{\nature \in \NATURE} \signed_{\nature}=1} 
\subset \RR^{\NATURE} \) with \emph{probability distributions} over~$\NATURE$, also called \emph{beliefs}.
The scalar product between a \ActionPayoff~$\primal \in \RR^{\NATURE}$ and a 
signed measure~$\signed \in \RR^{\NATURE}$ is 
\(  \scalpro{\signed}{\primal} = 
\sum_{\nature \in \NATURE} \primal_{\nature} \signed_{\nature} \). 
This scalar product induces a bilinear duality which is at the core of a series of works in nonexpected utility theory, such as
\cite{Gilboa-Schmeidler:1989,MacMarRusAmbiguity2006,cerreia2011complete}. 

A function \( \fonctionprimal \colon \RR^{\NATURE} \to \barRR \) is
said to be \emph{positively homogeneous}  if its epigraph is a cone.
The \emph{support function} of a subset
$\Primal\subset\RR^{\NATURE}$ is the closed convex
positively homogeneous function
\cite[Chapter~V, Definition~2.1.1]{Hiriart-Urruty-Lemarechal-I:1993} 
defined by
\begin{equation}
\SupportFunction{\Primal} \colon \RR^{\NATURE} \to \barRR \eqsepv    
    \SupportFunction{\Primal}\np{\signed}
    =
      \sup_{\primal\in \Primal} \scalpro{\signed}{\primal} \eqsepv \forall \signed\in\RR^{\NATURE}
      \eqfinp
      \label{eq:support_function}
\end{equation}

\subsubsubsubsection{Closed convex functions \( \BELIEF \to \barRR \)}

Because beliefs in the set \( \BELIEF \subset \RR^{\NATURE} \) will play a central role,
we denote by~\emph{\( \Conv{\BELIEF} \) the set of closed convex functions \( \BELIEF \to \barRR \)}.
For any function \( \fonctiondual \in \Conv{\BELIEF} \)
(and, more generally, for any function \( \fonctiondual \colon \BELIEF \to \barRR \)),
we denote by \( \LFM{\fonctiondual} \colon \RR^{\NATURE} \to \barRR \) 
the Fenchel conjugate of the function~\( \fonctionun \) (extended with the value \( +\infty \)
outside of~\( \BELIEF \)), that is,
\begin{equation}
  \LFM{\fonctiondual}\np{\primal}=
  \sup_{\belief\in\BELIEF} \scalpro{\belief}{\primal}-\fonctiondual\np{\belief}
  \eqsepv \forall \primal\in \RR^{\NATURE}
  \eqfinp
  \label{eq:LFM_beliefs}
\end{equation}
We introduce the notation
\begin{equation}
  \RestrictedSupportFunction{\Primal} \colon \BELIEF \to \barRR \eqsepv
\RestrictedSupportFunction{\Primal}\np{\belief}
=\SupportFunction{\Primal}\np{\belief}
\eqsepv \forall \belief \in \BELIEF 
\eqfinv
      \label{eq:support_function_restriction}
\end{equation}
for the \emph{restriction~\( \RestrictedSupportFunction{\Primal} \) of the support function~\( \SupportFunction{\Primal}\)}
in~\eqref{eq:support_function} to the closed convex subset \( \BELIEF \subset \RR^{\NATURE} \).
We have that \( \RestrictedSupportFunction{\Primal} \in \Conv{\BELIEF} \).
Indeed, the support function~\( \SupportFunction{\Primal}\)
in~\eqref{eq:support_function} is a closed convex function,
hence so is the restriction~\( \RestrictedSupportFunction{\Primal} \)
of~\( \SupportFunction{\Primal}\) to the closed convex subset~\( \BELIEF \subset \RR^{\NATURE} \).
As an example that we will use several times later,
we have that \( \RestrictedSupportFunction{\RR_{-}^{\NATURE}} =0 \),
which follows from the fact that \(\SupportFunction{\RR_{-}^{\NATURE}} \)
is the indicator function~\( \Indicator{\RR_{+}^{\NATURE}} \), which takes the value~0
on \( \RR_{+}^{\NATURE} \) hence on \( \BELIEF \subset \RR_{+}^{\NATURE} \) (see for example
\cite[Example.~2.3.1]{Hiriart-Urruty-Lemarechal-I:1993}). 

     \subsubsubsubsection{Closed convex functions \( \BELIEF \to \RR \)}

     We also denote by~\emph{\( \RegularConv{\BELIEF} \subset \Conv{\BELIEF} \)
the set of closed convex functions \( \BELIEF \to \RR \)}.
As  \( \BELIEF \) is bounded polyhedral, by \cite{Gale-Klee-Rockafellar:1968} we get that \( \RegularConv{\BELIEF} \)
is also the set of continuous\footnote{%
  Hence, the subscript~$c$ in \( \RegularConv{\BELIEF} \).}
  and bounded convex functions \( \BELIEF \to \RR \).
As functions in~\( \RegularConv{\BELIEF} \) take real values (and not extended
ones), they are easily amenable to susbstraction operations. This property will
be decisive when properly defining the value of information
in~\S\ref{Information_structures_and_value_of_information},
by avoiding indefinite integrals (mathematical expectations) in Definition~\ref{de:VoI}. 

For instance, the \emph{negentropy} \( -\Entropy \colon \BELIEF \to \barRR \) --- defined by
\begin{equation}
  -\Entropy\np{\belief}= \sum_{\nature\in\NATURE}\belief_\nature \log \belief_\nature
  \eqsepv \forall \belief\in\BELIEF
  \eqfinv
  \label{eq:negentropy} 
\end{equation}
with the convention that \( 0\log 0 = 0 \) --- belongs to~\( \RegularConv{\BELIEF} \).

\subsection{Classic decision problems under imperfect information}
\label{Classic_decision_problems_under_imperfect_information}

A {(classic) decision problem} (on~$\NATURE$) is usually given by a nonempty set of actions~$\NEUTRAL$ and by
a \emph{payoff or utility function}~$\NEUTRALpayoff \colon \NEUTRAL\times \NATURE\to \RR$.
Under belief~$\belief \in \BELIEF$, the decision maker chooses 
a decision~$\neutral\in\NEUTRAL$ that maximizes~$\sum_\nature \belief_\nature \NEUTRALpayoff(\neutral,\nature)$.
As far as we are concerned with \MaximalExpectedUtility,
we are going to illustrate how we can move from utility functions to 
value functions defined over beliefs~$\BELIEF$
and, also, to subsets of \ActionPayoff s in~\( \RR^{\NATURE} \). 


\begin{subequations}
The resulting \MaximalExpectedUtility\ gives rise to the so-called \emph{value function}
\begin{equation}
  \ValueFunction{\NEUTRALpayoff} \colon \BELIEF \to \barRR \eqsepv
\ValueFunction{\NEUTRALpayoff}(\belief)
= \sup_{\neutral\in\NEUTRAL}
\sum_\nature \belief_\nature \NEUTRALpayoff(\neutral,\nature)
\eqsepv \forall \belief \in \BELIEF 
  \eqfinp
    \label{eq:value_function_utility_classic}
\end{equation}
In all that follows, we suppose that the value function \( \ValueFunction{\NEUTRALpayoff} \)
takes finite values (on~\( \BELIEF \)).
Inspired by the approach in~\cite{DeLara-Gossner:2020}, we observe
that the expected utility
  \begin{align}
    \text{expected utility }
    &= 
  \sum_\nature \belief_\nature \NEUTRALpayoff(\neutral,\nature)
      = \scalpro{\belief}{\NEUTRALpayoff(\neutral,\cdot)}
      \eqfinv
\label{eq:expected_payoff}
\intertext{ depends on the \ActionPayoff\ 
\( \NEUTRALpayoff(\neutral,\cdot)=
\sequence{\NEUTRALpayoff(\neutral,\nature)}{\nature \in \NATURE} \in \RR^{\NATURE} \),
and that the maximal expected utility}
      \text{maximal expected utility }
    &=
      \sup_{\neutral\in\NEUTRAL}
  \sum_\nature \belief_\nature \NEUTRALpayoff(\neutral,\nature)
      = \sup_{\neutral\in\NEUTRAL}\scalpro{\belief}{ \NEUTRALpayoff(\neutral,\cdot)}
      \label{eq:maximal_expected_payoff}
  \end{align}
depends on the following set of \ActionPayoff s
\begin{equation}
\ACTION_{0}= \{ \NEUTRALpayoff(\neutral,\cdot), 
\neutral\in\NEUTRAL\} \subset \RR^{\NATURE}
\eqfinp
\label{eq:payoff_vectors}
\end{equation}
The {value function} defined in~\eqref{eq:value_function_utility_classic}
can now also be written as
\begin{equation}
  \ValueFunction{\NEUTRALpayoff}(\belief)
= \sup_{\neutral\in\NEUTRAL}
\sum_\nature \belief_\nature \NEUTRALpayoff(\neutral,\nature)
 = \RestrictedSupportFunction{\ACTION_{0}}(\belief)
\eqsepv \forall \belief \in \BELIEF 
  \eqfinv
  \label{eq:value_function=support_function_classic}
\end{equation}
where the equality follows from the definition~\eqref{eq:support_function_restriction}
of the restriction~\( \RestrictedSupportFunction{\ACTION_{0}} \)
of the support function~\( \SupportFunction{\ACTION_{0}} \)
to the closed convex subset~\( \BELIEF \subset \RR^{\NATURE} \).
So, as a first step, we have obtained that the value function~\( \ValueFunction{\NEUTRALpayoff} \)
in~\eqref{eq:value_function_utility_classic} can be replaced by the set~$\ACTION_{0}$ given by~\eqref{eq:payoff_vectors}.
Notice that if two decisions give the same \ActionPayoff\
--- that is, the same element in~$\ACTION_{0}$ ---
they can be conflated into a single decision, and this is precisely what the set~$\ACTION_{0}$ captures
(minimality of decisions for maximal expected utility). 

But we can go further. Indeed, if we take the convex closure (\( \ACTION_{0} \to \convexhull\ACTION_{0} \)),
and then add vectors in \( \RR_{-}^{\NATURE} \)
(\( \convexhull\ACTION_{0} \to \RR_{-}^{\NATURE}+ \convexhull\ACTION_{0} \)),
these two successive operations do not affect the maximal expected utility,
by~\eqref{eq:support_function_closedconvexhull} and because
adding vectors in \( \RR_{-}^{\NATURE} \) with all entries nonpositive
lowers the expected utility. 
Thus, the resulting set 
\begin{equation}
  \ACTION_{\NEUTRALpayoff} ={\RR_{-}^{\NATURE}+ \convexhull\ACTION_{0}}
  = \RR_{-}^{\NATURE}+ \convexhull\bp{ \NEUTRALpayoff(\neutral,\cdot), \neutral\in\NEUTRAL }
  \subset \RR^{\NATURE}
\label{eq:ACTION_classic}
\end{equation}
is a closed\footnote{%
The sum of two convex sets is convex, but the sum of two closed sets
is not necessarily closed. 
However, here the set~\( {\RR_{-}^{\NATURE}+ \convexhull\ACTION_{0}} \) is closed
as a consequence of Item~\ref{it:isomorphism_continuous_subsets} in Proposition~\ref{pr:isomorphism_continuous}, 
because we have supposed that the value function \( \ValueFunction{\NEUTRALpayoff} 
= \RestrictedSupportFunction{\ACTION_{0}} \)
in~\eqref{eq:value_function=support_function_classic}
takes finite values (on~\( \BELIEF \)).}
  convex comprehensive set\footnote{%
    We say that \( \ACTION \subset \RR^{\NATURE} \) is a \emph{comprehensive set}
(also called \emph{lower set})
if \( \RR_{-}^{\NATURE}+{\ACTION} \subset {\ACTION} \)
  (or, equivalently, if \( \RR_{-}^{\NATURE}+{\ACTION} = {\ACTION} \)
  as \( 0\in \RR_{-}^{\NATURE} \)).
  The interpretation of ``lower'' is as follows: let us define the (partial) order on~\( \RR^{\NATURE} \)
  as the one given by the closed convex cone~\( \RR_{+}^{\NATURE} \); 
then,  \( \ACTION \subset \RR^{\NATURE} \) is a lower set if and only if
\( \ACTION \) contains all its minorants.
\label{ft:comprehensive_set}
} 
such that
\begin{equation}
  \ValueFunction{\NEUTRALpayoff}(\belief)
= \sup_{\neutral\in\NEUTRAL}
\sum_\nature \belief_\nature \NEUTRALpayoff(\neutral,\nature)
 = \RestrictedSupportFunction{\ACTION_{0}}(\belief)
= \RestrictedSupportFunction{\ACTION_{\NEUTRALpayoff}}(\belief)
\eqsepv \forall \belief \in \BELIEF 
  \eqfinp 
  \label{eq:value_function=support_function_classic_bis}
\end{equation}
\label{eq:ACTION_classic+value_function=support_function_classic_bis}
\end{subequations}
 Thus, as far as we are concerned with \MaximalExpectedUtility\,
 we can replace the utility function~$\NEUTRALpayoff \colon \NEUTRAL\times \NATURE\to \RR$
 by the set~$\ACTION_{0}$ given by~\eqref{eq:payoff_vectors},
 and also by the closed convex comprehensive set~$\ACTION_{\NEUTRALpayoff}$ given by~\eqref{eq:ACTION_classic}.
 Thus doing, we have conflated a classic decision problem into a closed convex comprehensive set of \ActionPayoff s.
In what follows, we will abstract from the classic definition of a decision problem,
and directly consider such subsets of \ActionPayoff s.

\subsection{Abstract decision problems under imperfect information}
\label{Abstract_decision_problems_and_value_function}

\subsubsubsection{Definition of \RegularActionPayoffSet s}

After the discussion in~\S\ref{Classic_decision_problems_under_imperfect_information},
and especially Equation~\eqref{eq:ACTION_classic}, 
we consider sets of \ActionPayoff s as follows.

\begin{definition}
\label{de:4C_action_set} 
We say that the subset \( \ACTION \subset \RR^{\NATURE} \) is a \RegularActionPayoffSet\ on~$\NATURE$ if
\begin{itemize}
\item
  the set \( \ACTION \) is closed convex and comprehensive, that is,
\( \ACTION \) is a closed convex set and satisfies (see Footnote~\ref{ft:comprehensive_set})
    \( \RR_{-}^{\NATURE}+{\ACTION} \subset {\ACTION} \)
  (or, equivalently, \( \RR_{-}^{\NATURE}+{\ACTION} = {\ACTION} \)),
\item
  the set \( \ACTION \) is continuous\footnote{%
    This notion of continuity for closed convex sets has nothing to do with the smoothness of their frontier
    (see the examples below).
}
    in the sense that \( \RestrictedSupportFunction{\ACTION} \in\RegularConv{\BELIEF} \).
\end{itemize}
We denote by \( \REGULARACTIONPAYOFFSET \subset 2^{\RR^{\NATURE}} \) the set of all
\RegularActionPayoffSet s on~$\NATURE$.
\end{definition}
    For any \RegularActionPayoffSet~\( \ACTION \in \REGULARACTIONPAYOFFSET \),
the function~\( \RestrictedSupportFunction{\ACTION} \) is convex and bounded.
This property will be decisive when properly defining the value of information
in~\S\ref{Information_structures_and_value_of_information},
by avoiding indefinite integrals (mathematical expectations) in
Definition~\ref{de:VoI}.

As an example, the set \( \ACTION = \RR^{\NATURE} \) is a closed convex comprehensive set,
but is not continuous as \( \RestrictedSupportFunction{\RR^{\NATURE}} =+\infty \not\in\RegularConv{\BELIEF} \),
whereas the set \( \ACTION = \RR_{-}^{\NATURE} \) is a closed convex comprehensive set,
which is continuous as \( \RestrictedSupportFunction{\RR_{-}^{\NATURE}} =0\in\RegularConv{\BELIEF} \).

\subsubsubsubsection{Graphical examples of \RegularActionPayoffSet s}

Recall that a subset of~\( \RR^{\NATURE} \) is \emph{polyhedral} if it is a ﬁnite intersection of closed half-spaces
\cite[p.~256]{Bauschke-Combettes:2017}, hence is closed convex.
A \emph{polyhedral \RegularActionPayoffSet} is a continuous polyhedral comprehensive set.
Polyhedral \RegularActionPayoffSet s correspond to finite decisions sets in the classic setting of decision theory.
Two illustrations of \RegularActionPayoffSet s are given in Figure~\ref{fig:polyhedralActionPayoffSet} (polyhedral)
 and in Figure~\ref{fig:nonpolyhedralActionPayoffSet} (non polyhedral).

        \begin{figure}[hbt!]
      \centering
      \begin{tikzpicture}
  \coordinate (A) at (-3,0);      
  \coordinate (B) at (0,0);       \node[above] at (B) {B};
  \coordinate (C) at (4,-1);      \node[above right] at (C) {C};
  \coordinate (D) at (7,-5);      \node[right] at (D) {D};
  \coordinate (E) at (7,-7);     
  \coordinate (F) at (-3,-7);
  
  \fill[pattern=north east lines, pattern color=gray!70]
    (A) -- (B) -- (C) -- (D) -- (E) -- (F);

  \draw[thick]
    (A) -- (B) -- (C) -- (D) -- (E);
  \end{tikzpicture}
  \caption{Example of polyhedral {\RegularActionPayoffSet} in the case of \( \cardinality{\NATURE}=2 \) states of nature
and, in the classic decision problem setting, of a finite set of three actions (corresponding to the points~$B$, $C$ and $D$)}
      \label{fig:polyhedralActionPayoffSet}
    \end{figure}
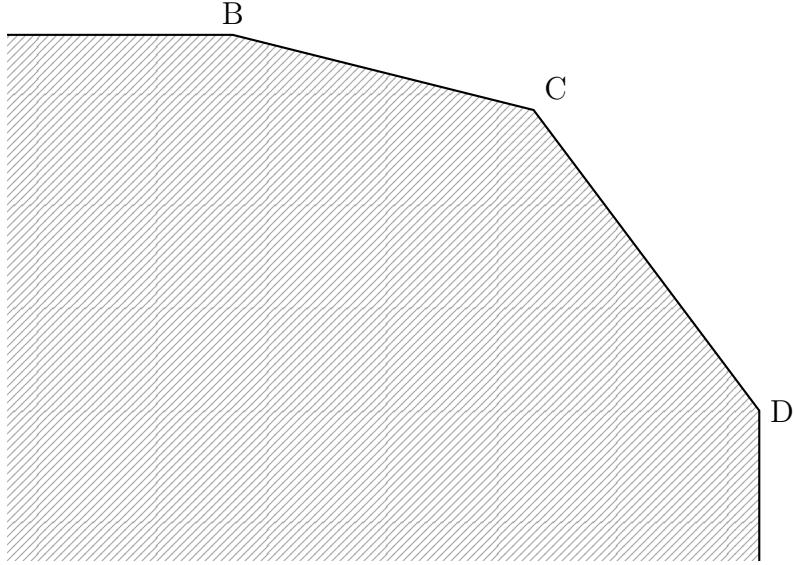

            \begin{figure}[hbt!]
      \centering
\begin{tikzpicture}
  \coordinate (A) at (-4,0);      
  \coordinate (B) at (0,0);       \node[above] at (B) {B};
  \coordinate (C) at (4,-1);      \node[above right] at (C) {C};
  \coordinate (D) at (8,-5);      \node[right] at (D) {D};
  \coordinate (E) at (9,-9);      \node[right] at (E) {E};
  \coordinate (F) at (9,-11);    
  \coordinate (G) at (-4,-11);    

  \fill[pattern=north east lines, pattern color=gray!70]
  (A) -- (B) -- (C)
  .. controls (6,-1.95) .. (D)
  -- (E) -- (F) -- (G);
  
  \draw[thick]
    (A) -- (B) -- (C)
    .. controls (6,-1.95) .. (D)
    -- (E) -- (F);
  \end{tikzpicture}
  \caption{Example of (nonpolyhedral) {\RegularActionPayoffSet} in the case of \( \cardinality{\NATURE}=2 \) states of nature
    and, in the classic decision problem setting, of an infinite set of actions
    which is the union of four actions (corresponding to the points~$B$, $C$, $D$ and $E$) and of a continuum of actions
    (between~$C$ and~$D$)}
      \label{fig:nonpolyhedralActionPayoffSet}
    \end{figure}
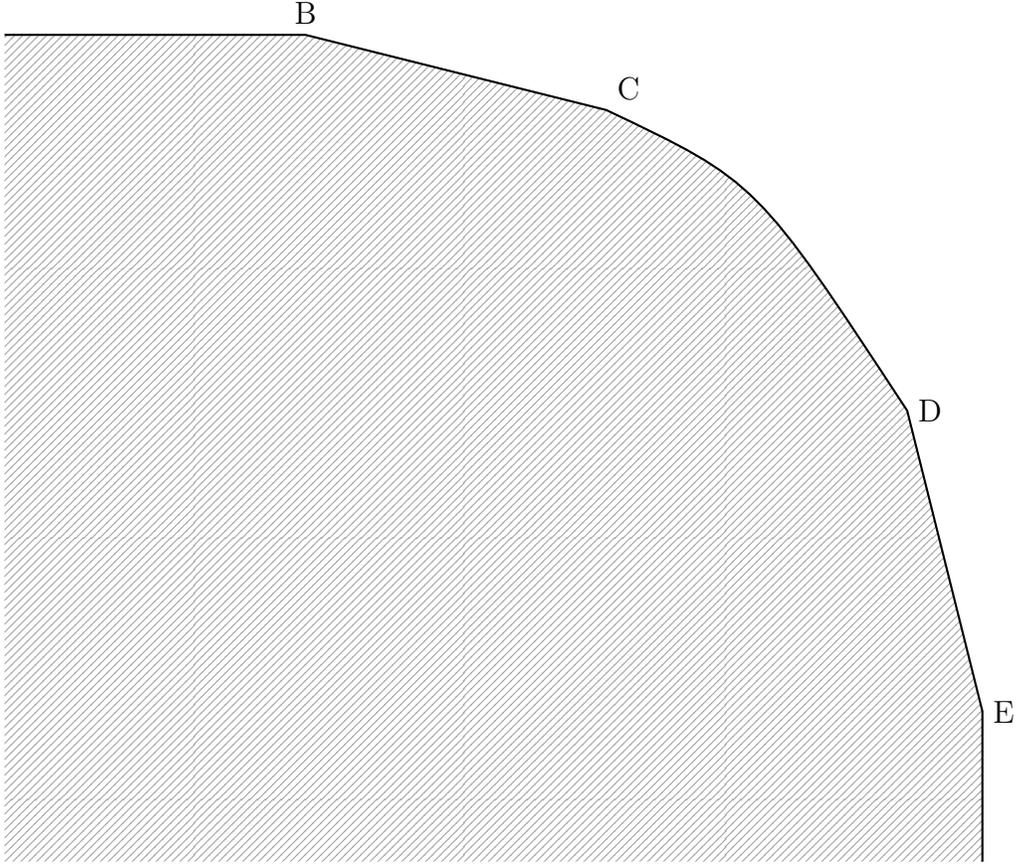

\subsubsubsection{Abstract decision problem}

In~\S\ref{Classic_decision_problems_under_imperfect_information},
we have seen that, with a classic decision problem for which  the value function~\( \ValueFunction{\NEUTRALpayoff} \)
in~\eqref{eq:value_function_utility_classic} takes finite values (on~\( \BELIEF \)), we can associate
a \RegularActionPayoffSet~\( \ACTION_{\NEUTRALpayoff} \in \REGULARACTIONPAYOFFSET \) in~\eqref{eq:ACTION_classic}
such that \( \ValueFunction{\NEUTRALpayoff} = \RestrictedSupportFunction{\ACTION_{\NEUTRALpayoff}} \)
by~\eqref{eq:value_function=support_function_classic_bis}. 
This motivates the following definition.

\begin{definition}
  \label{de:abstract_decision_problem}
  An \emph{abstract decision problem} is defined by a
  \RegularActionPayoffSet~\( \ACTION \in \REGULARACTIONPAYOFFSET \). 
  We will speak of the \emph{(abstract) decision maker}~\( \ACTION \).
\end{definition}

As an elementary example, a (degenerate) decision-maker making a decision taken from a singleton
is represented by \( \hat{\neutral}\in \RR^{\NATURE} \) and the associated
closed convex comprehensive set~\( \ACTION = \RR_{-}^{\NATURE}+\na{\hat{\neutral}} \).
The corresponding {value function} is\footnote{%
Here, for any \( \primal\in \RR^{\NATURE} \),
\( \scalpro{\cdot}{\primal} \) denotes the function
\( \BELIEF \ni\belief \mapsto \scalpro{\belief}{\primal} \)
(but it can denote the function \( \RR^{\NATURE} \ni\signed \mapsto \scalpro{\signed}{\primal} \)
in another context).
\label{ft:proscal_cdot_primal}
}
\( \RestrictedSupportFunction{\ACTION}= \scalpro{\cdot}{\hat{\neutral}} \)
because, by~\eqref{eq:support_function_restriction} and \eqref{eq:support_function_Minkowski}, 
we have that \( \RestrictedSupportFunction{\ACTION} =
\RestrictedSupportFunction{\RR_{-}^{\NATURE} + \na{\hat{\neutral}}}
=\RestrictedSupportFunction{\RR_{-}^{\NATURE}} + \RestrictedSupportFunction{\na{\hat{\neutral}}}
= 0 + \scalpro{\cdot}{\hat{\neutral}}
= \scalpro{\cdot}{\hat{\neutral}} \).
As \( \RestrictedSupportFunction{\ACTION} = \scalpro{\cdot}{\hat{\neutral}} \) is bounded
(see Footnote~\ref{ft:proscal_cdot_primal} recalling that we are considering functions defined
over~\( \BELIEF \)),
we get that \( \ACTION =\RR_{-}^{\NATURE}+\na{\hat{\neutral}} \in \REGULARACTIONPAYOFFSET \)
by Definition~\ref{de:4C_action_set}. 

\subsubsubsection{From abstract to classic decision problem}

With an abstract decision problem defined by a \RegularActionPayoffSet~\( \ACTION \in \REGULARACTIONPAYOFFSET \) on~$\NATURE$,
we can associate the classic decision problem defined by the decision set~\( \ACTION \subset \RR^{\NATURE} \) and the utility function
\begin{equation}
  \ACTION\times\NATURE \ni \np{\action,\nature} \mapsto \action_{\nature} \in\RR
  \eqfinp
\label{eq:from_abstract_to_classic_decision_problem}
\end{equation}
This way of doing is canonical, but there are other ways.
We will encounter in Theorem~\ref{th:more_valuable_information} and
Proposition~\ref{pr:more_valuable_information_implies_in_the_classic_setting}
the case of a DM with
\RegularActionPayoffSet\ of the form~\( \MORE = {\LESS +\MEDIUM} \), where \( \LESS, \MEDIUM \in \REGULARACTIONPAYOFFSET \)
(it will be proved in Item~\ref{it:isomorphism_continuous_subsets} in Proposition~\ref{pr:isomorphism_continuous} that
\( \MORE = {\LESS +\MEDIUM} \in \REGULARACTIONPAYOFFSET \)).
Among the many ways to associate with \( \MORE = {\LESS +\MEDIUM} \) a classic decision problem,
we will single out the one with Cartesian product decision set~\( \LESS\times\MEDIUM \)
and additively separable utility function
\begin{equation}
  \np{\LESS\times\MEDIUM}\times\NATURE \to \RR \eqsepv
  \bp{\np{\less,\medium},\nature} \mapsto \less_{\nature}+\medium_{\nature} \in\RR
  \eqfinp
  \label{eq:Cartesian_product_decision_set}
\end{equation}

\section{More valuable information}
\label{Fusion_makes_information_more_valuable}

In~\S\ref{Information_structures_and_value_of_information}
and~\S\ref{Comparison_of_information_structures_and_relative_value_of_information},
we present classic material in decision theory ---
information structures, value of information, comparison of information structures and
relative value of information --- but with our notations in Sect.~\ref{Abstract_decision-maker}.
Then, in~\S\ref{Characterization_of_more_valuable_information}
we state the main result of the paper:
the equivalence between more valuable information
and \ActionPayoffSet s addition.
To stress the role of flexibility, 
in~\S\ref{Characterization_of_more_valuable_information_in_the_classic_setting}
we recast our main result of the paper in the classic setting as follows:
the equivalence between more valuable information, on the one hand,
and multiplying decisions (multiplicative flexibility) and adding utility, on the other hand.

\subsection{Information structures and value of information}
\label{Information_structures_and_value_of_information}

Information structures have been introduced in the seminal paper~\cite{Marschak-Miyasawa:1968}.
Slightly departing from the literature, especially of what we did
in~\cite{DeLara-Gossner:2020}, we describe information through a
distribution of beliefs but without making reference to a given prior belief.
This approach (which is also the one in \cite{Artstein-Wets:1993})
will be justified in~\S\ref{Comparison_of_information_structures_and_relative_value_of_information}.

\begin{definition}
  \label{de:information_structure}
  Let $(\Omega,{\cal F},\PP)$ be a probability space, and
 let $\borel{\BELIEF}$ denote the Borel $\sigma$-algebra of the simplex~$\BELIEF$. 
An \emph{information structure} is a random variable
\begin{equation}
  \va{\beliefbis} \colon (\Omega,{\cal F},\PP) \to \Bp{\BELIEF,\bpborel{\BELIEF}}
  \eqfinp
\label{eq:information_structure}
\end{equation}
Thus, \( \va{\beliefbis}=\sequence{\va{\beliefbis}_\nature}{\nature\in\NATURE} \) is a random variable
with values in \( \BELIEF = \defset{ \sequence{\belief_\nature}{\nature\in\NATURE}
  \in \RR_{+}^{\NATURE} }{\sum_{\nature\in\NATURE}\belief_\nature = 1 } \).
We denote by \( \InformationStructures\np{\Omega,\BELIEF} \) the set of all information structures.
\end{definition}
We denote by~$\EE$ the mathematical expectation operator with respect to~$\PP$.
  The prior belief associated with an information structure~$\va{\beliefbis}\in\InformationStructures\np{\Omega,\BELIEF} $
is \( \EE\nc{\va{\beliefbis}}
=\bpsequence{\EE\bc{\va{\beliefbis}_\nature}}{\nature\in\NATURE} \in \BELIEF \),
which is i) well-defined as \( \BELIEF \subset \RR^{\NATURE} \) is compact, hence bounded,
ii) indeed an element of~\( \BELIEF \) as this latter set is convex
\cite[Theorem~10.2.6]{Dudley:2002}.
We also interpret \( \EE\nc{\va{\beliefbis}} \) as a constant information structure:
\( \EE\nc{\va{\beliefbis}} \in\InformationStructures\np{\Omega,\BELIEF} \). 

\begin{definition}
\label{de:VoI}
  Consider a decision-maker with \RegularActionPayoffSet~\( \ACTION \in \REGULARACTIONPAYOFFSET \),
as in Definition~\ref{de:4C_action_set},
  and let $\va{\beliefbis}\in\InformationStructures\np{\Omega,\BELIEF} $ be an information structure,
as in Definition~\ref{de:information_structure}.
  The \emph{value of information}~$\VoI_{\ACTION}\np{\va{\beliefbis}}$
 is the nonnegative real number given by: 
\begin{equation}
  \VoI_{\ACTION}\np{\va{\beliefbis}} 
  =  \EE \bc{\RestrictedSupportFunction{\ACTION}\np{\va{\beliefbis}}}
  -\RestrictedSupportFunction{\ACTION}\bp{\EE\nc{\va{\beliefbis}}}
  \in \RR_+
 \eqfinp
\label{eq:VoI}
\end{equation}
\end{definition}
Thus, the value of information~$\VoI_{\ACTION}\np{\va{\beliefbis}}$ is the
difference between the expected utility of
the {decision-maker with \RegularActionPayoffSet~$\ACTION$} 
who either receives information
according to~$\va{\beliefbis} \in\InformationStructures\np{\Omega,\BELIEF} $,
or whose prior belief is~$\EE\nc{\va{\beliefbis}} \in \BELIEF $.
We will show in the forthcoming Proposition~\ref{pr:relative_VoI} that all terms in~\eqref{eq:VoI} are well defined,
and that the result is indeed a nonnegative real number.

In the classic setting of decision problems under imperfect information,
as in~\S\ref{Classic_decision_problems_under_imperfect_information},
where a {decision problem} (on~$\NATURE$) is given by a nonempty decision set~$\NEUTRAL$ and by
a utility function~$\NEUTRALpayoff \colon \NEUTRAL\times \NATURE\to \RR$, we have that
(where $\ACTION_{\NEUTRALpayoff}$ is given by~\eqref{eq:ACTION_classic},
and using~\eqref{eq:value_function=support_function_classic})
\begin{equation}
  \VoI_{\ACTION_{\NEUTRALpayoff}}\np{\va{\beliefbis}} 
  =  \EE \bc{ \sup_{\neutral\in\NEUTRAL}
    \sum_\nature \va{\beliefbis}_\nature \NEUTRALpayoff(\neutral,\nature) }
  - \sup_{\neutral\in\NEUTRAL}
    \sum_\nature \EE \nc{ \va{\beliefbis}_\nature } \NEUTRALpayoff(\neutral,\nature)
 \eqfinp
\label{eq:VoI_classic}
\end{equation}


\subsection{Comparison of information structures and relative value of information}
\label{Comparison_of_information_structures_and_relative_value_of_information}

The value of information, as defined in~\S\ref{Information_structures_and_value_of_information},
displays an (inelegant) asymmetry in~\eqref{eq:VoI}
between the expectation 
\(  \EE \bc{\RestrictedSupportFunction{\ACTION}\np{\va{\beliefbis}}} \) of a function
and the function \( \RestrictedSupportFunction{\ACTION}\bp{\EE\nc{\va{\beliefbis}}} \)
of an expectation.
In fact, the last term may be interpreted like the first one,
but as the expectation 
\( \EE \Bc{\RestrictedSupportFunction{\ACTION}\bp{\EE\nc{\va{\beliefbis}}}} \) of a function
taken at the constant random variable~\( \EE\nc{\va{\beliefbis}} \). 
The forthcoming Definition~\ref{de:relative_VoI} of the {relative value of information}
does not display such asymmetry.

 \begin{definition}
   \label{de:comparison_information_structures}
  Let $(\Omega,{\cal F},\PP)$ be a probability space
  and let \( \va{\beliefbis\upper}, \va{\beliefbis\llower} \in \InformationStructures\np{\Omega,\BELIEF} \)
  be two information structures.
  We say that \( \va{\beliefbis\llower} \) is a \emph{garbling} of~\( \va{\beliefbis\upper} \)
  --- and we denote \( \va{\beliefbis\llower}\moinsfine\va{\beliefbis\upper} \) --- 
  or that \( \va{\beliefbis\upper} \) is a \emph{refinement} of~\( \va{\beliefbis\llower} \)
    --- and we denote \( \va{\beliefbis\upper}\plusfine\va{\beliefbis\llower} \) --- 
  if there exists a $\sigma$-algebra \( {\cal G}\subset{\cal F} \) such that
  \begin{equation}
    \va{\beliefbis\llower} = \espe\conditionaly{\va{\beliefbis\upper}}{{\cal G}} \quad
    \PP\textrm{-almost~surely.}
\label{eq:comparison_information_structures}
\end{equation}
\end{definition}
By taking the trivial $\sigma$-algebra \( {\cal G}=\na{\emptyset,\Omega}\subset{\cal F} \),
the prior belief \( \EE\nc{\va{\beliefbis}} \in \BELIEF \)
associated with an information structure~$\va{\beliefbis}\in\InformationStructures\np{\Omega,\BELIEF} $
is a garbling of~\( \va{\beliefbis} \) --- that is,
\( \EE\nc{\va{\beliefbis}}\moinsfine\va{\beliefbis} \) ---
or, equivalently,
\( \va{\beliefbis} \) is a {refinement} of~\( \EE\nc{\va{\beliefbis}}\) --- that
is, \( \va{\beliefbis} \plusfine \EE\nc{\va{\beliefbis}}\).
More generally, garbling corresponds to less information, whereas refinement corresponds to more information.

\begin{definition}
  \label{de:relative_VoI}
    Consider a decision-maker with \RegularActionPayoffSet~\( \ACTION \in \REGULARACTIONPAYOFFSET \),
as in Definition~\ref{de:4C_action_set},
  and let \( \va{\beliefbis\upper}, \va{\beliefbis\llower} \in \InformationStructures\np{\Omega,\BELIEF} \)
 be two information structures such that 
  \( \va{\beliefbis\upper}\plusfine\va{\beliefbis\llower} \).
  
The \emph{relative value of information}~$\VoI_{\ACTION}\np{\va{\beliefbis\upper}\slash\va{\beliefbis\llower}}$ is the
difference between the expected utility of
the {decision-maker with \RegularActionPayoffSet~$\ACTION$} 
who receives information according 
either to~$\va{\beliefbis\upper}$ or to~$\va{\beliefbis\llower}$.
More precisely, it is the nonnegative real number given by: 
\begin{equation}
  \VoI_{\ACTION}\np{\va{\beliefbis\upper}\slash\va{\beliefbis\llower}}
  =  \EE \bc{\RestrictedSupportFunction{\ACTION}\np{\va{\beliefbis\upper}}}-
  \EE \bc{\RestrictedSupportFunction{\ACTION}\np{\va{\beliefbis\llower}}}
  \in \RR_{+}
 \eqfinp
\label{eq:relative_VoI}
\end{equation}
\end{definition}

\begin{proposition}
\label{pr:relative_VoI}
    Consider a decision-maker with \RegularActionPayoffSet~\( \ACTION \in \REGULARACTIONPAYOFFSET \),
    as in Definition~\ref{de:4C_action_set},
    All terms in~\eqref{eq:relative_VoI} (hence in~\eqref{eq:VoI}) are well defined,
and the {relative value of information}~$\VoI_{\ACTION}\np{\va{\beliefbis\upper}\slash\va{\beliefbis\llower}}$
in~\eqref{eq:relative_VoI} (hence the value of information~$\VoI_{\ACTION}\np{\va{\beliefbis}}$ in~\eqref{eq:VoI})
is a nonnegative real number.
Moreover, we have that 
\begin{subequations}
  \begin{align}
    \VoI_{\ACTION}\np{\va{\beliefbis}}
        &=
          \VoI_{\ACTION}\np{\va{\beliefbis}\slash \EE\nc{\va{\beliefbis}}}
          \eqsepv \forall \va{\beliefbis}\in\InformationStructures\np{\Omega,\BELIEF} 
          \eqfinv
\label{eq:relative_VoI_bis_a}
\\
  \VoI_{\ACTION}\np{\va{\beliefbis\upper}\slash\va{\beliefbis\llower}}          
    &=
      \VoI_{\ACTION}\np{\va{\beliefbis\upper}}
      - \VoI_{\ACTION}\np{\va{\beliefbis\llower}}
      \eqsepv \forall \va{\beliefbis\upper}\plusfine\va{\beliefbis\llower}
      \in \InformationStructures\np{\Omega,\BELIEF} 
      \eqfinp
      \label{eq:relative_VoI_bis_b}
  \end{align}
\label{eq:relative_VoI_bis}
\end{subequations}

\end{proposition}

\begin{proof}
We show that the {relative value of information}~$\VoI_{\ACTION}\np{\va{\beliefbis\upper}\slash\va{\beliefbis\llower}}$
in~\eqref{eq:relative_VoI} is well defined, by showing that each term in the difference is a well-defined real number
(neither \( -\infty \) nor \( +\infty \)).
For this purpose, let $\va{\beliefbis}\in\InformationStructures\np{\Omega,\BELIEF} $ be an information structure.
We are going to show that
\( \EE \bc{\RestrictedSupportFunction{\ACTION}\np{\va{\beliefbis}}} \)
is a well-defined real number.
On the one hand, the function~\( \RestrictedSupportFunction{\ACTION} \) is continuous bounded convex
as \( \ACTION \in \REGULARACTIONPAYOFFSET \) (see Definition~\ref{de:4C_action_set}).
Thus, \( \RestrictedSupportFunction{\ACTION}\np{\va{\beliefbis}} \) is the
composition of the continuous (hence, measurable) function~$\RestrictedSupportFunction{\ACTION}$
with the random variable~$\va{\beliefbis}$, hence is a random variable.
On the other hand, the random variable~\( \RestrictedSupportFunction{\ACTION}\np{\va{\beliefbis}}
\) takes value in the bounded set~$\RestrictedSupportFunction{\ACTION}(\BELIEF)$.
As a consequence, \( \EE \bc{\RestrictedSupportFunction{\ACTION}\np{\va{\beliefbis}}} \)
is a well-defined real number, and so are 
the {relative value of information}~$\VoI_{\ACTION}\np{\va{\beliefbis\upper}\slash\va{\beliefbis\llower}}$
in~\eqref{eq:relative_VoI} and the value of information~$\VoI_{\ACTION}\np{\va{\beliefbis}}$ in~\eqref{eq:VoI}.

That the {relative value of information}~$\VoI_{\ACTION}\np{\va{\beliefbis\upper}\slash\va{\beliefbis\llower}}$
(hence the value of information~$\VoI_{\ACTION}\np{\va{\beliefbis}}$ in~\eqref{eq:VoI})
in~\eqref{eq:relative_VoI} is nonnegative follows from Jensen inequality,
but with conditional expectations, applied to the convex function 
\( \RestrictedSupportFunction{\ACTION} \)
(we refer the reader to \cite[p.~33]{Doob-1953},
\cite[Chap.~II,41.4]{Dellacherie-Meyer:1975}, 
\cite{Artstein-Wets:1993}, \cite{Artstein1999:gains}, among others).

Equation~\eqref{eq:relative_VoI_bis_a} is Equation~\eqref{eq:relative_VoI} with
\( \va{\beliefbis\upper}= \va{\beliefbis} \) and \( \va{\beliefbis\llower}= \EE\nc{\va{\beliefbis}} \).
Equation~\eqref{eq:relative_VoI_bis_b} follows from Equations~\eqref{eq:relative_VoI}
and~\eqref{eq:VoI}. Indeed, from~\eqref{eq:comparison_information_structures}, we get that
\( \va{\beliefbis\upper}\plusfine\va{\beliefbis\llower} \implies
\EE\bc{\va{\beliefbis\upper}} = \EE\bc{\va{\beliefbis\llower}} \).
\end{proof}

\subsection{Characterization of more valuable information}
\label{Characterization_of_more_valuable_information}

First, we define what is ``more valuable information'' by comparing two decision-makers.

\begin{definition}
  Condider two decision-makers with \RegularActionPayoffSet s \(\MORE \in \REGULARACTIONPAYOFFSET \)
  and \(\LESS \in \REGULARACTIONPAYOFFSET \). 
We say that the decision-maker~$\MORE$ \emph{values information (weakly) more} than the
decision-maker~$\LESS$ if
\begin{subequations}
  \begin{equation}
  \VoI_{\MORE}\np{\va{\beliefbis}} \geq \VoI_{\LESS}\np{\va{\beliefbis}} 
\eqsepv \forall \va{\beliefbis}\in\InformationStructures\np{\Omega,\BELIEF}
 \eqfinp
\label{eq:VoIMore}
\end{equation}
We say that the decision-maker~$\MORE$ \emph{values information strongly more} than the
    decision-maker~$\LESS$ if
\begin{equation}
  \forall \va{\beliefbis\upper}, \va{\beliefbis\llower} \in \InformationStructures\np{\Omega,\BELIEF}
  \eqsepv \Bp{
  \va{\beliefbis\upper}\plusfine\va{\beliefbis\llower} \implies
  \VoI_{\MORE}\np{\va{\beliefbis\upper}\slash\va{\beliefbis\llower}}
  \geq
    \VoI_{\LESS}\np{\va{\beliefbis\upper}\slash\va{\beliefbis\llower}} }
 \eqfinp
\label{eq:VoIMore_strong}
\end{equation}  
\end{subequations}
\label{de:VoIMore}
\end{definition}

By taking \( \va{\beliefbis\llower} = \EE\nc{\va{\beliefbis\upper}} \)
in~\eqref{eq:VoIMore_strong}
(see Equation~\eqref{eq:relative_VoI_bis_a}), we obtain that
the definition of ``valuing information strongly more'' implies
the definition of ``valuing information more'' --- that we call
\emph{valuing information weakly more}, when needed, to ease the comparison.
\medskip 

Second, we characterize ``more valuable information''.
Our main result is the following.
\begin{theorem}
\label{th:more_valuable_information}
Consider two decision-makers with \RegularActionPayoffSet s
  \(\MORE \in \REGULARACTIONPAYOFFSET \) and \(\LESS \in \REGULARACTIONPAYOFFSET \). 
  Then, the following statements are equivalent.
  \begin{enumerate}
  \item
    \label{it:values_more_information_strongly}
    The decision-maker~$\MORE$ values information strongly more than the
    decision-maker~$\LESS$.
      \item
    \label{it:values_more_information}
    The decision-maker~$\MORE$ values information (weakly) more than the
    decision-maker~$\LESS$.
  \item
       \label{it:difference_is_convex}
       The function \( \RestrictedSupportFunction{\MORE}-\RestrictedSupportFunction{\LESS} \) is convex
       (on~$\BELIEF$).
     \item
             \label{it:fusion}
             The decision-maker~$\MORE$ is obtained by adding \ActionPayoff s to the decision-maker~$\LESS$,
             more precisely there exists a~\RegularActionPayoffSet~\( \MEDIUM \in \REGULARACTIONPAYOFFSET \) such that
\begin{subequations}
    \begin{equation}
      \MORE = {\LESS +\MEDIUM} 
      \eqfinp 
    \end{equation}
  \item
    \label{it:fusion_star-difference}
    The {star-difference} in~\eqref{eq:stardifference} satisfies
    \(   \MORE\stardifference\LESS \in \REGULARACTIONPAYOFFSET \)
    (hence \( \MORE\stardifference\LESS \neq \emptyset \)), and 
\begin{equation}
  \MORE={\LESS +\np{\MORE\stardifference\LESS}}
  \eqfinp 
\end{equation}
 \end{subequations}
  \end{enumerate}
\end{theorem}
The equivalence between Item~\ref{it:values_more_information}
and Item~\ref{it:difference_is_convex} has already been established, in one form or another,
in \cite[Proposition~0]{Jones-Ostroy:1984},
\cite[Theorem~3.1]{whitmeyer2024makinginformationvaluable} (see also
\cite[Footnote~8]{whitmeyer2024makinginformationvaluable} which points to another source).
That  Item~\ref{it:values_more_information} 
implies
Item~\ref{it:values_more_information_strongly} is new.
The equivalence between Item~\ref{it:values_more_information}
and 
Item~\ref{it:fusion} (or Item~\ref{it:fusion_star-difference})
is new:
this is the main result of the paper.

The difficulty is to show that Item~\ref{it:values_more_information} implies Item~\ref{it:fusion}.
Indeed, the reverse implication (Item~\ref{it:fusion} implies Item~\ref{it:values_more_information})
is rather easy (although we have not found trace of this observation in the literature).
We illustrate this latter property --- and also recast the main result of Theorem~\ref{th:more_valuable_information}
--- in the classic setting in the next~\S\ref{Characterization_of_more_valuable_information_in_the_classic_setting}.

\subsection{Characterization of more valuable information in the classic setting}
\label{Characterization_of_more_valuable_information_in_the_classic_setting}

Here, we will we recast the main result of Theorem~\ref{th:more_valuable_information} in the classic setting as follows:
the equivalence between more valuable information, on the one hand,
and multiplying decisions (multiplicative flexibility) and adding utility, on the other hand.


The following Proposition~\ref{pr:more_valuable_information_implies_in_the_classic_setting} is easy to prove
(but, as said above, we have not found it in the literature).

\begin{proposition}
\label{pr:more_valuable_information_implies_in_the_classic_setting}
Consider a decision-maker~$\LESS$ given by a decision set~$\NEUTRAL_{\LESS}$ and by 
a utility function~$\NEUTRALpayoff_{\LESS} \colon \NEUTRAL_{\LESS}\times \NATURE\to \RR$.
Given another decision set~$\NEUTRAL_{\MEDIUM}$ and another
utility function~$\NEUTRALpayoff_{\MEDIUM} \colon \NEUTRAL_{\MEDIUM}\times \NATURE\to \RR$,
we form the new decision-maker~$\MORE$ with 
Cartesian product decision set~$\NEUTRAL_{\MORE}=\NEUTRAL_{\LESS}\times\NEUTRAL_{\MEDIUM}$ and with
additively separable 
utility function~$\NEUTRALpayoff_{\MORE}=\NEUTRALpayoff_{\LESS}+\NEUTRALpayoff_{\MEDIUM} 
\colon \np{\NEUTRAL_{\LESS}\times\NEUTRAL_{\MEDIUM}}\times \NATURE\to \RR$
defined by
\begin{equation}
 \NEUTRALpayoff_{\MORE}\bp{\np{\less,\medium},\nature}
  =\NEUTRALpayoff_{\LESS}(\less,\nature)+\NEUTRALpayoff_{\MEDIUM}(\medium,\nature) \eqsepv
\forall \np{\less,\medium}\in\NEUTRAL_{\LESS}\times\NEUTRAL_{\MEDIUM}
  \eqsepv \forall \nature \in \NATURE 
  \eqfinp
\end{equation}
Suppose that the value functions \( \ValueFunction{\NEUTRALpayoff_{\LESS}} \)
and \( \ValueFunction{\NEUTRALpayoff_{\MEDIUM}} \) in~\eqref{eq:value_function_utility_classic}
take finite values (on~\( \BELIEF \)).
Then the decision-maker~$\MORE$ values information more than the decision-maker~$\LESS$.
\end{proposition}

\begin{proof}
By~\eqref{eq:VoI_classic}, we get that (all the suprema in the proof below are finite by assumption)
\begin{subequations}
\begin{align*}
  \VoI_{\MORE}\np{\va{\beliefbis}} 
  &=
    \EE \bc{ \sup_{\np{\less,\medium}\in\NEUTRAL_{\LESS}\times\NEUTRAL_{\MEDIUM}}
    \sum_\nature \va{\beliefbis}_\nature \bp{\NEUTRALpayoff_{\LESS}(\less,\nature)+\NEUTRALpayoff_{\MEDIUM}(\medium,\nature)} }
  \\
  &-
    \sup_{\np{\less,\medium}\in\NEUTRAL_{\LESS}\times\NEUTRAL_{\MEDIUM}}
  \sum_\nature \EE \nc{ \va{\beliefbis}_\nature }
\bp{\NEUTRALpayoff_{\LESS}(\less,\nature)+\NEUTRALpayoff_{\MEDIUM}(\medium,\nature)} 
\nonumber  \\
  &=
    \EE \bc{ \sup_{\less\in\NEUTRAL_{\LESS}}
    \sum_\nature \va{\beliefbis}_\nature \NEUTRALpayoff_{\LESS}(\less,\nature) 
+ \sup_{\medium\in\NEUTRAL_{\MEDIUM}}
\sum_\nature \va{\beliefbis}_\nature \NEUTRALpayoff_{\MEDIUM}(\medium,\nature) }
\nonumber  \\
  &-
    \EE \bc{ \sup_{\less\in\NEUTRAL_{\LESS}}
    \sum_\nature \EE \nc{ \va{\beliefbis}_\nature } \NEUTRALpayoff_{\LESS}(\less,\nature) 
+ \sup_{\medium\in\NEUTRAL_{\MEDIUM}}
\sum_\nature \EE \nc{ \va{\beliefbis}_\nature } \NEUTRALpayoff_{\MEDIUM}(\medium,\nature) }
\nonumber  \\
  &=
    \VoI_{\LESS}\np{\va{\beliefbis}} + \underbrace{\VoI_{\MEDIUM}\np{\va{\beliefbis}}}_{\geq 0}
    \geq \VoI_{\LESS}\np{\va{\beliefbis}}
    \eqfinp 
\end{align*}
\end{subequations}
This ends the proof.
\end{proof}

The new decision-maker~$\MORE$ has a decision set which is obtained from 
the decision set~$\NEUTRAL_{\LESS}$ of the decision-maker~$\LESS$ by the 
Cartesian product~$\NEUTRAL_{\MORE}=\NEUTRAL_{\LESS}\times\NEUTRAL_{\MEDIUM}$.
We coin this property of multiplying decisions as \emph{multiplicative flexibility}.
This is indeed flexibility as the decision set~$\NEUTRAL_{\LESS}$ has been embedded in
$\NEUTRAL_{\LESS}\times\NEUTRAL_{\MEDIUM}$: by fixing any arbitrary decision in~$\NEUTRAL_{\MEDIUM}$,
the decision-maker~$\LESS$ maintains his original decisions in~$\NEUTRAL_{\LESS}$.
This is \emph{multiplicative} flexibility because the embedding is in the 
Cartesian product~$\NEUTRAL_{\LESS}\times\NEUTRAL_{\MEDIUM}$, hence decisions are multiplied.

Regarding the utility function of the new decision-maker~$\MORE$, it is obtained
from the utility function of the decision-maker~$\LESS$ by utility addition,
giving the additively separable utility function~$\NEUTRALpayoff_{\MORE}=\NEUTRALpayoff_{\LESS}+\NEUTRALpayoff_{\MEDIUM}$. 
\medskip

  
Thus, Proposition~\ref{pr:more_valuable_information_implies_in_the_classic_setting} shows that
multiplying decisions and adding utilities are sufficient to make information valuable.
The reverse statement takes the following form.

\begin{proposition}
\label{pr:more_valuable_information_isimplied_in_the_classic_setting}
Consider a decision-maker~$\LESS$ given by a decision set~$\NEUTRAL_{\LESS}$ and by 
a utility function~$\NEUTRALpayoff_{\LESS} \colon \NEUTRAL_{\LESS}\times \NATURE\to \RR$,
and another decision-maker~$\MORE$ with decision set~$\NEUTRAL_{\MORE}$ and
utility function~$\NEUTRALpayoff_{\MORE}$.
Suppose that the value functions \( \ValueFunction{\NEUTRALpayoff_{\LESS}} \)
and \( \ValueFunction{\NEUTRALpayoff_{\MORE}} \) in~\eqref{eq:value_function_utility_classic}
take finite values (on~\( \BELIEF \)).

Then, if the decision-maker~$\MORE$ values information more than the decision-maker~$\LESS$,
there exists \( \MEDIUM \in \REGULARACTIONPAYOFFSET \)
such that the decision makers~$\LESS$ and~$\MORE$ can be equivalently replaced by 
\begin{itemize}
\item
the abstract decision makers
\( \ACTION_{\NEUTRALpayoff_{\LESS}} \in \REGULARACTIONPAYOFFSET \) in~\eqref{eq:ACTION_classic}
and \( \ACTION_{\NEUTRALpayoff_{\LESS}}+\MEDIUM \in \REGULARACTIONPAYOFFSET \), respectively,
\item
  the classic decision makers with 
  decision set~\( \ACTION_{\NEUTRALpayoff_{\LESS}}\)
  and Cartesian product decision set~\( \ACTION_{\NEUTRALpayoff_{\LESS}}\times \MEDIUM \), respectively,
and utility functions
\( \ACTION_{\NEUTRALpayoff_{\LESS}}\times\NATURE \ni \np{\action,\nature} \mapsto \action_{\nature} \in\RR \)
and additively separable utility function
\( \np{\ACTION_{\NEUTRALpayoff_{\LESS}}\times\MEDIUM} \times\NATURE
\ni   \bp{\np{\less,\medium},\nature} \mapsto \less_{\nature}+\medium_{\nature} \in\RR \)
as in~\eqref{eq:Cartesian_product_decision_set}, respectively.  
\end{itemize}
\end{proposition}
By ``equivalently replaced'', we mean that the value functions are unchanged:
\( \ValueFunction{\NEUTRALpayoff_{\LESS}} = \RestrictedSupportFunction{\ACTION_{\NEUTRALpayoff_{\LESS}}}\)
and \( \ValueFunction{\NEUTRALpayoff_{\MORE}} = \RestrictedSupportFunction{\ACTION_{\NEUTRALpayoff_{\LESS}}\times \MEDIUM} \).
Thus, in a sense, Proposition~\ref{pr:more_valuable_information_isimplied_in_the_classic_setting}
shows that multiplying decisions and adding utilities are necessary to make information valuable.

We do not give a proof as Proposition~\ref{pr:more_valuable_information_isimplied_in_the_classic_setting}
results from a recasting, in the classic setting, of the statement that
Item~\ref{it:values_more_information} implies Item~\ref{it:fusion} in Theorem~\ref{th:more_valuable_information},
where the moves between classic and abstract setting have been detailed in~\S\ref{Classic_decision_problems_under_imperfect_information} and in~\S\ref{Abstract_decision_problems_and_value_function}.

\section{Flexibility and more valuable information}
\label{Flexibility_and_more_valuable_information}

In~\S\ref{Characterization_of_more_valuable_information_in_the_classic_setting},
we have recast the main result of Theorem~\ref{th:more_valuable_information}
by stressing the role of flexibility:
the equivalence between more valuable information, on the one hand,
and multiplying decisions (multiplicative flexibility) and adding utility, on the other hand.
We develop this approach in this Sect.~\ref{Flexibility_and_more_valuable_information}.

In~\S\ref{The_dioids_of_ExpectedUtilityMaximizer_s},
we present dioids of \ExpectedUtilityMaximizer s, provide economic interpretations,
define union and fusion of decision-makers, and finally
deduce two kinds of flexibility, by union and by fusion.
In~\S\ref{More_on_flexibility_and_more_valuable_information},
we discuss when flexibility by union can lead to more valuable information.
In~\S\ref{Relation_with_the_literature}, 
we relate our results to those in three recent papers
\cite{whitmeyer2024makinginformationvaluable},
\cite{Denti:2022} and \cite{Yoder:2022}.

\subsection{Dioids of \ExpectedUtilityMaximizer s}
\label{The_dioids_of_ExpectedUtilityMaximizer_s}

In~\S\ref{Homomorphic_dioids_of_ExpectedUtilityMaximizer_s},
we provide background on dioids and then 
we equip the sets \( \RegularConv{\BELIEF} \) of continuous convex functions over beliefs
and the set \( \REGULARACTIONPAYOFFSET \) of \RegularActionPayoffSet s
with two operations leading to homomorphic dioid structures.
In~\S\ref{Economic_interpretation_of_the_dioids_of_ExpectedUtilityMaximizer_s},
we provide economic interpretations of these two dioids.
In~\S\ref{Union_and_fusion_of_decision-makers}, we name and define the two operations of the dioids
as union and fusion of decision-makers.
In~\S\ref{Two_kinds_of_flexibility:_union_and_fusion},
we deduce two kinds of flexibility: by union and by fusion.

\subsubsection{Homomorphic dioids of \ExpectedUtilityMaximizer s}
\label{Homomorphic_dioids_of_ExpectedUtilityMaximizer_s}

\subsubsubsection{Background on dioids}

The definition of a \emph{dioid} is given in \cite[Definition~4.1]{Baccelli-Cohen-Olsder-Quadrat:1992}. 
It is a set endowed with two operations~$\oplus$ (``sum'') and~$\otimes$ (``product''),
making it an algebra which is commutative (for sum~$\oplus$),
associative (for both operations~$\oplus$ and~$\otimes$),
and distributive (of $\otimes$ \wrt\ $\oplus$, where \wrt\ stands for ``with respect to''), 
with the additional properties that
there is a neutral (or zero) element~\( \epsilon \) for~$\oplus$ which is an absorbing element for~$\otimes$,
there is a unit element~\( e \) for~$\otimes$ and, characteristically,
the operation~$\oplus$ is idempotent
(\( a\oplus a = a \), for all element~$a$ of the dioid).
A dioid is commutative if the product~$\otimes$ is commutative.
A homomorphism between two dioids is a mapping that sends sums to sums, products to products,
neutral element to neutral element, and unit element to unit element
\cite[Definition~4.20]{Baccelli-Cohen-Olsder-Quadrat:1992}. 
%

\subsubsubsection{Inverse homomorphisms between dioids}

We equip the sets \( \RegularConv{\BELIEF} \) of continuous convex functions over beliefs
and the set \( \REGULARACTIONPAYOFFSET \) of \RegularActionPayoffSet s
with two operations leading to homomorphic dioid structures.
%

\begin{proposition}
  \label{pr:isomorphism_continuous}
  \quad
  \begin{enumerate}
  \item
      \label{it:isomorphism_continuous_functions}
  Endowed with the two operations\footnote{%
In~\eqref{eq:two_operations_functions_otimes_continuous}, we have as special cases
    \( (-\infty) \otimes (-\infty) = (-\infty) + (-\infty) =-\infty \)
    and, if \( \fonctionun \in \RegularConv{\BELIEF} \),     
    \( \fonctionun \otimes (-\infty) = \fonctionun + (-\infty) =-\infty \)
    because \( \fonctionun \in \RegularConv{\BELIEF} \) is bounded, hence never takes the value~$+\infty$.}
\begin{subequations}
  \begin{align}
\fonctionun \oplus \fonctiondeux 
    &=
  \sup\na{\fonctionun,\fonctiondeux}
  \eqsepv \forall \fonctionun, \fonctiondeux \in \RegularConv{\BELIEF} \cup \na{-\infty}
      \eqfinv
  \label{eq:two_operations_functions_oplus_continuous}
    \\
    \fonctionun \otimes \fonctiondeux
    &=
      \fonctionun + \fonctiondeux
  \eqsepv \forall \fonctionun, \fonctiondeux \in \RegularConv{\BELIEF} \cup \na{-\infty}
      \eqfinv
        \label{eq:two_operations_functions_otimes_continuous}
  \end{align}
  \label{eq:two_operations_functions_continuous}
\end{subequations}
\( \bp{ \RegularConv{\BELIEF} \cup \na{-\infty}, \oplus, \otimes } \) is a dioid,
with neutral (or zero) element~\( -\infty \) and unit element~\( 0 \).  

\item
       \label{it:isomorphism_continuous_subsets}
Endowed with the two operations
  \begin{subequations}
    \begin{align}
 \ACTION \oplus \ACTIONbis 
      &=
\closedconvexhull\np{\ACTION\cup\ACTIONbis}
  \eqsepv \forall \ACTION, \ACTIONbis \in \REGULARACTIONPAYOFFSET \cup \na{\emptyset}
        \eqfinv
        \label{eq:two_operations_subsets_oplus_continuous}
      \\
      \ACTION \otimes \ACTIONbis
      &=
\ACTION + \ACTIONbis 
  \eqsepv \forall \ACTION, \ACTIONbis \in \REGULARACTIONPAYOFFSET \cup \na{\emptyset}
  \eqfinv
\label{eq:two_operations_subsets_otimes_continuous}     
    \end{align}
\label{eq:two_operations_subsets_continuous}
\end{subequations}
\( \np{ \REGULARACTIONPAYOFFSET \cup \na{\emptyset}, \oplus, \otimes } \) is a dioid,
with neutral (or zero) element~\( \emptyset \) and unit element~\( \RR_{-}^{\NATURE} \).

\item
         \label{it:isomorphism_continuous_isomorphism}
         Denoting\footnotemark
\footnotetext{Equation~\eqref{eq:isomorphism_b_continuous_def} has the equivalent formulation
 \( \InverseRestrictedSupportFunction{\fonctionun}=
 \defset{\primal\in \RR^{\NATURE}}{ \LFM{\fonctionun}\np{\primal} \leq 0 } \),
 where \( \LFM{\fonctionun} \) denotes the Fenchel conjugate of the function~\( \fonctionun \) (extended with the value \( +\infty \)
 outside of~\( \BELIEF \)), as in~\eqref{eq:LFM_beliefs}.
\label{ft:isomorphism_b_continuous_def}}
  \begin{equation}
    \InverseRestrictedSupportFunction{\fonctionun}=
      \defset{\primal\in \RR^{\NATURE}}{ \scalpro{\belief}{\primal}
        \leq \fonctionun\np{\belief} \eqsepv \forall \belief\in\BELIEF }
      \eqfinv 
      \label{eq:isomorphism_b_continuous_def}
  \end{equation}
the following mappings, from sets to functions
 \begin{subequations}
  \begin{align}
    \FromACTIONPAYOFFSETtoConvBELIEF \colon 
    \REGULARACTIONPAYOFFSET \cup \na{\emptyset} \to \RegularConv{\BELIEF} \cup \na{-\infty} \eqsepv
    &
      \ACTION \mapsto \RestrictedSupportFunction{\ACTION}
        \label{eq:isomorphism_a_continuous}
    \intertext{and from functions to sets}
        \FromConvBELIEFtoACTIONPAYOFFSET \colon 
    \RegularConv{\BELIEF} \cup \na{-\infty} \to \REGULARACTIONPAYOFFSET \cup \na{\emptyset} \eqsepv
    &
      \fonctionun \mapsto \InverseRestrictedSupportFunction{\fonctionun}
      \label{eq:isomorphism_b_continuous}
  \end{align}
  \label{eq:isomorphism_continuous}
\end{subequations}
are homormorphisms inverse one to the other by 
\begin{subequations}
  \begin{align}
    \np{\FromConvBELIEFtoACTIONPAYOFFSET \circ \FromACTIONPAYOFFSETtoConvBELIEF}\np{\ACTION}
    =\InverseRestrictedSupportFunction{\RestrictedSupportFunction{\ACTION}}
    &=
      \ACTION
    \eqsepv \forall \ACTION\in\REGULARACTIONPAYOFFSET \cup \na{\emptyset} 
    \eqfinv
   \label{eq:one-to-one_isomorphisms_a_continuous} 
    \\
\np{\FromACTIONPAYOFFSETtoConvBELIEF \circ \FromConvBELIEFtoACTIONPAYOFFSET}\np{\fonctionun}
=    \RestrictedSupportFunction{%
    \InverseRestrictedSupportFunction{\fonctionun}}
    &=
      \fonctionun \eqsepv
      \forall \fonctionun\in\RegularConv{\BELIEF} \cup \na{-\infty}
      \eqfinp 
   \label{eq:one-to-one_isomorphisms_b_continuous}
  \end{align}
   \label{eq:one-to-one_isomorphisms_continuous}
 \end{subequations}
  \end{enumerate}
\end{proposition}
We call \( \bp{ \REGULARACTIONPAYOFFSET \cup \na{\emptyset}, \oplus, \otimes } \) 
the \emph{dioid of \RegularActionPayoffSet s (on~$\NATURE$)}.
By isomorphism, we coin \emph{\RegularValueFunction} any continuous convex function over beliefs,
that is, any function in~\( \RegularConv{\BELIEF} \).
Then, we call \( \bp{ \RegularConv{\BELIEF} \cup \na{-\infty}, \oplus, \otimes } \) 
 the \emph{dioid of \RegularValueFunction s (over beliefs)}.

\subsubsubsection{Illustration of the isomomorphism between dioids}
As an illustration of Proposition~\ref{pr:isomorphism_continuous},
with the negentropy \( -\Entropy \in \RegularConv{\BELIEF} \) 
in~\eqref{eq:negentropy}, we get from~\eqref{eq:isomorphism_a_continuous} that\footnote{%
  We use that the Fenchel conjugate~\eqref{eq:LFM_beliefs} of the negentropy \( -\Entropy \) (extended with the value \( +\infty \)
  outside of~\( \BELIEF \)) is the so-called \emph{LogSumExp} function defined by
  \( \textrm{LogSumExp}\np{\primal} = \log\bp{ \sum_{\nature\in\NATURE} e^{\primal_\nature} } \),
  for any \( \primal\in \RR^{\NATURE} \).}
\begin{subequations}
  \begin{align}
    \FromConvBELIEFtoACTIONPAYOFFSET\np{-\Entropy}
    &=
  \defset{\primal\in \RR^{\NATURE}}{ \sum_{\nature\in\NATURE} e^{\primal_\nature} \leq 1 }
      \eqfinv
    \intertext{hence, by~\eqref{eq:one-to-one_isomorphisms_b_continuous}, that} 
    -\Entropy\np{\belief}
    &=
      \sup \defset{\scalpro{\belief}{\primal}}{ \primal\in \RR^{\NATURE} \text{ such that }
      \sum_{\nature\in\NATURE} e^{\primal_\nature} \leq 1 }
  \eqsepv \forall \belief\in\BELIEF
      \eqfinp
  \label{eq:negentropy=value_function}       
  \end{align}
\end{subequations}
Equation~\eqref{eq:negentropy=value_function} leads to two interpretations of
the negentropy~\( -\Entropy\) as the value function, in the classic decision setting, 
\begin{itemize}
\item
  either of a DM with decision set~\( \FromConvBELIEFtoACTIONPAYOFFSET\np{-\Entropy} \)
  and utility function \(  \defset{\primal\in \RR^{\NATURE}}{ \sum_{\nature\in\NATURE} e^{\primal_\nature} \leq 1 }
  \times\NATURE \ni \np{\primal,\nature} \mapsto \primal_{\nature} \in \RR \),
  as in~\eqref{eq:from_abstract_to_classic_decision_problem},
\item 
 or of a DM who selects a vector \( \sequence{\neutral_{\nature'}}{\nature' \in \NATURE} \) 
(after the change of variables \( \neutral=e^{\primal}\))
--- 
one position \( \neutral_{\nature'} \in \RR_{++}^{\NATURE} \) for each state of nature \( \nature' \in \NATURE \)
(as inspired by \cite{CabralesGossnerSerranoEntropyAER2013}) --- 
and then maximizes the logarithmic utility \( \NEUTRALpayoff \colon \RR_{++}^{\NATURE}\times\NATURE \to\RR \) given by
\( \NEUTRALpayoff\np{\sequence{\neutral_{\nature'}}{\nature' \in \NATURE},\nature} =
\log\neutral_{\nature} \), under the (budget) constraint
\( \sum_{\nature'\in\NATURE} \neutral_{\nature'} \leq 1 \). 
\end{itemize}

\subsubsection{Economic interpretation of the dioids of \ExpectedUtilityMaximizer s}
\label{Economic_interpretation_of_the_dioids_of_ExpectedUtilityMaximizer_s}

We illustrate the economic interpretation of the dioids in~\S\ref{Homomorphic_dioids_of_ExpectedUtilityMaximizer_s}
with the decision problem outlined in the introduction Sect.~\ref{Introduction}:
selecting meals in the menu of an unknown restaurant.

The dioid \( \bp{ \RegularConv{\BELIEF} \cup \na{-\infty}, \oplus, \otimes } \) 
  of \RegularValueFunction s over beliefs
is an abstract representation of the \emph{output}, in term of expected utility,
of a classic decision problem.
\begin{itemize}
\item 
The operation~\( \oplus \) corresponds to taking the best (supremum) of two expected utilities:
to choose between starters \emph{or} desserts, one compares the maximal, over starters, utilities
with the maximal, over desserts, utilities; then, one takes the supremum
 (as one will take the best between starters and desserts, but not both).
\item 
The operation~\( \otimes \) correspond to adding two expected utilities:
to choose \emph{within} starters \emph{and within} desserts, one computes the maximal, over starters, utilities
and the maximal, over desserts, utilities; then, one makes the sum (as one will take both starters and desserts).
\end{itemize}

By contrast, the dioid \( \bp{ \REGULARACTIONPAYOFFSET \cup \na{\emptyset}, \oplus, \otimes } \) 
of \RegularActionPayoffSet s on~$\NATURE$
is an abstract representation of the \emph{input} of a classic decision problem,
as an element of a \RegularActionPayoffSet\ can be identified with the utility~act yielded by a decision
and, ultimately, with a decision (indeed, if two decisions give the same \ActionPayoff, they
can be conflated into a single decision).
Thus, we identify a \RegularActionPayoffSet\ with a decision set.
\begin{itemize}
\item 
The operation~\( \oplus \) corresponds to making a union of decision sets:
to choose between starters \emph{or} desserts, one makes a list made of all starters and all desserts;
then one selects the best in the (linear) list.
\item 
The operation~\( \otimes \) correspond to multiplying decisions:
to choose \emph{within} starters \emph{and within} desserts,
one makes a list with all pairs (starter, dessert);
then one selects the best in the (product) list.
 Decisions in \( \ACTION\otimes\ACTIONbis = \ACTION+\ACTIONbis \) are obtained as images of the mapping
  \( \np{\action,\actionbis} \ni \ACTION \times \ACTIONbis \mapsto
  \action + \actionbis \in \ACTION + \ACTIONbis \).
  This is why we say that decisions are multiplied even if the image actions,
  obtained by sum, can be absorbed by the sum (see also Equation~\eqref{eq:Cartesian_product_decision_set}).
\end{itemize}

 \subsubsection{Union and fusion of decision-makers}
\label{Union_and_fusion_of_decision-makers}

Following the economic interpretation of the dioids of \ExpectedUtilityMaximizer s
in~\S\ref{Economic_interpretation_of_the_dioids_of_ExpectedUtilityMaximizer_s},
we now name the two operations of the dioids.

\subsubsubsection{Definition of union and fusion of decision-makers}

\begin{definition}
  \label{de:fusion_of_two_decision-makers}
  Consider two decision-makers --- each facing an abstract decision problem 
under imperfect information on a state of nature, belonging to the common
set~$\NATURE$ ---
  one with \RegularActionPayoffSet~\(\ACTION  \in \REGULARACTIONPAYOFFSET \)
(and \RegularValueFunction\ \( \fonctionun =\RestrictedSupportFunction{\ACTION} \in \RegularConv{\BELIEF} \)),
and the other one with
\RegularActionPayoffSet~\(\ACTIONbis  \in \REGULARACTIONPAYOFFSET \)
(and \RegularValueFunction\ \( \fonctiondeux =\RestrictedSupportFunction{\ACTIONbis} \in
\RegularConv{\BELIEF} \)).

\begin{subequations}
\begin{itemize}
\item
The \emph{union} of the two decision-makers is the new decision-maker
with \RegularActionPayoffSet\ \( {\ACTION\oplus\ACTIONbis} \in \REGULARACTIONPAYOFFSET \)
(and \RegularValueFunction\  \( \fonctionun\oplus\fonctiondeux =
\RestrictedSupportFunction{\ACTION\oplus\ACTIONbis} \in \RegularConv{\BELIEF} \))
given by~\eqref{eq:two_operations_subsets_oplus_continuous}.
%
\item
The \emph{fusion} of the two decision-makers is the new decision-maker
with \RegularActionPayoffSet\ \( {\ACTION\otimes\ACTIONbis} \in \REGULARACTIONPAYOFFSET \)
(and \RegularValueFunction\ \( \fonctionun\otimes\fonctiondeux =
\RestrictedSupportFunction{\ACTION\otimes\ACTIONbis} \in \RegularConv{\BELIEF} \)),
given by~\eqref{eq:two_operations_subsets_otimes_continuous}.

%
\end{itemize}
\end{subequations}
\end{definition}

The above definitions make sense as both \( {\ACTION\oplus\ACTIONbis} \) and \( {\ACTION\otimes\ACTIONbis} \) indeed
are~\RegularActionPayoffSet s since \( \bp{ \REGULARACTIONPAYOFFSET \cup \na{\emptyset}, \oplus, \otimes } \) is a dioid
(see Definition~\ref{de:4C_action_set} and Item~\ref{it:isomorphism_continuous_subsets}
in Proposition~\ref{pr:isomorphism_continuous}), hence
\( \ACTION \in \REGULARACTIONPAYOFFSET \) and \( \ACTIONbis \in \REGULARACTIONPAYOFFSET \) imply that
\( {\ACTION\oplus\ACTIONbis}, {\ACTION\otimes\ACTIONbis} \in \REGULARACTIONPAYOFFSET \). 
As discussed in~\S\ref{Economic_interpretation_of_the_dioids_of_ExpectedUtilityMaximizer_s},
with the \RegularActionPayoffSet\ \( {\ACTION\oplus\ACTIONbis} \) (union), one is 
\emph{taking the best of two utilities} which corresponds to
\emph{making a union of two decisions sets};
by contrast, with the \RegularActionPayoffSet\ \( {\ACTION\otimes\ACTIONbis} \) (fusion), one is
\emph{adding utilities} but \emph{multiplying decisions}.

\subsubsubsection{Interpretation of union and fusion in the classic setting}

Suppose that the decision problem of the first economic agent is given by a decision set~$\NEUTRAL$ and by
a utility function~$\NEUTRALpayoff_{\NEUTRAL} \colon \NEUTRAL\times \NATURE\to \RR$,
whereas the decision problem of the second economic agent is given by a decision set~$\NEUTRALbis$ and by
a utility function~$\NEUTRALpayoff_{\NEUTRALbis}\colon \NEUTRALbis\times \NATURE\to \RR$,
with \( \NEUTRAL\cap\NEUTRALbis=\emptyset\). 

\begin{itemize}
\item 
The {union of the two economic agents} is the decision problem given by
the union decision set~$\NEUTRAL\cup\NEUTRALbis$ and by\footnote{%
  The function~\( \1_{\NEUTRAL} \) takes the value~1 on~\( \NEUTRAL \) and~0 elsewhere,
  and the same for~\( \1_{\NEUTRALbis} \).}
the utility function~$\NEUTRALpayoff_{\NEUTRAL}\1_{\NEUTRAL} +\NEUTRALpayoff_{\NEUTRALbis}\1_{\NEUTRALbis}
\colon \np{\NEUTRAL\cup\NEUTRALbis}\times \NATURE\to \RR$,
defined by
\begin{subequations}
\begin{equation}
  \np{\NEUTRALpayoff_{\NEUTRAL}\1_{\NEUTRAL} +\NEUTRALpayoff_{\NEUTRALbis}\1_{\NEUTRALbis}}\np{z,\nature}
  =
  \begin{cases}
    \NEUTRALpayoff_{\NEUTRAL}(\neutral,\nature) & \textrm{if } z=\neutral\in\NEUTRAL
    \\
   \NEUTRALpayoff_{\NEUTRALbis}(\neutralbis,\nature) & \textrm{if } z=\neutralbis\in\NEUTRALbis
  \end{cases}
\eqsepv \forall z \in \NEUTRAL\cup\NEUTRALbis  \eqsepv \forall \nature \in \NATURE 
  \eqfinp
\label{eq:union_in_the_classic_setting}
\end{equation}
The value function in~\eqref{eq:value_function_utility_classic} is given,
for any belief \( \belief \in \BELIEF \), by
\begin{equation}
  \begin{split}
  \sup_{z \in \NEUTRAL\cup\NEUTRALbis}\sum_\nature \belief_\nature 
   \np{\NEUTRALpayoff_{\NEUTRAL}\1_{\NEUTRAL} +\NEUTRALpayoff_{\NEUTRALbis}\1_{\NEUTRALbis}}\np{z,\nature}
 \\  =\sup\Ba{
     \sup_{\neutral\in\NEUTRAL}\sum_\nature \belief_\nature \NEUTRALpayoff_{\NEUTRAL}(\neutral,\nature),
     \sup_{\neutralbis\in\NEUTRALbis}\sum_\nature \belief_\nature \NEUTRALpayoff_{\NEUTRALbis}(\neutralbis,\nature),
}
 =\sup\ba{\ValueFunction{\NEUTRALpayoff_{\NEUTRAL}}(\belief),\ValueFunction{\NEUTRALpayoff_{\NEUTRALbis}}(\belief)}
  \eqfinp
  \end{split}
\end{equation}
\end{subequations}
\item 
The {fusion of the two economic agents} is the decision problem given by
the Cartesian product decision set~$\NEUTRAL\times \NEUTRALbis$ and by  
the utility function~$\NEUTRALpayoff_{\NEUTRAL} +\NEUTRALpayoff_{\NEUTRALbis} \colon \NEUTRAL\times \NEUTRALbis\times \NATURE\to \RR$,
defined by
\begin{subequations}
\begin{equation}
  \np{\NEUTRALpayoff_{\NEUTRAL} + \NEUTRALpayoff_{\NEUTRALbis}}\bp{\np{\neutral,\neutralbis},\nature}
  =\NEUTRALpayoff_{\NEUTRAL}(\neutral,\nature)+\NEUTRALpayoff_{\NEUTRALbis}(\neutralbis,\nature) \eqsepv
  \forall \neutral\in\NEUTRAL \eqsepv \forall \neutralbis\in \NEUTRALbis  \eqsepv \forall \nature \in \NATURE 
  \eqfinp
  \label{eq:fusion_in_the_classic_setting}
\end{equation}
Here again, we see that, by fusion of two decision-makers, we are {multiplying decisions}
and adding utilities.
Notice that, when an individual makes two successive decisions, one after the other,
and that he has time-additive separable preferences, his utility is the sum of two utilities:
hence this individual is the fusion of his two successive economic agents,
but if the state of nature is the same at both stages.

The value function in~\eqref{eq:value_function_utility_classic} is given,
for any belief \( \belief \in \BELIEF \), by
\begin{equation}
  \begin{split}
  \sup_{\np{\neutral,\neutralbis}\in \NEUTRAL\times\NEUTRALbis}\sum_\nature \belief_\nature 
 \np{\NEUTRALpayoff_{\NEUTRAL} + \NEUTRALpayoff_{\NEUTRALbis}}\bp{\np{\neutral,\neutralbis},\nature}
 \\  = \sup_{\neutral\in\NEUTRAL}\sum_\nature \belief_\nature \NEUTRALpayoff_{\NEUTRAL}(\neutral,\nature) + 
     \sup_{\neutralbis\in\NEUTRALbis}\sum_\nature \belief_\nature \NEUTRALpayoff_{\NEUTRALbis}(\neutralbis,\nature)
 =\ValueFunction{\NEUTRALpayoff_{\NEUTRAL}}(\belief)+\ValueFunction{\NEUTRALpayoff_{\NEUTRALbis}}(\belief)
  \eqfinp
  \end{split}
\end{equation}
 \end{subequations}
\end{itemize}

\subsubsection{Two kinds of flexibility: by union and by fusion}
\label{Two_kinds_of_flexibility:_union_and_fusion}

From the definition of union and fusion of decision-makers
in~\S\ref{Union_and_fusion_of_decision-makers}, 
we deduce two kinds of flexibility: by union and by fusion.

\begin{definition}
  \label{de:additively_multiplicatively_more_flexible} 
  Consider two decision-makers with \RegularActionPayoffSet s
  \(\MORE \in \REGULARACTIONPAYOFFSET \) and \(\LESS \in \REGULARACTIONPAYOFFSET \).
  \begin{itemize}
  \item 
  We say that the decision-maker~\(\MORE \) is \emph{\( \oplus \)-more flexible}
(or \emph{more flexible by union})
  than the decision-maker~\(\LESS \) if
  there exists a~\RegularActionPayoffSet~\( \MEDIUM \in \REGULARACTIONPAYOFFSET \) such that
  \( \MORE={\LESS\oplus\MEDIUM} \).
\item 
  We say that the decision-maker~\(\MORE \) is \emph{\( \otimes \)-more flexible}
  (or \emph{more flexible by fusion})
  than the decision-maker~\(\LESS \) if
  there exists a~\RegularActionPayoffSet~\( \MEDIUM \in \REGULARACTIONPAYOFFSET \) such that
 \( \MORE={\LESS\otimes\MEDIUM} \).
  \end{itemize}
\end{definition}

Following the discussion in~\S\ref{Economic_interpretation_of_the_dioids_of_ExpectedUtilityMaximizer_s},
we can interpret \(\oplus\)- and \( \otimes \)- flexibilities as follows.
\begin{itemize}
\item 
The decision-maker~\(\MORE \) is {\( \oplus \)-more flexible} than the decision-maker~\(\LESS \) if 
decisions in~\( \MEDIUM \) are added (by union) to those of~\(\LESS \) to obtain~\(\MORE \),
after closed convex closure (the resulting operation
\( \closedconvexhull\np{\LESS\cup\MEDIUM} = \LESS\oplus\MEDIUM \) is the sum~$\oplus$
in~\eqref{eq:two_operations_subsets_oplus_continuous}).
Thus, the \(\oplus\)-flexibility is obtained by uniting more options and utilities
(see~\eqref{eq:union_in_the_classic_setting} in the classic setting).
The \(\oplus\)-flexibility is the one considered in~\cite{whitmeyer2024makinginformationvaluable}.
\item 
The decision-maker~\(\MORE \) is {\( \otimes \)-more flexible} than the decision-maker~\(\LESS \) if 
there exists a decision-maker with \RegularActionPayoffSet~\( \MEDIUM \in \REGULARACTIONPAYOFFSET \) such that
\(\MORE \) is the fusion of~\(\LESS \) with~\( \MEDIUM \)
(the resulting operation \( {\LESS +\MEDIUM} = \LESS\otimes\MEDIUM \)
 is the product~$\otimes$ in~\eqref{eq:two_operations_subsets_otimes_continuous}).
 Thus, the \(\otimes\)-flexibility is obtained by multiplying options
 and by adding utilities
(see~\eqref{eq:fusion_in_the_classic_setting} in the classic setting): 
the decision-maker~\(\MORE \) selects options in the product set \(\LESS \times \MEDIUM \)
and obtains an action in~\(\LESS \otimes \MEDIUM \), thus adding utilities.
Notice that, even if decisions are multiplied in~\(\LESS \times \MEDIUM \),
some couples of decisions \(\np{\less,\medium}\in\LESS \times \MEDIUM \) may
yield actions \(\less+\medium\in\LESS+\MEDIUM \) that are identical
or lower than actions in~\( \LESS+\MEDIUM \), that is,
there might be less actions in~\( \LESS+\MEDIUM \) than in~\(\LESS \times \MEDIUM \).
\end{itemize}

\subsection{Flexibility by union and more valuable information}
\label{More_on_flexibility_and_more_valuable_information}

The main result of Theorem~\ref{th:more_valuable_information} is that
\(\otimes\)-flexibility (flexibility by fusion)
is a necessary and sufficient condition for more valuable information.

\begin{theorem}
  \label{th:more_valuable_information_bis}
Consider two decision-makers with \RegularActionPayoffSet s
  \(\MORE \in \REGULARACTIONPAYOFFSET \) and \(\LESS \in \REGULARACTIONPAYOFFSET \). 
  Then, the following statements are equivalent.
  \begin{enumerate}
      \item
    \label{it:values_more_information_bis}
    The decision-maker~$\MORE$ values information (weakly) more than the
    decision-maker~$\LESS$.
\item
     \label{it:multiplicatively_more_flexible}
  The decision-maker~$\MORE$ is {\( \otimes \)-more flexible} than the
   decision-maker~$\LESS$.
  \end{enumerate}
\end{theorem}

This seems to settle the question regarding flexibility.
However, flexibilities by union or fusion are not exclusive, so that
we ask the question: what is the room for \(\oplus\)-flexibility to lead to more valuable information?
Given a decision-maker with \RegularActionPayoffSet~\(\LESS \in \REGULARACTIONPAYOFFSET \), 
the answers lies in the possibility, or not, to find two \RegularActionPayoffSet s
\( \ACTION \in \REGULARACTIONPAYOFFSET \) and
\( \MEDIUM \in \REGULARACTIONPAYOFFSET \) that solve the equation
\( \LESS \oplus \ACTION = \LESS \otimes \MEDIUM \).
There are trivial (without economic relevance) solutions like
\( \LESS \oplus \emptyset = \LESS = \LESS \otimes \RR_{-}^{\NATURE} \)
--- as \( \emptyset \) is the neutral (or zero) element and \( \RR_{-}^{\NATURE} \) is the unit element
of the dioid of \RegularActionPayoffSet s 
(see Item~\ref{it:isomorphism_continuous_subsets} in Proposition~\ref{pr:isomorphism_continuous}) ---
or \( \LESS \oplus \LESS = \LESS = \LESS \otimes \RR_{-}^{\NATURE} \) ---
using the idempotency of the $\oplus$~operation.
There are less trivial examples like the one depicted in Figure~\ref{fig:polyhedralActionPayoffSet_oplus_otimes}.

In the coming Propositions~\ref{pr:additive_flexibility_more_valuable_information}
and~\ref{pr:additive_flexibility_more_valuable_information_little},
we are going to provide conditions --- necessary, sufficient, and necessary and sufficient --- 
under which more valuable information can be obtained by \(\oplus\)-flexibility.
Thus doing, we contribute to the program carried  in~\cite{whitmeyer2024makinginformationvaluable}.
We use the notion of normal cone lattice as introduced in Definition~\ref{de:Normal_cone_lattice}.


\begin{proposition}
\label{pr:additive_flexibility_more_valuable_information}  
  Consider two decision-makers with \RegularActionPayoffSet s
  \(\LESS \in \REGULARACTIONPAYOFFSET \) and \( \ACTION \in \REGULARACTIONPAYOFFSET \).
    The union of both decision-makers has \RegularActionPayoffSet~\( \LESS \oplus \ACTION \).
\begin{enumerate}
\item
\label{it:additive_flexibility_more_valuable_information_CNS}  
The decision-maker~\( \LESS \oplus \ACTION \) values more information than the decision-maker~\( \LESS \)
if and only if 
\( \max\na{0, \RestrictedSupportFunction{\ACTION}-\RestrictedSupportFunction{\LESS}} \) is a convex function
(on~\( \BELIEF \)).
\item
  \label{it:additive_flexibility_more_valuable_information_CN}  
If the decision-maker~\( \LESS \oplus \ACTION \) values more information than the decision-maker~\( \LESS \), 
then, on the one hand,\footnote{%
  We use the notation \( \na{ \RestrictedSupportFunction{\ACTION}\leq\RestrictedSupportFunction{\LESS} } 
  =\defset{\belief\in\BELIEF}{ \RestrictedSupportFunction{\ACTION}\np{\belief}
    \leq\RestrictedSupportFunction{\LESS}\np{\belief} } \), and later
  \( \na{ \RestrictedSupportFunction{\LESS} \leq \RestrictedSupportFunction{\ACTION} =\scalpro{\cdot}{\action} } \)
  \( = \defset{\belief\in\BELIEF}{
\RestrictedSupportFunction{\LESS}\np{\belief} \leq
    \RestrictedSupportFunction{\ACTION}\np{\belief}
 =\scalpro{\belief}{\action} } \),  
  \( \na{ \RestrictedSupportFunction{\LESS}=\scalpro{\cdot}{\less} }
  = \defset{\belief\in\BELIEF}{ \RestrictedSupportFunction{\LESS}\np{\belief}
    = \scalpro{\belief}{\less} } \),
  \( \na{ \RestrictedSupportFunction{\ACTION} \leq \scalpro{\cdot}{\less} }\)
\(  =\defset{\belief\in\BELIEF}{ \RestrictedSupportFunction{\ACTION}\np{\belief}
    \leq \scalpro{\belief}{\less} } \),
  \( \na{ \RestrictedSupportFunction{\ACTION} \leq \RestrictedSupportFunction{\LESS} } \)
\( = \defset{\belief\in\BELIEF}{ \RestrictedSupportFunction{\ACTION}\np{\belief}
    \leq\RestrictedSupportFunction{\LESS}\np{\belief} } \), etc.
}
\( \na{ \RestrictedSupportFunction{\ACTION}\leq\RestrictedSupportFunction{\LESS} } \) is a convex subset
of~\( \BELIEF \) and, on the other hand,
the normal cone lattice~$\NORMALCONE(\LESS\oplus\ACTION)$ refines (is included in)
the normal cone lattice~$\NORMALCONE(\LESS)$,
which implies that 
\begin{equation}
  \forall\action\in\ACTION \eqsepv \exists \less\in\LESS \eqsepv
\na{ \RestrictedSupportFunction{\LESS} \leq \RestrictedSupportFunction{\ACTION} =\scalpro{\cdot}{\action} }  
\subset
\na{ \RestrictedSupportFunction{\LESS}=\scalpro{\cdot}{\less} }
  \eqfinp
  \label{eq:additive_flexibility_more_valuable_information_CN}
\end{equation}
\item
  \label{it:additive_flexibility_more_valuable_information_CS}  
 If either there exists \( \ACTIONbis  \in \REGULARACTIONPAYOFFSET \) such that
  \( \ACTION = \LESS \otimes \ACTIONbis \),
  or if 
  there exists \( \less\in\LESS \) such that
  \begin{subequations}
    \begin{equation}
\na{ \scalpro{\cdot}{\less}\leq\RestrictedSupportFunction{\ACTION} }
\subset
\na{ \RestrictedSupportFunction{\LESS}\leq\RestrictedSupportFunction{\ACTION}}
\subset 
\na{ \RestrictedSupportFunction{\LESS}=\scalpro{\cdot}{\less} }
    \eqfinv
    \label{eq:additive_flexibility_more_valuable_information_CS}
  \end{equation}
then the decision-maker~\( \LESS \oplus \ACTION \) values more information than
the decision-maker~\( \LESS \),
and we have that
\begin{align}
\max\na{0, \RestrictedSupportFunction{\ACTION}-\RestrictedSupportFunction{\LESS}}
  &=
    \max\na{0, \RestrictedSupportFunction{\ACTION}-\scalpro{\cdot}{\less} }
    =
\RestrictedSupportFunction{ \RR_{-}^{\NATURE} \oplus \np{\ACTION-\less} }
\eqfinv
    \label{eq:additive_flexibility_more_valuable_information_CS_convex_function}
\\
  \LESS \oplus \ACTION 
  &=
    \LESS \otimes \bp{ \RR_{-}^{\NATURE} \oplus \np{\ACTION-\less} }
  \eqfinp 
\end{align}
\end{subequations}
\end{enumerate}
\end{proposition}


Becoming \emph{a little more flexible} is a notion introduced
in \cite[Sect.~2]{whitmeyer2024makinginformationvaluable}, where flexibility 
is obtained by the union of a single new decision.

\begin{proposition}
  \label{pr:additive_flexibility_more_valuable_information_little}
  Consider a decision-maker with \RegularActionPayoffSet~\(\LESS \in \REGULARACTIONPAYOFFSET \).
  Consider also \( \hat{\neutral}\in \RR^{\NATURE} \) and the associated
  decision-maker with \RegularActionPayoffSet~\( 
  \RR_{-}^{\NATURE} \otimes \na{\hat{\neutral}} \in \REGULARACTIONPAYOFFSET \).
  The union of both decision-makers has \RegularActionPayoffSet~\( 
  \LESS \oplus \np{ \RR_{-}^{\NATURE} \otimes \na{\hat{\neutral}} } \).
%
\begin{enumerate}
\item
\label{it:additive_flexibility_more_valuable_information_little_CNS}  
The decision-maker~\( \LESS \oplus \np{ \RR_{-}^{\NATURE} \otimes \na{\hat{\neutral}} } \)
values more information than the decision-maker~\( \LESS \)
if and only if 
\( \max\ba{0, \scalpro{\cdot}{\hat{\neutral}}-\RestrictedSupportFunction{\LESS}} \) is a convex function
(on~\( \BELIEF \)).
\item
  \label{it:additive_flexibility_more_valuable_information_little_CN}
  If the decision-maker~\( \LESS \oplus \np{ \RR_{-}^{\NATURE} \otimes \na{\hat{\neutral}} } \)
  values more information than the decision-maker~\( \LESS \), 
then, on the one hand,
\( \na{ \scalpro{\cdot}{\hat{\neutral}} \leq\RestrictedSupportFunction{\LESS} } \) is a convex subset
of~\( \BELIEF \) and, on the other hand,
there exists \( \less\in\LESS \) such that
\begin{equation}
  \na{ \RestrictedSupportFunction{\LESS} \leq \scalpro{\cdot}{\hat{\neutral}} }
  \subset
\na{ \RestrictedSupportFunction{\LESS}=\scalpro{\cdot}{\less} }
\eqfinp
\label{eq:additive_flexibility_more_valuable_information_little_CN}
\end{equation}
\item
  \label{it:additive_flexibility_more_valuable_information_little_CS}  
  If there exists \( \less\in\LESS \) such that
  \begin{subequations}
    \begin{equation}
\na{ \scalpro{\cdot}{\less} \leq \scalpro{\cdot}{\hat{\neutral}} }
\subset
\na{ \RestrictedSupportFunction{\LESS} \leq \scalpro{\cdot}{\hat{\neutral}} }
\subset 
\na{ \RestrictedSupportFunction{\LESS}=\scalpro{\cdot}{\less} }
    \eqfinv
    \label{eq:additive_flexibility_more_valuable_information_little_CS}
  \end{equation}
then the decision-maker~\( \LESS \oplus \np{ \RR_{-}^{\NATURE} \otimes \na{\hat{\neutral}} } \)
values more information than the decision-maker~\( \LESS \),
and we have that
\begin{align}
\max\ba{0, \scalpro{\cdot}{\hat{\neutral}}-\RestrictedSupportFunction{\LESS}}
  &=
    \max\ba{0, \scalpro{\cdot}{\hat{\neutral}-\less}}
    = \RestrictedSupportFunction{\RR_{-}^{\NATURE}\otimes\ClosedIntervalClosed{0}{\hat{\neutral}-\less}} 
\eqfinv
    \label{eq:additive_flexibility_more_valuable_information_little_CS_convex_function}
\\
  \LESS \oplus \np{ \RR_{-}^{\NATURE} \otimes \na{\hat{\neutral}} }
  &=
    \LESS \otimes 
    \np{\RR_{-}^{\NATURE}\otimes\ClosedIntervalClosed{0}{\hat{\neutral}-\less}} 
  \eqfinv
    \label{eq:additive_flexibility_more_valuable_information_little_CS}  
\end{align}
\end{subequations}
  where \( \ClosedIntervalClosed{0}{\hat{\neutral}-\less} \subset \RR^{\NATURE} \)
  denotes the segment between~$0$ and~\( \hat{\neutral}-\less \).

\end{enumerate}
\end{proposition}

       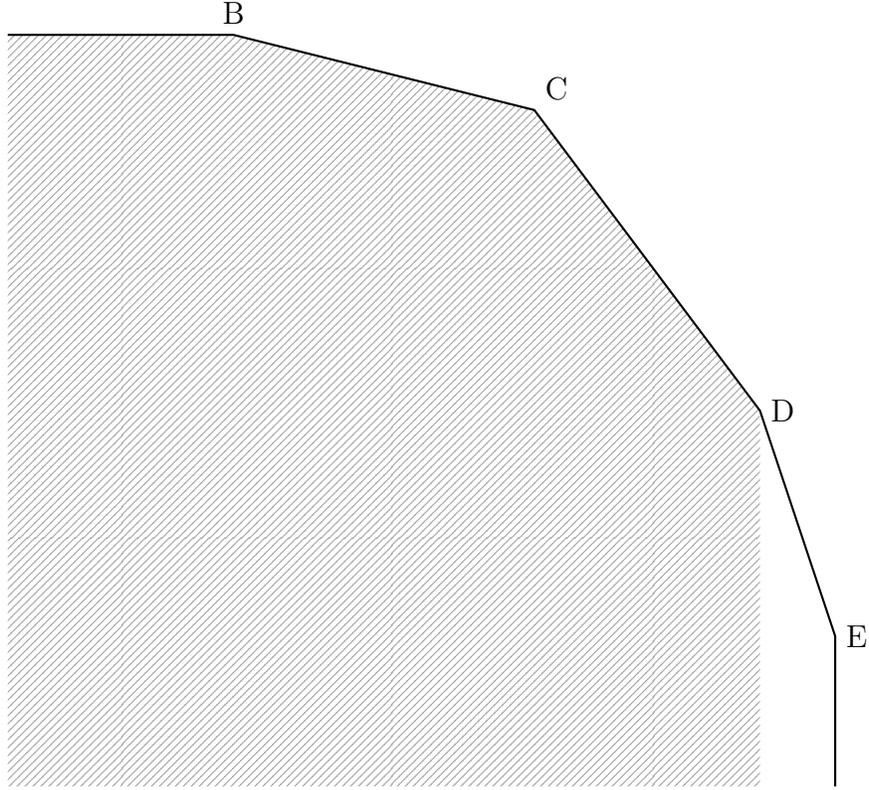
\begin{figure}[hbt!]
      \centering
      \begin{tikzpicture}
   \coordinate (A) at (-3,0);    
  \coordinate (B) at (0,0);       \node[above] at (B) {B};
  \coordinate (C) at (4,-1);      \node[above right] at (C) {C};
  \coordinate (D) at (7,-5);      \node[right] at (D) {D};
  \coordinate (E) at (8,-8);     \node[right] at (E) {E};
  \coordinate (F) at (8,-10);   
  \coordinate (G) at (7,-10);
  \coordinate (H) at (-3,-10);
  
  \fill[pattern=north east lines, pattern color=gray!70]
    (A) -- (B) -- (C) -- (D) -- (G) -- (H);

  \draw[thick]
    (A) -- (B) -- (C) -- (D) -- (E) -- (F);
  \end{tikzpicture}
  \caption{If \( \LESS \) denotes the hashed polyhedral {\ActionPayoffSet} (see 
    Figure~\ref{fig:polyhedralActionPayoffSet}), the polyhedral {\ActionPayoffSet} to the left of
    the polyline \(\infty\)-B-C-D-E-\(\infty\) is obtained either as \( \LESS
    \oplus \na{E} = \closedconvexhull\bp{\LESS\cup\na{E}} \) or as 
    \( \LESS \otimes \ClosedIntervalClosed{0}{E-D} =  \LESS +
    \ClosedIntervalClosed{0}{E-D} \)}
      \label{fig:polyhedralActionPayoffSet_oplus_otimes}
    \end{figure}

\subsection{Relation with the literature}
\label{Relation_with_the_literature}

We relate our main result, Theorem~\ref{th:more_valuable_information} ---
and its variants Propositions~\ref{pr:more_valuable_information_implies_in_the_classic_setting}
and~\ref{pr:more_valuable_information_isimplied_in_the_classic_setting},
and also Propositions~\ref{pr:additive_flexibility_more_valuable_information}
and~\ref{pr:additive_flexibility_more_valuable_information_little} --- 
to results in three recent papers 
\cite{whitmeyer2024makinginformationvaluable},
\cite{Denti:2022} and \cite{Yoder:2022}.

\subsubsection{\cite{whitmeyer2024makinginformationvaluable}}

%

What we call \emph{valuing information (weakly) more} in Definition~\ref{de:VoIMore} is called
\emph{generating a greater value of information}
in \cite[Definition~2.1]{whitmeyer2024makinginformationvaluable}, but the definitions are the same.
We do not consider the notion of \emph{generating less information acquisition} of
\cite[Definition~2.2]{whitmeyer2024makinginformationvaluable}.

\subsubsubsection{\cite[Theorem~3.1]{whitmeyer2024makinginformationvaluable} and
Theorem~\ref{th:more_valuable_information}}

In our Theorem~\ref{th:more_valuable_information},
the equivalence between Item~\ref{it:values_more_information}
and Item~\ref{it:difference_is_convex} corresponds to
the equivalence between (i) and (ii) in
\cite[Theorem~3.1]{whitmeyer2024makinginformationvaluable},
a result already established, in one form or another,
in \cite[Proposition~0]{Jones-Ostroy:1984} (see also
\cite[Footnote~8]{whitmeyer2024makinginformationvaluable} which points to another source).
This is not our main contribution.
The main result of our paper is the equivalence between Item~\ref{it:values_more_information}
and Item~\ref{it:fusion} (or Item~\ref{it:fusion_star-difference})
in Theorem~\ref{th:more_valuable_information}, a result which is new (and not obtained
in \cite{whitmeyer2024makinginformationvaluable}).

\subsubsubsection{Refining and becoming more flexible}

Becoming more flexible 
in \cite[Sect.~2]{whitmeyer2024makinginformationvaluable} is obtained by
the union of a set of actions, which corresponds 
to \(\oplus\)-flexibility (Definition~\ref{de:additively_multiplicatively_more_flexible}).

To analyze the impact of flexibility on the value of information, Whitmeyer studies,
in \cite[\S~3.2]{whitmeyer2024makinginformationvaluable},
polyhedral functions --- inherited as value functions associated with finite decision sets
--- by projecting their (polyhedral) epigraph
onto the beliefs, thus obtaining a finite collection of polytopes, called \emph{cells}.
We make the connection with classic notions in the geometry of convex sets,
not necessarily polyhedral. For this purpose, we refer the reader to
the background in Sect.~\ref{Faces_exposed_faces_and_normal_cones}.
The \emph{cells} are a particular case (because the definition goes beyond polyhedra)
of \emph{normal cones} intersected with beliefs
as recalled in Definition~\ref{de:Normal_cone_lattice};
the \emph{regular polyhedral subdivision}~$C$ is a particular case (because the definition goes beyond polyhedra)
of \emph{normal cone lattice}\footnote{%
For convex polytopes, the {normal cone lattice} is called \emph{normal fan}.}
(restricted to beliefs), as recalled in Definition~\ref{de:Normal_cone_lattice};
the notion of \emph{finer} subdivision (\( \hat{C} \plusfine C \))
is a particular case (because the definition goes beyond polyhedra)
of comparison of normal cones (intersected with beliefs) by inclusion.

Then, \cite[Lemma~3.8]{whitmeyer2024makinginformationvaluable}
states that if the difference of two polyhedral functions is convex, then necessarily
the regular polyhedral subdivision of the first function refines that of the second.
In the following result, we obtain that 
\cite[Lemma~3.8]{whitmeyer2024makinginformationvaluable} can be extended to the
case of infinite action sets.
We use the notion of normal cone lattice as introduced in Definition~\ref{de:Normal_cone_lattice}.

\begin{proposition}
Consider two decision-makers with \RegularActionPayoffSet s
  \(\MORE \in \REGULARACTIONPAYOFFSET \) and \(\LESS \in \REGULARACTIONPAYOFFSET \). 
  If any of the equivalent statements in Theorem~\ref{th:more_valuable_information} holds true,
  then the normal cone lattice~$\NORMALCONE(\MORE)$ refines (is included in)
  the normal cone lattice~$\NORMALCONE(\LESS)$.
\end{proposition}

\begin{proof}
On the one hand, by Item~\ref{it:fusion} in Theorem~\ref{th:more_valuable_information}
there exists a~\RegularActionPayoffSet~\( \MEDIUM \in \REGULARACTIONPAYOFFSET \) such that
\( \MORE = {\LESS +\MEDIUM} = \overline{\LESS +\MEDIUM} \).
Hence any element \( \more\in \MORE = \LESS+\MEDIUM \)
is of the form \( \more=\less+\medium \) where \( \less\in\LESS \) and \( \medium\in\MEDIUM \).
On the other hand, by Equation~\eqref{eq:Normal_cone_Minkowski}, we have that 
\( \NormalCone\np{ \overline{\LESS +\MEDIUM}, \less+\medium }
= \NormalCone\np{\LESS,\less} \cap \NormalCone\np{\MEDIUM,\medium} \).
Thus, we get that 
\( \NormalCone\np{ \MORE, \more}
= \NormalCone\np{ \overline{\LESS +\MEDIUM}, \less+\medium }
= \NormalCone\np{\LESS,\less} \cap \NormalCone\np{\MEDIUM,\medium}
\subset \NormalCone\np{\LESS,\less} \), which is refinement.
\end{proof}

\subsubsubsection{Refining and becoming a little more flexible}

Becoming a little more flexible 
in \cite[Sect.~2]{whitmeyer2024makinginformationvaluable} is obtained by
the union of a singleton, which 
is a special case of \(\oplus\)-flexibility
(Definition~\ref{de:additively_multiplicatively_more_flexible}).

For finite decision sets, \cite[Lemma~4.2]{whitmeyer2024makinginformationvaluable} claims that if
the regular polyhedral subdivision of the little more flexible decision-maker
is a refinement of the regular polyhedral subdivision of the original (less flexible) decision-maker,
then little flexibility leads to more valuable information. 
This has to be compared with Item~\ref{it:additive_flexibility_more_valuable_information_little_CS}
in our Proposition~\ref{pr:additive_flexibility_more_valuable_information_little}.
To obtain that little flexibility leads to more valuable information,
we require~\eqref{eq:additive_flexibility_more_valuable_information_little_CS}:
there exists \( \less\in\LESS \) such that
\( \na{ \scalpro{\cdot}{\less} \leq \scalpro{\cdot}{\hat{\neutral}} }
\subset
\na{ \RestrictedSupportFunction{\LESS} \leq \scalpro{\cdot}{\hat{\neutral}} }
\subset 
\na{ \RestrictedSupportFunction{\LESS}=\scalpro{\cdot}{\less} } \).
In~\eqref{eq:additive_flexibility_more_valuable_information_little_CS},
the condition 
\( \na{ \RestrictedSupportFunction{\LESS} \leq \scalpro{\cdot}{\hat{\neutral}} }
\subset 
\na{ \RestrictedSupportFunction{\LESS}=\scalpro{\cdot}{\less} } \)
is equivalent\footnote{%
  See Footnote~\ref{ft:almost}, where the necessary condition is also sufficient here,
  as we are in the case where \( \convexhull\np{\LESS\cup\ACTION} =
  {\LESS\oplus\ACTION} = \closedconvexhull\np{\LESS\cup\ACTION} \) with
\( \ACTION = \RR_{-}^{\NATURE}+\na{\hat{\neutral}} \).
}
  to the property that the normal cone lattice~$\NORMALCONE(\LESS\oplus\na{\hat{\neutral}})$ refines (is included in)
  the normal cone lattice~$\NORMALCONE(\LESS)$, as shown in the proof of
  Proposition~\ref{pr:additive_flexibility_more_valuable_information_little}
  in~\S\ref{Proof_of_Proposition_ref_pr:additive_flexibility_more_valuable_information_little}.
  However, in~\eqref{eq:additive_flexibility_more_valuable_information_little_CS},
  we also require the assumption
\( \na{ \scalpro{\cdot}{\less} \leq \scalpro{\cdot}{\hat{\neutral}} }
\subset
\na{ \RestrictedSupportFunction{\LESS} \leq \scalpro{\cdot}{\hat{\neutral}} } \),  
which is not stated in
\cite[Lemma~4.2]{whitmeyer2024makinginformationvaluable}
(and we do not see in what way it is related to the refinement assumption in
\cite[Lemma~4.2]{whitmeyer2024makinginformationvaluable}).

\subsubsubsection{Affine transformations of the agent's utility function}

We obtain the equivalent of \cite[Proposition~4.9]{whitmeyer2024makinginformationvaluable}
by using the property that
\( \kappa\LESS = \LESS + (\kappa-1) \LESS = \LESS \otimes (\kappa-1) \LESS \) if \( \kappa\geq 1 \)
and
\( \LESS = \kappa\LESS + (1-\kappa) \LESS = \kappa\LESS \otimes (1-\kappa) \LESS \) if \( \kappa\leq 1 \)
by \cite[p.~140]{Schneider:2014} as \( \LESS \) is convex,
and then by the equivalence between Item~\ref{it:values_more_information}
and Item~\ref{it:fusion} in Theorem~\ref{th:more_valuable_information}.


\subsubsection{\cite{Denti:2022}}

In \cite[I.~B., p.~3221]{Denti:2022}, value functions of a prior belief on \( \NATURE=\Theta \)
are denoted by \( \phi_F \), where $F$ is a finite menu of utility acts (\( F \subset \RR^\Theta \)). 
Then, in \cite[V.~A., p.~3221]{Denti:2022}, Denti points out that the property that \( \phi_F-\phi_G \) is convex
increases the incentive to acquire information and leads to more extreme beliefs.
Our result that Item~\ref{it:difference_is_convex} (\( \phi_F-\phi_G \) is convex)
in Theorem~\ref{th:more_valuable_information} implies
Item~\ref{it:values_more_information} (information is more valuable) is in phase with the claim of an 
increase in the incentive to acquire information, as this latter has more value.

Moreover, using Item~\ref{it:fusion} (or Item~\ref{it:fusion_star-difference}) in Theorem~\ref{th:more_valuable_information}
--- but also Propositions~\ref{pr:more_valuable_information_implies_in_the_classic_setting}
and~\ref{pr:more_valuable_information_isimplied_in_the_classic_setting},
and Propositions~\ref{pr:additive_flexibility_more_valuable_information}
and~\ref{pr:additive_flexibility_more_valuable_information_little} --- 
can help providing
menus~$F$ and~$G$ satisfying \( \phi_F-\phi_G \) is convex,
hence to develop experimental designs as in \cite[IV.~B., IV.~C]{Denti:2022}.
More precisely, Item~\ref{it:fusion} in our Theorem~\ref{th:more_valuable_information}
points to extending the menu~$G$ by \emph{multiplying options}
--- or, under specific conditions in Item~\ref{it:additive_flexibility_more_valuable_information_CS}
in our Proposition~\ref{pr:additive_flexibility_more_valuable_information}
and in Item~\ref{it:additive_flexibility_more_valuable_information_little_CS}
in our Proposition~\ref{pr:additive_flexibility_more_valuable_information_little}, by adding options ---
to obtain that \( \phi_F-\phi_G \) is convex.

\subsubsection{\cite{Yoder:2022}}

Right before \cite[Proposition~1]{Yoder:2022}, Yoder says 
that a function~$f$ is \emph{additively more concave} than a function~$g$
if $f = g + h$ for some continuous, strictly concave function~$h$.
The functions $f$ and $g$ typically are value functions of a prior belief on \( \NATURE=\na{0,1} \),
as appears in the sentence right after and also in \cite[Proposition~1]{Yoder:2022}.
When $f$ and $g$ are continuous value functions, additively more concave implies that
$g-f$ is convex, hence that Item~\ref{it:difference_is_convex}
in Theorem~\ref{th:more_valuable_information} holds true, hence that Item~\ref{it:values_more_information} holds true,
which ties up with the discussion after \cite[Proposition~1]{Yoder:2022}.

Using Item~\ref{it:fusion} in Theorem~\ref{th:more_valuable_information} ---
and also Item~\ref{it:additive_flexibility_more_valuable_information_CS}
in our Proposition~\ref{pr:additive_flexibility_more_valuable_information}
and Item~\ref{it:additive_flexibility_more_valuable_information_little_CS}
in our Proposition~\ref{pr:additive_flexibility_more_valuable_information_little} ---
can help providing additively more concave pairs of value functions $f$ and $g$
(with the caveat that $g-f$ is convex, and not necessarily strictly convex).

\section{Conclusion}
\label{Conclusion}

In this paper, we have focused on the question:
 for a given economic agent, what changes in decision variables and utility function
 make information more valuable?
 This question is studied in \cite{whitmeyer2024makinginformationvaluable}
 for changes that we coined as ``\(\oplus\)-flexibility'', or flexibility by union.
 The results in \cite{whitmeyer2024makinginformationvaluable} indicate
 that, by extending the set of options by adding (union) decision variables,
 there is no systematic comparison regarding the value of information;
 stringent conditions are required to make information more valuable by adding decision variables.

 With our main result in this paper, we propose an explanation for the above observation.
 Systematic comparison regarding the value of information is obtained if and
 only if the set of options is extended by \emph{multiplying} decision
 variables and adding utility --- that we coined as ``\(\otimes\)-flexibility'', or flexibility by fusion.
 Aside the original set of options, the economic agent has to select
 \emph{also} in a new set of options, leading to a pair of decisions (instead of
 a single one in the original problem), and then adding utilities.
 
 We have also introduced abstract decision problems and 
 dioids of \ExpectedUtilityMaximizer s that represent such problems under two
 different complementary angles. These dioids express two operations on
 decision makers, namely union and fusion.
 We hope that these new tools can help contribute to the literature on the value
 of information.
\bigskip

\textbf{Acknowledgments:} I thank Carlos Al\'os-Ferrer for his economic comments,
and Antoine Deza and Lionel Pournin for discussions about the Minkowski addition.

 \appendix

\section{Additional background on convex analysis}

We consider a nonempty finite set~$\NATURE$, that represents states of nature.

\subsection{Additional background on duality}

\subsubsubsubsection{Functions}

As we manipulate functions with values in~$\barRR = [-\infty,+\infty] $,
we adopt, when needed, the Moreau \emph{lower ($\LowPlus$) and upper ($\UppPlus$) additions} \cite{Moreau:1970},
which extend the usual addition~($+$) with 
\( \np{+\infty} \LowPlus \np{-\infty}=\np{-\infty} \LowPlus \np{+\infty}=-\infty \) and
\( \np{+\infty} \UppPlus \np{-\infty}=\np{-\infty} \UppPlus \np{+\infty}=+\infty \).
For any set~\( \UNCERTAIN \) and any subset \( \Uncertain \subset \UNCERTAIN \), we denote by
$\Indicator{\Uncertain} \colon \UNCERTAIN \to \barRR $ the
\emph{indicator function} of the set~$\Uncertain$, defined by
\( \Indicator{\Uncertain}\np{\uncertain} = 0 \) if
\( \uncertain \in \Uncertain \), and
\( \Indicator{\Uncertain}\np{\uncertain} = +\infty \) if
\( \uncertain \not\in \Uncertain \).

\subsubsubsubsection{Subsets}

For any two subsets $\Primal, \Primalbis\subset\RR^{\NATURE}$, we have that
\cite[Lemma~2.1]{Fradelizi-Madiman-Marsiglietti-Zvavitch:2018}
\begin{equation}
    \convexhull\np{\Primal+\Primalbis}
    =
      \convexhull\np{\Primal} + \convexhull\np{\Primalbis}
      \eqfinv 
      \label{eq:convexhull_Minkowski_sum}  
\end{equation}

\subsubsubsubsection{Duality}

The {support function} of a subset in~\eqref{eq:support_function} 
satisfies the following well-known properties\footnote{%
The \( \LowPlus \) in~\eqref{eq:support_function_Minkowski} accounts for the case where
\( \Primal=\emptyset \), hence \( \SupportFunction{\Primal}=-\infty \).}
\cite[Chapter~V, Section~2]{Hiriart-Urruty-Lemarechal-I:1993},
\begin{subequations}
  \begin{align}
    \SupportFunction{\Primal}
    &=
      \SupportFunction{\overline{\Primal}}=
      \SupportFunction{\convexhull{\Primal}}
      = \SupportFunction{\closedconvexhull{\Primal}}
      \eqsepv \forall \Primal, \Primalbis \subset\RR^{\NATURE}
      \eqfinv
       \label{eq:support_function_closedconvexhull}
    \\
    \SupportFunction{\Primal+\Primalbis}
    &=
      \SupportFunction{\Primal} \LowPlus \SupportFunction{\Primalbis}
            \eqsepv \forall \Primal, \Primalbis \subset\RR^{\NATURE}
      \eqfinv
      \label{eq:support_function_Minkowski}
          \\
    \SupportFunction{\Primal}
    &=
      \SupportFunction{\Primalbis} \implies
      \closedconvexhull{\Primal}=\closedconvexhull{\Primalbis}
            \eqsepv \forall \Primal, \Primalbis \subset\RR^{\NATURE}
      \eqfinv
      \label{eq:support_function_equal}
    \\
\SupportFunction{\cup_{i\in I}\Primal_i}
    &=
      \sup_{i\in I}\SupportFunction{\Primal_i}
      \eqsepv \forall \sequence{\Primal_i}{i\in I} \subset\RR^{\NATURE} 
      \eqfinp 
      \label{eq:support_function_union_sup}
  \end{align}
\end{subequations}

\subsection{Background on faces, exposed faces and normal cones}
\label{Faces_exposed_faces_and_normal_cones}

We follow \cite{Weis:2012} in the case of the finite dimensional real Euclidean vector
space~\( \RR^{\NATURE} \). 

\renewcommand{\PRIMAL}{\RR^{\NATURE}}
\renewcommand{\DUAL}{\RR^{\NATURE}}
\newcommand{\PrimalConvex}{X}
\newcommand{\DualConvex}{Y}
\newcommand{\PrimalConvexbis}{X'}
\newcommand{\DualConvexbis}{Y'}

\begin{subequations}
  \begin{definition}[Exposed face lattice]
  For any nonempty closed convex subset $\Convex \subset \PRIMAL$ and dual vector~$\dual \in \DUAL$,
  the exposed face of $\Convex$ at~$\dual$ is 
  \begin{equation}
    \ExposedFace(\Convex,\dual) 
    =\argmax_{\primal\in\Convex} \proscal{\primal}{\dual} \subset \Convex
    \eqfinp
\label{eq:ExposedFace}
  \end{equation}
  The set of exposed faces of~$\Convex$ will be denoted $\EXPOSEDFACE(\Convex)$
  and, when equipped with a lattice structure, will be called
  the \emph{exposed face lattice} of~$\Convex$.
\label{de:exposed_face_lattice}
\end{definition}
Using the support function~\( \SupportFunction{\Convex}\) in~\eqref{eq:support_function},
we get that
\begin{equation}
  \primal\in \ExposedFace(\Convex,\dual) \iff
  \primal\in\Convex \text{ and }
  \SupportFunction{\Convex}\np{\dual} - \proscal{\primal}{\dual} = 0
  \eqfinp
  \label{eq:ExposedFace_support_function}
\end{equation}
\end{subequations}
With this and~\eqref{eq:support_function_Minkowski}, we easily obtain that, for any nonempty closed convex subsets
$\Convex_1, \Convex_2 \subset \PRIMAL$ and dual vector~$\dual \in \DUAL$, we have that
\begin{equation}
  \ExposedFace(\overline{\Convex_1+\Convex_2},\dual) = 
  \ExposedFace(\Convex_1,\dual) + \ExposedFace(\Convex_2,\dual)
      \eqfinp
  \end{equation}

We denote by $\ri(\Primal)$ the relative interior of a nonempty convex subset~$\Primal$ of~$\PRIMAL$.

  \begin{definition}[Normal cone lattice]
    
\begin{subequations}
  For any nonempty closed convex subset $\Convex \subset \PRIMAL$
  and primal vector~\( \primal\in\Convex \), the normal cone~$\NormalCone(\Convex,\primal)$ 
  is defined by the conjugacy relation 
\begin{equation}
 \primal \in \Convex \mtext{ and } \dual \in \NormalCone(\Convex,\primal) 
\iff 
\primal \in \ExposedFace(\Convex,\dual) 
\eqfinv
\end{equation}
or, equivalently, by
\begin{equation}
 \primal \in \Convex \mtext{ and } \dual \in \NormalCone(\Convex,\primal) 
 \iff
  \primal\in\Convex \text{ and }
  \SupportFunction{\Convex}\np{\dual} - \proscal{\primal}{\dual} = 0
  \eqfinp
  \label{eq:NormalCone_support_function}
\end{equation}
\end{subequations}
Then, we define the \emph{normal cone} of a nonempty convex subset~$\Primal$ 
of~$\Convex$ by $\NormalCone(\Convex,\Primal)= \NormalCone(\Convex,\primal)$ for any $\primal\in \ri(\Primal)$.
The definition is consistent as $\NormalCone(\Convex,\primal)=
\NormalCone(\Convex,\primalbis)$
when $\primal$ and $\primalbis$ are both in $\ri(\Primal)$.
The definition is extended when $\Primal=\emptyset$ by
$\NormalCone(\Convex,\emptyset)= \DUAL$.

The \emph{normal cone lattice}~$\NORMALCONE(\Convex)$ of the nonempty closed convex subset~$\Convex \subset \PRIMAL$ is the
  set of the normal cones of all exposed faces, that is,
  \begin{equation}
    \NORMALCONE(\Convex)
    =\bset{ \NormalCone(\Convex,\Face)}{ \Face \in \EXPOSEDFACE(\Convex)}
    \eqfinp
  \end{equation}
  \label{de:Normal_cone_lattice}
\end{definition}

With~\eqref{eq:ExposedFace_support_function} and~\eqref{eq:support_function_Minkowski},
we easily obtain that, for any nonempty closed convex subsets
$\Convex_1, \Convex_2 \subset \PRIMAL$ and (primal) vectors
\( \primal_1 \in \Convex_1, \primal_2 \in \Convex_2 \), we have that
\begin{equation}
  \NormalCone\np{\overline{\Convex_1+\Convex_2},\primal_1+\primal_2}
= \NormalCone\np{\Convex_1,\primal_1} \cap \NormalCone\np{\Convex_2,\primal_2}
      \eqfinp
      \label{eq:Normal_cone_Minkowski}
    \end{equation}

    \section{Technical Propositions and proofs}

    We consider a nonempty finite set~$\NATURE$, that represents states of nature.
    
\subsection{Positively homogeneous extension of a function}


We show how a function over beliefs~\( \BELIEF \subset \RR^{\NATURE} \)
can be uniquely extended into a positively homogeneous function over~\( \RR^{\NATURE} \),
and related properties.

\begin{proposition}
  \label{pr:positively_homogeneous_extension}
  \quad
  \begin{enumerate}
  \item
\label{it:positively_homogeneous_extension_extension}
  Let \( \fonctionun \colon \BELIEF \to \barRR \) be a function,
which is not identically equal to~\( -\infty \).
  Then, there exists a unique function
  \( \fonctiondeux \colon \RR^{\NATURE}\to \barRR \) such that
 \begin{subequations}
    \begin{align}
      \fonctiondeux(\belief)
      &=
        \fonctionun(\belief)
      \eqsepv \forall \belief \in \BELIEF
           \eqfinv
                       \label{eq:positively_homogeneous_extension_coincide}
       \\
      \fonctiondeux(\lambda\signed)
      &=
\lambda \fonctiondeux(\signed)
        \eqsepv \forall \lambda \in \RR_{++}
                \eqsepv \forall \signed \in \RR_{+}^{\NATURE}\setminus\na{0} 
        \eqfinv
             \label{eq:positively_homogeneous_extension_homogeneous}
      \\
     \fonctiondeux(0)
      &=
        0
                \eqfinv
             \label{eq:positively_homogeneous_extension_0}
      \\
      \fonctiondeux(\signed)
      &=
+\infty  \eqsepv \forall \signed \in \RR^{\NATURE}\setminus\RR_{+}^{\NATURE}
        \eqfinp
             \label{eq:positively_homogeneous_extension_infty}
    \end{align}
      \label{eq:positively_homogeneous_extension}
    \end{subequations}
    We denote the function~\( \fonctiondeux \) by \( \widehat{\fonctionun} \colon \RR^{\NATURE}\to \barRR \).
    In case \( \fonctionun \equiv -\infty \) is the function identically equal to~\( -\infty \),
    we set \( \widehat{\fonctionun} \equiv -\infty \)
    (which does not satisfy~\eqref{eq:positively_homogeneous_extension_0} and~\eqref{eq:positively_homogeneous_extension_infty}).
    The function~\( \widehat{\fonctionun} \) is positively homogeneous,
    and we call it the \emph{positively homogeneous extension} of the function
     \( \fonctionun \colon \BELIEF \to \barRR \).
\item
  \label{it:positively_homogeneous_extension_comprehensive_set}
We have that 
      \begin{equation}
        \ACTION\in\ACTIONPAYOFFSET \implies
        \widehat{\RestrictedSupportFunction{\ACTION}}=\SupportFunction{\ACTION}
        \eqfinp
        \label{eq:positively_homogeneous_extension_comprehensive_set}
      \end{equation}
    \item
      \label{it:positively_homogeneous_extension_Conv} 
      Let \( \fonctionun \colon \BELIEF \to \barRR \) be a closed convex function,
      that is, \( \fonctionun \in \Conv{\BELIEF} \). 
Then, the function~\( \widehat{\fonctionun} \) is positively homogeneous
closed convex. 
Moreover, we have that
    \begin{subequations}
      \begin{equation}
\fonctionun \in \Conv{\BELIEF} \implies
        \widehat{\fonctionun} = \SupportFunction{\InverseRestrictedSupportFunction{\fonctionun}}
        \eqfinv 
\label{eq:positively_homogeneous_extension_Conv}
      \end{equation}
      where the subset \( \InverseRestrictedSupportFunction{\fonctionun} \)
      is defined in~\eqref{eq:isomorphism_b}, and satisfies
    \begin{equation}
      \InverseRestrictedSupportFunction{\fonctionun} 
      = \defset{\primal\in \RR^{\NATURE}}{ \scalpro{\signed}{\primal}
        \leq \widehat{\fonctionun}\np{\signed} \eqsepv \forall \signed\in \RR^{\NATURE}}
      \eqsepv \forall \fonctionun \in \Conv{\BELIEF} 
\eqfinp
\label{eq:positively_homogeneous_extension_InverseRestrictedSupportFunction}
\end{equation}
\end{subequations}
  \end{enumerate}
  %
\end{proposition}

\begin{proof}
  \quad
    \begin{enumerate}
  \item
    (Proof of Item~\ref{it:positively_homogeneous_extension_extension})
    
  To prove that there exists a unique function \( \fonctiondeux \colon \RR^{\NATURE}\to \barRR \) that
satisfies~\eqref{eq:positively_homogeneous_extension},
it suffices to show~\eqref{eq:positively_homogeneous_extension_homogeneous}.
For that purpose, denoting by $1 \in \RR_{+}^{\NATURE}$ the vector with all components equal
to~$1$, 
if~\eqref{eq:positively_homogeneous_extension_coincide} and~\eqref{eq:positively_homogeneous_extension_homogeneous}
hold true, we necessarily have that 
\begin{equation}
  \fonctiondeux(\signed) =  \scalpro{\signed}{1}
  \fonctionun\np{\frac{\signed}{\scalpro{\signed}{1}}}
        \eqsepv \forall \signed \in \RR_{+}^{\NATURE}\setminus\na{0} 
        \eqfinp 
        \label{eq:positively_homogeneous_extension_formula}
      \end{equation}
      So, we get uniqueness. Then, it is easy to see that \eqref{eq:positively_homogeneous_extension_formula}
      implies~\eqref{eq:positively_homogeneous_extension_coincide} and~\eqref{eq:positively_homogeneous_extension_homogeneous}.
      By extending the function \( \fonctiondeux \) beyond \( \RR_{+}^{\NATURE}\setminus\na{0} \)
      by~\eqref{eq:positively_homogeneous_extension_0} and~\eqref{eq:positively_homogeneous_extension_infty},
      we obtain the unique function \( \fonctiondeux \colon \RR^{\NATURE}\to \barRR \) that
      satisfies~\eqref{eq:positively_homogeneous_extension}.
      From now on, we denote \( \widehat{\fonctionun} = \fonctiondeux \),
      and \( \widehat{\fonctionun} \equiv -\infty \) in the case \( \fonctionun \equiv -\infty \).
It is straightforward, from~\eqref{eq:positively_homogeneous_extension}, to observe that
\begin{equation}
  \epigraph\widehat{\fonctionun} = \np{\RR_{++} \epigraph\fonctionun} \cup \np{\na{0}\times\RR_{+}}
  \eqfinp
\label{eq:epigraph_fonctiondeux}
\end{equation}
As \( \epigraph\widehat{\fonctionun} \) is a cone,
the function~\( \widehat{\fonctionun} \) is positively homogeneous.
This remains true in the case \( \fonctionun \equiv -\infty \),
as \( \widehat{\fonctionun} \equiv -\infty \) is positively homogeneous.

  \item
    (Proof of Item~\ref{it:positively_homogeneous_extension_comprehensive_set})
    
    Let \( \ACTION\in\ACTIONPAYOFFSET \). 
    We set \( \fonctionun=\RestrictedSupportFunction{\ACTION} \).
If \( \ACTION\neq\emptyset \), we observe that the function~\( \fonctiondeux=\SupportFunction{\ACTION} \)
    satisfies~\eqref{eq:positively_homogeneous_extension_coincide},
    \eqref{eq:positively_homogeneous_extension_homogeneous}
    and~\eqref{eq:positively_homogeneous_extension_0} (because \( \ACTION\neq\emptyset \)).
    The function~\( \fonctiondeux=\SupportFunction{\ACTION} \)
    also satisfies~\eqref{eq:positively_homogeneous_extension_infty} because
\begin{align*}
  \SupportFunction{\ACTION}
  &=
    \SupportFunction{\ACTION \otimes \RR_{-}^{\NATURE}}
    \tag{as \( \ACTION \otimes \RR_{-}^{\NATURE} = \ACTION \) by Definition~\ref{de:action_set}
of~\( \ACTIONPAYOFFSET \)}
  \\
  &=
    \SupportFunction{\closedconvexhull\np{\ACTION + \RR_{-}^{\NATURE}}}
\tag{by definition~\eqref{eq:two_operations_subsets_otimes} of \( \ACTION \otimes \RR_{-}^{\NATURE} \)}
  \\
  &=
    \SupportFunction{\ACTION + \RR_{-}^{\NATURE}}
    \tag{by~\eqref{eq:support_function_closedconvexhull}}
  \\
  &=
\SupportFunction{\ACTION} \LowPlus \SupportFunction{\RR_{-}^{\NATURE}}
     \tag{by~\eqref{eq:support_function_Minkowski}}
  \\
   &=
\SupportFunction{\ACTION} \LowPlus \Indicator{\RR_{+}^{\NATURE}} 
     \tag{as it is easily seen that
     \(\SupportFunction{\RR_{-}^{\NATURE}} = \Indicator{\RR_{+}^{\NATURE}} \)
\cite[Example.~2.3.1]{Hiriart-Urruty-Lemarechal-I:1993} 
     }
\eqfinp
\end{align*}
Now, as \( \ACTION\neq\emptyset \), the function~\( \SupportFunction{\ACTION} \)
never takes the value~\( -\infty \), so that
\( \SupportFunction{\ACTION} = \SupportFunction{\ACTION} \LowPlus \Indicator{\RR_{+}^{\NATURE}} =
\SupportFunction{\ACTION} + \Indicator{\RR_{+}^{\NATURE}} \)
takes the value~\( +\infty \) outside of~\( \RR_{+}^{\NATURE} \).
Thus, we have proven that the function~\( \fonctiondeux=\SupportFunction{\ACTION} \)
satisfies~\eqref{eq:positively_homogeneous_extension}.
By the just proven Item~\ref{it:positively_homogeneous_extension_extension},
we conclude that \( \widehat{\RestrictedSupportFunction{\ACTION}}=\SupportFunction{\ACTION} \).
This is also true if \( \ACTION=\emptyset \), because then
\( \fonctionun=\RestrictedSupportFunction{\ACTION} =\RestrictedSupportFunction{\emptyset}\equiv-\infty \)
and \( \widehat{\RestrictedSupportFunction{\ACTION}}=
\widehat{\fonctionun} \equiv -\infty  = \SupportFunction{\emptyset} \). 

      \item
        (Proof of Item~\ref{it:positively_homogeneous_extension_Conv})

      Let \( \fonctionun \colon \BELIEF \to \barRR \) be a closed convex function,
      that is, \( \fonctionun \in \Conv{\BELIEF} \). 
   In case, \( \fonctionun \equiv -\infty \) is the function identically equal to~\( -\infty \),
   \eqref{eq:positively_homogeneous_extension_Conv} holds true
   as \( \widehat{\fonctionun} \equiv -\infty = \SupportFunction{\emptyset} \),
   and~\eqref{eq:positively_homogeneous_extension_InverseRestrictedSupportFunction}
    holds true with \( \InverseRestrictedSupportFunction{\fonctionun}=\emptyset \).

So, from now on, we suppose that the function \( \fonctionun \colon \BELIEF \to \barRR \) is 
proper convex \lsc.
As the function \( \fonctionun \colon \BELIEF \to \barRR \) is proper,
it is obvious by~\eqref{eq:positively_homogeneous_extension_coincide}
that the function \( \widehat{\fonctionun} \colon \RR^{\NATURE}\to \barRR \) is proper.
We are going to prove that the conic epigraph~\( \epigraph\widehat{\fonctionun} \)
given by~\eqref{eq:epigraph_fonctiondeux} is closed convex.
For that purpose, we show that
\( \epigraph\widehat{\fonctionun} = \overline{\RR_{++} \epigraph\fonctionun} \)
by two inclusions.

To prove the inclusion \( \epigraph\widehat{\fonctionun} \subset \overline{\RR_{++} \epigraph\fonctionun} \),
it suffices to show that
\( \na{0}\times \RR_{+} \subset \overline{\RR_{++} \epigraph\fonctionun} \),
because of~\eqref{eq:epigraph_fonctiondeux}.
As the function \( \fonctionun \colon \BELIEF \to \RR \) is proper,
there exists  \( \belief\in\BELIEF \) such that 
\( -\infty < \fonctionun(\belief) < +\infty \).
Then, for any \( t\in \RR_{+} \), we get that 
  \( \RR_{++}\epigraph\fonctionun \ni
  \bp{ \frac{1}{n} \belief, \frac{1}{n} \fonctionun\np{\belief}+t} 
  \to_{n\to +\infty} \np{0,t} \in \na{0}\times \RR_{+} \). 
  Thus, we have shown that
  \( \na{0}\times \RR_{+} \subset \overline{\RR_{++} \epigraph\fonctionun} \),
  hence that 
  \( \epigraph\widehat{\fonctionun} \subset \overline{\RR_{++} \epigraph\fonctionun} \)
  by~\eqref{eq:epigraph_fonctiondeux}.

  To prove the reverse inclusion
  \( \epigraph\widehat{\fonctionun} \supset  \overline{\RR_{++} \epigraph\fonctionun} \), 
  we consider sequences 
\( \sequence{\belief_n}{n\in\NN} \subset \BELIEF \), 
\( \sequence{\lambda_n}{n\in\NN} \subset \RR_{++} \), 
and \( \sequence{t_n}{n\in\NN} \subset \RR_{+} \)
such that the sequence
\( \sequence{\bp{\lambda_n\belief_n,\lambda_n\fonctionun(\belief_n)+t_n}}{n\in\NN}
\subset \RR_{++} \epigraph\fonctionun \) 
converges to \( \np{\bar\signed,\bar z} \in \RR^{\NATURE}\times \RR \).
Notice that \( \np{\bar\signed,\bar z} \in \RR_{+}^{\NATURE}\times \RR \),
as $\RR_{++}\BELIEF = \RR_{+}^{\NATURE}\setminus\na{0}$, hence
$\overline{\RR_{++}\BELIEF} = \RR_{+}^{\NATURE}$.
We are going to show that \( \np{\bar\signed,\bar z} \in \epigraph\widehat{\fonctionun} \)
given by~\eqref{eq:epigraph_fonctiondeux}.
From \( \lambda_n\belief_n \to_{n\to +\infty} \bar\signed \) and
\( \scalpro{\belief_n}{1}=1 \), we get that \( \lambda_n = \scalpro{\lambda_n\belief_n}{1} \to_{n\to +\infty}
\scalpro{\bar\signed}{1} \in \RR_{+} \).
As \( \sequence{\belief_n}{n\in\NN} \subset \BELIEF \), where \( \BELIEF \) is a compact set,
we can always suppose (up to relabeling) that there exists \( \bar\belief\in\BELIEF \) such that 
\( \belief_n \to_{n\to +\infty} \bar\belief \in\BELIEF \),
hence that \( \liminf_{n\to +\infty}\fonctionun(\belief_n) \geq \fonctionun(\bar\belief) \)
since the function \( \fonctionun \colon \BELIEF \to \RR \) is \lsc.
Then, from \( \bp{\lambda_n\belief_n,\lambda_n\fonctionun(\belief_n)+t_n}
\to_{n\to +\infty} \np{\bar\signed,\bar z} \),
  we deduce, on the one hand, that 
  \( \bar\signed =  \lim_{n\to +\infty}\lambda_n\belief_n = \scalpro{\bar\signed}{1}\bar\belief \)
  and, on the other hand, that
  \begin{equation*}  
 \bar z =  \lim_{n\to +\infty}\lambda_n\fonctionun(\belief_n)+\underbrace{t_n}_{\geq 0}
  \geq \liminf_{n\to +\infty}\lambda_n\fonctionun(\belief_n)
  = \underbrace{\scalpro{\bar\signed}{1}}_{\geq 0} \liminf_{n\to +\infty}\fonctionun(\belief_n) 
  \geq \scalpro{\bar\signed}{1} \fonctionun(\bar\belief)
   \eqfinv 
\end{equation*}
  hence that
\( \bar\signed = \scalpro{\bar\signed}{1}\bar\belief \) and 
\( \bar z \geq \scalpro{\bar\signed}{1} \fonctionun(\bar\belief) \).
Then, we consider two cases.
In the case where \( \scalpro{\bar\signed}{1} > 0 \), we get that
\( \np{\bar\signed,\bar z} \in \RR_{++} \epigraph\fonctionun \subset \epigraph\widehat{\fonctionun} \)
by~\eqref{eq:epigraph_fonctiondeux}.
In the case where \( \scalpro{\bar\signed}{1} = 0 \), we get that
\( \np{\bar\signed,\bar z}=\np{0,\bar z} \in \na{0}\times\RR_{+} \subset \epigraph\widehat{\fonctionun} \)
by~\eqref{eq:epigraph_fonctiondeux}.
Thus, we have proved that \( \overline{\RR_{++} \epigraph\fonctionun} \subset \epigraph\widehat{\fonctionun} \).

Having proven both
\( \epigraph\widehat{\fonctionun} \subset  \overline{\RR_{++} \epigraph\fonctionun} \)
and \( \overline{\RR_{++} \epigraph\fonctionun} \subset \epigraph\widehat{\fonctionun} \),
we conclude that
\( \epigraph\widehat{\fonctionun} =\overline{\RR_{++}\epigraph\fonctionun} \).
As a consequence that the function \( \fonctionun \colon \BELIEF \to \RR \) is
convex, \( \epigraph\fonctionun \) is a convex set, hence
\( \overline{\RR_{++}\epigraph\fonctionun} \) is a closed convex cone, hence
\( \epigraph\widehat{\fonctionun} =\overline{\RR_{++}\epigraph\fonctionun} \)
is a closed convex cone \cite[Proposition~6.2]{Bauschke-Combettes:2017}. 
 By \cite[Chapter~V, Proposition~1.1.3]{Hiriart-Urruty-Lemarechal-I:1993}, 
the function~\( \widehat{\fonctionun} \) is sublinear and \lsc.
As a consequence,
by \cite[Chapter~V, Theorem~3.1.1]{Hiriart-Urruty-Lemarechal-I:1993}, 
we have that \( \widehat{\fonctionun} = \SupportFunction{S} \), where
\(  S= \defset{\primal\in \RR^{\NATURE}}{ \scalpro{\belief}{\primal}
    \leq \widehat{\fonctionun}\np{\signed} \eqsepv \forall \signed\in \RR^{\NATURE}} \).
As the function~\( \widehat{\fonctionun} \) satisfies~\eqref{eq:positively_homogeneous_extension},
it is easy to see that
\begin{equation*}
  S=\defset{\primal\in \RR^{\NATURE}}{ \scalpro{\signed}{\primal}
    \leq \widehat{\fonctionun}\np{\signed} \eqsepv \forall \signed\in \RR^{\NATURE}}
  = \defset{\primal\in \RR^{\NATURE}}{ \scalpro{\belief}{\primal}
  \leq \fonctionun\np{\belief} \eqsepv \forall \belief\in\BELIEF } =
\InverseRestrictedSupportFunction{\fonctionun}
\eqfinv
\end{equation*}
where this latter set is defined in~\eqref{eq:isomorphism_b}.
Thus, we have obtained that 
\( \widehat{\fonctionun} = \SupportFunction{\InverseRestrictedSupportFunction{\fonctionun}} \),
where \( \InverseRestrictedSupportFunction{\fonctionun} \)
satisfies~\eqref{eq:positively_homogeneous_extension_InverseRestrictedSupportFunction}. 
The subset \( \InverseRestrictedSupportFunction{\fonctionun} \) is not empty
because the function~\( \widehat{\fonctionun} \) is proper.

  \end{enumerate}

This ends the proof.
\end{proof}

    \subsection{More on dioids of \ExpectedUtilityMaximizer s}

In~\S\ref{Dioids_of_value_functions_over_beliefs}, we introduce dioids of value functions over beliefs.
In~\S\ref{Dioids_of_ActionPayoffSet_s}, we introduce dioids of \ActionPayoffSet s.
In~\S\ref{Inverse_homomorphisms_between_dioids}, we prove that the two dioids above
are isomorphic.

\subsubsubsection{Additional background on dioids}
In addition to the notions introduced in~\S\ref{The_dioids_of_ExpectedUtilityMaximizer_s},
we provide additional background on dioids.
A dioid is \emph{complete} if it is closed for infinite ``sums''
and if distributivity of $\otimes$ \wrt\ $\oplus$ holds true with infinite ``sums''
\cite[Definition~4.32]{Baccelli-Cohen-Olsder-Quadrat:1992}. 
%
%
A mapping between two dioids is \emph{lower semicontinuous} (\lsc) if
it sends infinite sums to infinite sums
\cite[Definition~4.43]{Baccelli-Cohen-Olsder-Quadrat:1992}. 
Thus, a \lsc\ homormorphism between two dioids sends infinite sums to infinite sums, products to products,
neutral element to neutral element, and unit element to unit element
%


\subsubsection{Dioids of functions}
\label{Dioids_of_value_functions_over_beliefs}

We equip the set~\( \barRR^{\RR^{\NATURE}} \) of
functions \( \RR^{\NATURE}\to\barRR \) with the two operations~\( \oplus, \otimes \) given by
\begin{subequations}
  \begin{align}
    \oplus_{i\in I}\fonctionun_i
    &=
      \sup_{i\in I}\fonctionun_i 
      \eqsepv \forall \sequence{\fonctionun_i}{i\in I} 
      \text{ with } \fonctionun_i \colon \RR^{\NATURE}\to\barRR \eqsepv \forall i\in I
      \eqfinv
  \label{eq:two_operations_functions_oplus}
    \\
    \fonctionun \otimes \fonctiondeux
    &=
      \fonctionun \LowPlus \fonctiondeux
  \eqsepv \forall \fonctionun, \fonctiondeux \colon \RR^{\NATURE}\to\barRR 
      \eqfinp
        \label{eq:two_operations_functions_otimes}
  \end{align}
  \label{eq:two_operations_functions}
\end{subequations}


\begin{proposition}
  \label{pr:dioids_of_value_functions_over_beliefs}
  Endowed with the two operations~\eqref{eq:two_operations_functions},
\( \np{ \barRR^{\RR^{\NATURE}}, \oplus, \otimes } \) is a complete dioid.
By restricting to the set~\( \Conv{\BELIEF} \) of closed convex functions \( \BELIEF \to \barRR \),
we obtain a complete dioid, that we call
the \emph{dioid \( \bp{ \Conv{\BELIEF}, \oplus, \otimes } \) of value functions over beliefs}.
By restricting to~\( \RegularConv{\BELIEF} \cup \na{-\infty}\), we obtain 
the dioid \( \bp{ \RegularConv{\BELIEF} \cup \na{-\infty}, \oplus, \otimes } \) of
  \RegularValueFunction s (over beliefs) (defined right after 
Proposition~\ref{pr:isomorphism_continuous}). 
\end{proposition}

\begin{proof}
  It is easy to establish that \( \np{ \barRR^{\RR^{\NATURE}}, \oplus, \otimes } \)
  is a commutative and associative algebra.
  The neutral (or zero) element (for~$\oplus$) is \( \epsilon=-\infty \),
  the function identically equal to~\( -\infty \),
  which is an absorbing element for~$\otimes$
  (as it is absorbing for the~$\LowPlus$ operation in~\eqref{eq:two_operations_functions_otimes}).
The unit element (for~$\otimes$) is \( e=0 \), the function identically equal to~\( 0 \).
The operation~$\oplus$ is idempotent.
There remains to show that $\otimes$ is distributive \wrt~$\oplus$.
For this purpose, we consider a family \( \sequence{\fonctionun_i}{i\in I} \)
with \( \fonctionun_i \colon \RR^{\NATURE}\to\barRR \), for all \( i\in I \),
and \( \fonctionun \colon \RR^{\NATURE}\to\barRR \).
As \( \fonctionun \LowPlus \sup_{i\in I}\fonctionun_i 
= \sup_{i\in I}\np{\fonctionun \LowPlus \fonctionun_i} \)
by property of the Moreau lower addition~$\LowPlus$
\wrt\ the supremum operation \cite[Equation~(4.4)]{Moreau:1970}, we get that 
\begin{equation*}
  \fonctionun \otimes \np{\oplus_{i\in I}\fonctionun_i}
  = \fonctionun \LowPlus \sup_{i\in I}\fonctionun_i 
  = \sup_{i\in I}\np{\fonctionun \LowPlus \fonctionun_i}
= \oplus_{i\in I}\np{\fonctionun \otimes \fonctionun_i }
 \eqfinp 
\end{equation*}
Hence, \( \np{ \barRR^{\RR^{\NATURE}}, \oplus, \otimes } \) satisfies the axioms of a dioid.
The dioid is complete, as all the axioms hold true with
the supremum of finite or infinite family of functions. 
In addition, the function identically equal to~\( +\infty \) is an absorbing element for~$\oplus$.
\medskip

\( \bp{ \Conv{\BELIEF}, \oplus, \otimes } \) is a \emph{subdioid}
\cite[Definition~4.19]{Baccelli-Cohen-Olsder-Quadrat:1992}, 
because \( \Conv{\BELIEF} \) contains the neutral (or zero) element~\( \epsilon=-\infty \)
and the unit element~\( e=0 \), and is closed under the~$\oplus$ operation~\eqref{eq:two_operations_functions_oplus},
as closed convex functions are stable by the supremum operation.
The situation is more delicate with the~$\otimes$ operation~\eqref{eq:two_operations_functions_otimes}:
\lsc\ functions are stable by the~$\LowPlus$ addition (but not by the~$\UppPlus$ addition \cite[p.~22]{Moreau:1966-1967}),
whereas convex functions are stable by the~$\UppPlus$ addition (but not by the~$\LowPlus$ addition\footnote{%
  A valley function~$v_A$ takes the value $-\infty$ on~$A \subset \RR^{\NATURE}$ 
  and~$+\infty$   outside of~$A$ \cite{Penot:2000}. Then, if $A$ and $B$ are convex subsets of~$\RR^{\NATURE}$,
  \( v_A \UppPlus v_B =v_{A\cap B} \) is a convex function, whereas
  \( v_A \LowPlus v_B =v_{A\cup B} \) is generally not a convex function.}).
However, closed convex functions are the functions~$-\infty$ and~$+\infty$, united with 
proper convex \lsc\ functions. As a consequence, we get that
(i) on \( \Conv{\BELIEF}\setminus\na{-\infty} \)
the~$\UppPlus$ and~$\LowPlus$ additions give the same result as proper functions do not take the value $-\infty$,
(ii) the~$\LowPlus$ addition of the function~$-\infty$ with any function in~\( \Conv{\BELIEF}\) gives~$-\infty$.
Therefore, we have shown that closed convex functions are stable by the~$\LowPlus$ addition.
The dioid \( \bp{ \Conv{\BELIEF}, \oplus, \otimes } \) is complete,
as the supremum of finite or infinite family of functions in~\( \Conv{\BELIEF} \)
belongs to~\( \Conv{\BELIEF} \).
\medskip

\( \bp{ \RegularConv{\BELIEF} \cup \na{-\infty}, \oplus, \otimes } \) is a subdioid,
because \( \RegularConv{\BELIEF} \cup \na{-\infty} \) contains the neutral (or zero) element~\( \epsilon=-\infty \)
and the unit element~\( e=0 \), and is closed under the two operations~\eqref{eq:two_operations_functions}
(here, the operation~$\LowPlus$ can be replaced by the ordinary~$+$ as functions in
\( \RegularConv{\BELIEF} \) take finite values).
The dioid \( \bp{ \RegularConv{\BELIEF} \cup \na{-\infty}, \oplus, \otimes } \) is not complete,
as the supremum of an infinite family of functions in~\( \RegularConv{\BELIEF} \)
may take the value~$+\infty$.
\medskip

This ends the proof.
\end{proof}

\subsubsection{Dioids of \ActionPayoffSet s}
\label{Dioids_of_ActionPayoffSet_s}

We equip the set~\( 2^{\RR^{\NATURE}} \) of subsets of~\( \RR^{\NATURE} \)
with the two operations~\( \oplus, \otimes \) given by
  \begin{subequations}
    \begin{align}
      \oplus_{i\in I}\ACTION_i
      &=
        \closedconvexhull\np{\cup_{i\in I}\ACTION_i} 
\eqsepv \forall \sequence{\ACTION_i}{i\in I} \subset 2^{\RR^{\NATURE}} 
        \eqfinv
        \label{eq:two_operations_subsets_oplus}
      \\
      \ACTION \otimes \ACTIONbis
      &=
 \closedconvexhull\np{{\ACTION + \ACTIONbis}}       
  \eqsepv \forall \ACTION, \ACTIONbis \in 2^{\RR^{\NATURE}} 
  \eqfinp
\label{eq:two_operations_subsets_otimes}      
    \end{align}
\label{eq:two_operations_subsets}
Notice that the two operations always yield closed convex sets\footnote{%
  By contrast, the corresponding operations~\eqref{eq:two_operations_functions} on functions do not
  necessarily yield convex functions.}
  and that,
    when \(  \ACTION, \ACTIONbis \) are convex sets, the operation~\eqref{eq:two_operations_subsets_otimes}
    has the alternative expression\footnote{%
    Indeed, the set \( \ACTION + \ACTIONbis \) is convex, as the sum of two convex sets,
hence \(  \closedconvexhull\np{{\ACTION + \ACTIONbis}}=
\overline{\convexhull\np{{\ACTION + \ACTIONbis}}}=
\overline{\ACTION + \ACTIONbis} \), where we have used that
\( \closedconvexhull\Primal=\overline{\convexhull\Primal} \) \
for any subset \( \Primal\subset\RR^{\NATURE} \)
\cite[Proposition~3.46]{Bauschke-Combettes:2017}. 
}
    \begin{equation}
      \ACTION \otimes \ACTIONbis = \overline{\ACTION + \ACTIONbis}
      \eqsepv \forall \ACTION, \ACTIONbis \textrm{ convex } \in 2^{\RR^{\NATURE}}
      \eqfinp 
      \label{eq:two_operations_subsets_otimes_alternative_expression}
    \end{equation}
\end{subequations}


\subsubsubsection{Closed convex comprehensive sets}

\begin{proposition}
  \label{pr:EffectiveComprehensiveSet}
  Let \( \ACTION \subset \RR^{\NATURE} \). The following statements are equivalent.
  \begin{enumerate}
  \item
\label{it:closed_convex_EffectiveComprehensiveSet}
\( \ACTION \) is a closed convex comprehensive set, that is,
\( \ACTION \) is a closed convex set and satisfies (see Footnote~\ref{ft:comprehensive_set})
    \( \RR_{-}^{\NATURE}+{\ACTION} \subset {\ACTION} \)
  (or, equivalently, \( \RR_{-}^{\NATURE}+{\ACTION} = {\ACTION} \)).
\item
  \label{it:EffectiveComprehensiveSet_otimes_ACTION}
  \( \ACTION = \RR_{-}^{\NATURE} \otimes \ACTION \).
\item
  \label{it:EffectiveComprehensiveSet_otimes_ACTIONbis}
  There exists \( \ACTIONbis \subset \RR^{\NATURE} \) such that 
  \( \ACTION = \RR_{-}^{\NATURE} \otimes \ACTIONbis \).
  \end{enumerate}
\end{proposition}

\begin{proof}
Let \( \ACTION \subset \RR^{\NATURE} \).

\noindent$\bullet$
Suppose that \( \ACTION \) is a closed convex comprehensive set (Item~\ref{it:closed_convex_EffectiveComprehensiveSet}).
By Footnote~\ref{ft:comprehensive_set}, we have that \( \RR_{-}^{\NATURE} + \ACTION =  \ACTION \),
hence \( \overline{\RR_{-}^{\NATURE} + \ACTION} = \overline{\ACTION} = \ACTION \)
since the set~\( \ACTION \) is closed by assumption.
By~\eqref{eq:two_operations_subsets_otimes_alternative_expression},
as both \( \RR_{-}^{\NATURE} \) and \( \ACTION \) are convex sets,
we conclude that \( \RR_{-}^{\NATURE} \otimes \ACTION = \overline{\RR_{-}^{\NATURE} + \ACTION} = \ACTION \),
which is Item~\ref{it:EffectiveComprehensiveSet_otimes_ACTION}.
 \bigskip 

 \noindent$\bullet$
 It is obvious that Item~\ref{it:EffectiveComprehensiveSet_otimes_ACTION} implies
 Item~\ref{it:EffectiveComprehensiveSet_otimes_ACTIONbis}.
  \bigskip 

  \noindent$\bullet$
  Suppose that there exists \( \ACTIONbis \subset \RR^{\NATURE} \) such that 
  \( \ACTION = \RR_{-}^{\NATURE} \otimes \ACTIONbis \) (Item~\ref{it:EffectiveComprehensiveSet_otimes_ACTIONbis}).
  By~\eqref{eq:two_operations_subsets_otimes}, the set \( \RR_{-}^{\NATURE} \otimes \ACTIONbis \)
  is closed convex, hence so is~\( \ACTION \).
  By~\eqref{eq:two_operations_subsets_otimes_alternative_expression},
as both \( \RR_{-}^{\NATURE} \) and \( \ACTIONbis \) are convex sets,
we have that \( \ACTION=\RR_{-}^{\NATURE} \otimes \ACTIONbis = \overline{\RR_{-}^{\NATURE} + \ACTIONbis} \),
Thus, we get that \( \RR_{-}^{\NATURE} + \ACTION = \RR_{-}^{\NATURE} + \overline{\RR_{-}^{\NATURE} + \ACTIONbis} \),
where it is easy to obtain that \( \RR_{-}^{\NATURE} + \overline{\RR_{-}^{\NATURE} + \ACTIONbis}
= \overline{\RR_{-}^{\NATURE} + \ACTIONbis} \).
As a consequence, we have that \( \RR_{-}^{\NATURE} + \ACTION =
\overline{\RR_{-}^{\NATURE} + \ACTIONbis} = \ACTION \),
hence \( \ACTION \) is a comprehensive set.
We have proved Item~\ref{it:closed_convex_EffectiveComprehensiveSet}.

\end{proof}

\begin{definition}
\label{de:action_set} 
We say that \( \ACTION \subset \RR^{\NATURE} \) is a
\emph{\ActionPayoffSet} (on~$\NATURE$) if it satisfies any of the three
equivalent statements of Proposition~\ref{pr:EffectiveComprehensiveSet}.
      We denote by \( \ACTIONPAYOFFSET \subset 2^{\RR^{\NATURE}} \) the set of all
      \ActionPayoffSet s on~$\NATURE$, united with the empty set~$\emptyset$.
\end{definition}

\subsubsection{Inverse homomorphisms between dioids}
\label{Inverse_homomorphisms_between_dioids}

\begin{proposition}
  \label{pr:isomorphism} 
  With the two operations~\eqref{eq:two_operations_subsets}, 
  \( \bp{ \ACTIONPAYOFFSET, \oplus, \otimes } \) is a complete 
  dioid, 
with neutral (or zero) element~\( \emptyset \) and unit element~\( \RR_{-}^{\NATURE} \).  
The following mappings, from sets to functions
 \begin{subequations}
  \begin{align}
    \FromACTIONPAYOFFSETtoConvBELIEF \colon 
    \ACTIONPAYOFFSET \to \Conv{\BELIEF} \eqsepv
    &
      \ACTION \mapsto \RestrictedSupportFunction{\ACTION}
        \label{eq:isomorphism_a}
    \intertext{and from functions to sets\footnotemark}
        \FromConvBELIEFtoACTIONPAYOFFSET \colon 
    \Conv{\BELIEF} \to \ACTIONPAYOFFSET \eqsepv
    &
      \fonctionun \mapsto \InverseRestrictedSupportFunction{\fonctionun}=
      \defset{\primal\in \RR^{\NATURE}}{ \scalpro{\belief}{\primal}
    \leq \fonctionun\np{\belief} \eqsepv \forall \belief\in\BELIEF }
      \label{eq:isomorphism_b}
  \end{align}
  \label{eq:isomorphism}
  \footnotetext{See Footnote~\ref{ft:isomorphism_b_continuous_def}.}
\end{subequations}
are 
\lsc\ homormorphisms between the dioids 
\( \bp{ \ACTIONPAYOFFSET, \oplus, \otimes } \) and \( \bp{ \Conv{\BELIEF}, \oplus, \otimes } \),
inverse one to the other by 
\begin{subequations}
  \begin{align}
    \np{\FromConvBELIEFtoACTIONPAYOFFSET \circ \FromACTIONPAYOFFSETtoConvBELIEF}\np{\ACTION}
    =\InverseRestrictedSupportFunction{\RestrictedSupportFunction{\ACTION}}
    &=
      \ACTION
    \eqsepv \forall \ACTION\in\ACTIONPAYOFFSET 
    \eqfinv
   \label{eq:one-to-one_isomorphisms_a}    
    \\
\np{\FromACTIONPAYOFFSETtoConvBELIEF \circ \FromConvBELIEFtoACTIONPAYOFFSET}\np{\fonctionun}
=    \RestrictedSupportFunction{%
    \InverseRestrictedSupportFunction{\fonctionun}}
    &=
      \fonctionun \eqsepv
      \forall \fonctionun\in\Conv{\BELIEF}
      \eqfinp 
   \label{eq:one-to-one_isomorphisms_b}    
  \end{align}
   \label{eq:one-to-one_isomorphisms}    
 \end{subequations}
%
 %
\end{proposition}

%



\begin{proof}
First, we are going to show that the mappings~\eqref{eq:isomorphism}
are well defined, and in one-to-one correspondence by~\eqref{eq:one-to-one_isomorphisms}.

Let \( \ACTION\in\ACTIONPAYOFFSET \).
By the background on convex analysis in~\S\ref{Duality_between_payoffs_and_beliefs},
we know that \( \RestrictedSupportFunction{\ACTION}\in\Conv{\BELIEF} \),
hence the mapping~\eqref{eq:isomorphism_a} is well defined.
Now, we use the mapping \( \barRR^{\BELIEF} \ni {\fonctionun} \mapsto \widehat{\fonctionun} \in \barRR^{\RR^{\NATURE}} \),
as defined in Proposition~\ref{pr:positively_homogeneous_extension} in Appendix, to show 
that~\eqref{eq:one-to-one_isomorphisms_a} holds true.
Indeed, we have that
\begin{align*}
  \InverseRestrictedSupportFunction{\RestrictedSupportFunction{\ACTION}}
  &=
     \defset{\primal\in \RR^{\NATURE}}{ \scalpro{\belief}{\primal}
    \leq \RestrictedSupportFunction{\ACTION}\np{\belief} \eqsepv \forall \belief\in\BELIEF }    
\tag{by~\eqref{eq:isomorphism_b}}
  \\ 
  &=
\defset{\primal\in \RR^{\NATURE}}{ \scalpro{\belief}{\primal}
    \leq \SupportFunction{\ACTION}\np{\signed} \eqsepv \forall \signed\in \RR^{\NATURE}} 
    \intertext{by~\eqref{eq:positively_homogeneous_extension_InverseRestrictedSupportFunction}
    and by~\eqref{eq:positively_homogeneous_extension_comprehensive_set}
    in Proposition~\ref{pr:positively_homogeneous_extension},
    giving that \( \widehat{\RestrictedSupportFunction{\ACTION}}=\SupportFunction{\ACTION} \)}
  &=    
    \ACTION
    \tag{by \cite[Chapter~V, Theorem~3.1.1]{Hiriart-Urruty-Lemarechal-I:1993}, 
    since \( \ACTION \) is closed convex}
    \eqfinp
\end{align*}
    
Let \( \fonctionun\in\Conv{\BELIEF} \).
The subset \( \InverseRestrictedSupportFunction{\fonctionun} \) in~\eqref{eq:isomorphism_b}
is the intersection of half-spaces, and so is closed convex.
We are going to show that it is also a comprehensive set.
Indeed, \( \InverseRestrictedSupportFunction{\fonctionun} + \RR_{-}^{\NATURE}
\subset \InverseRestrictedSupportFunction{\fonctionun} \) because
\( \scalpro{\belief}{\primal+\primalbis} \leq \scalpro{\belief}{\primal}+0
=\scalpro{\belief}{\primal} \)
for any \( \primalbis\in \RR_{-}^{\NATURE} \) and for any \( \belief\in\BELIEF \).
As \( 0\in\RR_{-}^{\NATURE} \), we deduce that \( \InverseRestrictedSupportFunction{\fonctionun}
\subset \InverseRestrictedSupportFunction{\fonctionun} + \RR_{-}^{\NATURE} \), hence that
\( \InverseRestrictedSupportFunction{\fonctionun} + \RR_{-}^{\NATURE}
= \InverseRestrictedSupportFunction{\fonctionun} \).
Finally, we get that
\begin{align*}
\InverseRestrictedSupportFunction{\fonctionun}\otimes \RR_{-}^{\NATURE} 
  &=
    \overline{\InverseRestrictedSupportFunction{\fonctionun}+\RR_{-}^{\NATURE}}
    \tag{by~\eqref{eq:two_operations_subsets_otimes_alternative_expression} as both
    \(\InverseRestrictedSupportFunction{\fonctionun}\) and \(\RR_{-}^{\NATURE}\) are convex} 
  \\
  &=
\overline{\InverseRestrictedSupportFunction{\fonctionun}}
\tag{as we have just proved that \( \InverseRestrictedSupportFunction{\fonctionun} + \RR_{-}^{\NATURE}
= \InverseRestrictedSupportFunction{\fonctionun} \)}  
  \\
  &=
\InverseRestrictedSupportFunction{\fonctionun}
    \tag{as we have just seen that \( \InverseRestrictedSupportFunction{\fonctionun} \) is closed}
    \eqfinp 
\end{align*}
Finally, we have obtained that \( \InverseRestrictedSupportFunction{\fonctionun}\) 
is a closed convex comprehensive set.
So, we get that \( \InverseRestrictedSupportFunction{\fonctionun}\in\ACTIONPAYOFFSET \),
by Definition~\ref{de:action_set}, 
hence the mapping~\eqref{eq:isomorphism_b} is well defined.

We also have that~\eqref{eq:one-to-one_isomorphisms_b} holds true.
Indeed, on the one hand, we have that 
\( \widehat{\RestrictedSupportFunction{%
    \InverseRestrictedSupportFunction{\fonctionun}}}=
\SupportFunction{%
  \InverseRestrictedSupportFunction{\fonctionun}} \)
by~\eqref{eq:positively_homogeneous_extension_comprehensive_set},
as we have just shown that
\( \InverseRestrictedSupportFunction{\fonctionun}\in\ACTIONPAYOFFSET \).
On the other hand, we have that 
\( \widehat{\fonctionun} = \SupportFunction{\InverseRestrictedSupportFunction{\fonctionun}} \)
by~\eqref{eq:positively_homogeneous_extension_Conv},
as \( \fonctionun \in \Conv{\BELIEF} \) by assumption.
Thus, we get that \( \widehat{\RestrictedSupportFunction{%
    \InverseRestrictedSupportFunction{\fonctionun}}}=
\widehat{\fonctionun} \), from which we deduce that
\( \RestrictedSupportFunction{%
    \InverseRestrictedSupportFunction{\fonctionun}}
=      \fonctionun \) from the very definition of 
the {positively homogeneous extension} of a function \( \BELIEF \to \barRR \)
in Proposition~\ref{pr:positively_homogeneous_extension}. 
\medskip

Second, we are going to show that the mapping~\( \FromACTIONPAYOFFSETtoConvBELIEF \) in~\eqref{eq:isomorphism_a} 
sends infinite sums to infinite sums and products to products. 
This results from the following formulas
\begin{subequations}
  \begin{align}
\RestrictedSupportFunction{\oplus_{i\in I}\ACTION_i}
    &=
      \oplus_{i\in I}\RestrictedSupportFunction{\ACTION_i}
      \eqsepv \forall \sequence{\ACTION_i}{i\in I} \subset \RR^{\NATURE}
      \eqfinv
      \label{eq:RestrictedSupportFunction_homomorphism_oplus}
    \\
\RestrictedSupportFunction{\ACTION\otimes\ACTIONbis}
    &=
      \RestrictedSupportFunction{\ACTION}\otimes\RestrictedSupportFunction{\ACTIONbis} 
      \eqsepv \forall \ACTION, \ACTIONbis \in \RR^{\NATURE} 
      \eqfinv
      \label{eq:RestrictedSupportFunction_homomorphism_otimes}      
  \end{align}
\label{eq:RestrictedSupportFunction_homomorphism}
\end{subequations}
that we prove now. 
%
For any family \( \sequence{\ACTION_i}{i\in I} \subset \RR^{\NATURE} \)
(it is not necessary that \( \sequence{\ACTION_i}{i\in I} \subset \ACTIONPAYOFFSET \)), we have that
\begin{align*}
\SupportFunction{\oplus_{i\in I}\ACTION_i}
  &=
    \SupportFunction{\closedconvexhull\np{\cup_{i\in I}\ACTION_i}}
    \tag{by definition~\eqref{eq:two_operations_subsets_oplus} of \( \oplus_{i\in I}\ACTION_i \)}
    \\
  &=
    \SupportFunction{\cup_{i\in I}\ACTION_i}
   \tag{by~\eqref{eq:support_function_closedconvexhull}}
  \\
  &=
    \sup_{i\in I}\SupportFunction{\ACTION_I}
    \tag{by~\eqref{eq:support_function_union_sup}}
\\    
  &=
    \oplus_{i\in I}\SupportFunction{\ACTION_i}
    \eqfinp
    \tag{by definition~\eqref{eq:two_operations_functions_oplus}
    of \( \oplus_{i\in I}\SupportFunction{\ACTION_i}\)}
\end{align*}
By~\eqref{eq:support_function_restriction}, we deduce that
\( \RestrictedSupportFunction{\oplus_{i\in I}\ACTION_i}
= \oplus_{i\in I}\RestrictedSupportFunction{\ACTION_i} \),
hence~\eqref{eq:RestrictedSupportFunction_homomorphism_oplus}.
For any \( \ACTION, \ACTIONbis \subset \RR^{\NATURE} \)
(it is not necessary that \( \ACTION, \ACTIONbis \in \ACTIONPAYOFFSET \)), we have that
\begin{align*}
\SupportFunction{\ACTION\otimes\ACTIONbis}
  &=
    \SupportFunction{ \closedconvexhull\np{{\ACTION + \ACTIONbis}} }
    \tag{by definition~\eqref{eq:two_operations_subsets_otimes}      
    of \( \ACTION\otimes\ACTIONbis \)}
\\    
  &=
    \SupportFunction{\ACTION+\ACTIONbis}
   \tag{by~\eqref{eq:support_function_closedconvexhull}}
  \\
  &=    
    \SupportFunction{\ACTION}  \LowPlus \SupportFunction{\ACTIONbis}
    \tag{by~\eqref{eq:support_function_Minkowski}}
  \\
  &=    
 \SupportFunction{\ACTION}\otimes\SupportFunction{\ACTIONbis}
    \eqfinp
    \tag{by definition~\eqref{eq:two_operations_functions_otimes} of
    \( \SupportFunction{\ACTION}\otimes\SupportFunction{\ACTIONbis} \)}
\end{align*}
By~\eqref{eq:support_function_restriction}, we deduce that 
\( \RestrictedSupportFunction{\ACTION\otimes\ACTIONbis}
= \RestrictedSupportFunction{\ACTION}\otimes\RestrictedSupportFunction{\ACTIONbis} \),
hence~\eqref{eq:RestrictedSupportFunction_homomorphism_otimes}.
From~\eqref{eq:RestrictedSupportFunction_homomorphism}, we get that
\( \FromACTIONPAYOFFSETtoConvBELIEF\np{\oplus_{i\in I}\ACTION_i}
= \oplus_{i\in I}\FromACTIONPAYOFFSETtoConvBELIEF\np{\ACTION_i} \)
and that
\( \FromACTIONPAYOFFSETtoConvBELIEF\np{\ACTION\otimes\ACTIONbis}
=\FromACTIONPAYOFFSETtoConvBELIEF\np{\ACTION}\otimes\FromACTIONPAYOFFSETtoConvBELIEF\np{\ACTIONbis} \).
\medskip

Third, 
it is straighforward to deduce that the inverse mapping~\( \FromConvBELIEFtoACTIONPAYOFFSET \)
in~\eqref{eq:isomorphism_b} also
sends infinite sums to infinite sums, and products to products
(see the proof of \cite[Lemma~4.22]{Baccelli-Cohen-Olsder-Quadrat:1992}). 
%
%
As \( \bp{ \Conv{\BELIEF}, \oplus, \otimes } \) is a complete 
dioid,
we conclude that its image \( \bp{ \ACTIONPAYOFFSET, \oplus, \otimes } \)
by the mapping~\( \FromConvBELIEFtoACTIONPAYOFFSET \) in~\eqref{eq:isomorphism_b}
is also a complete 
dioid, because the mapping~\( \FromConvBELIEFtoACTIONPAYOFFSET \)
sends infinite sums to infinite sums, and products to products.
As \( \epsilon=-\infty \) is the neutral (or zero) element of~\( \Conv{\BELIEF} \),
and~\( e=0 \) is the unit element of~\( \Conv{\BELIEF} \),
we get that
\( \FromConvBELIEFtoACTIONPAYOFFSET\np{-\infty} = \emptyset \)
is the neutral (or zero) element of~\( \ACTIONPAYOFFSET \)
and \( \FromConvBELIEFtoACTIONPAYOFFSET\np{0} = \RR_{-}^{\NATURE} \)
is the unit element of~\( \ACTIONPAYOFFSET \).
%
\medskip

Fourth, we conclude that the mappings
\( \FromACTIONPAYOFFSETtoConvBELIEF \) and \( \FromConvBELIEFtoACTIONPAYOFFSET \)
in~\eqref{eq:isomorphism} are \lsc\ homomorphisms.
Indeed, they both sends infinite sums to infinite sums, products to products,
neutral (or zero) element to neutral (or zero) element,
and unit element to unit element.



 \medskip

 This ends the proof.
\end{proof}

Notice that the proof goes reverse to the Proposition~\ref{pr:isomorphism} statement:
we do not check that \( \bp{ \ACTIONPAYOFFSET, \oplus, \otimes } \) satisfies the axioms of a complete dioid,
but use \lsc\ homomorphisms to deduce it (at the end of the proof) from the fact that
\( \bp{ \Conv{\BELIEF}, \oplus, \otimes } \) is a dioid.
We take this route because, if it is easy to establish that \( \bp{ \ACTIONPAYOFFSET, \oplus, \otimes } \) 
is a commutative algebra with idempotent~$\oplus$, it is not immediate to show that it is associative
and distributive.

\subsection{Proof of Proposition~\ref{pr:isomorphism_continuous}}
\label{Proof_of_Proposition_ref_pr:isomorphism_continuous}

\begin{proof}
  \begin{enumerate}
   \item
    (Proof of Item~\ref{it:isomorphism_continuous_functions})
    
By Proposition~\ref{pr:dioids_of_value_functions_over_beliefs},
\( \bp{ \RegularConv{\BELIEF} \cup \na{-\infty}, \oplus, \otimes } \) is a diod.

\item
    (Proof of Item~\ref{it:isomorphism_continuous_subsets})

    By Proposition~\ref{pr:isomorphism}, the image
    \( \bp{ \REGULARACTIONPAYOFFSET \cup \na{\emptyset}, \oplus, \otimes } \)
    of the dioid \( \bp{ \RegularConv{\BELIEF} \cup \na{-\infty}, \oplus, \otimes } \)
by the homomorphism~\eqref{eq:isomorphism_b} is a diod.

Now, we prove~\eqref{eq:two_operations_subsets_otimes_continuous}.
    Let \( \ACTION, \ACTIONbis \in\REGULARACTIONPAYOFFSET \) (the case where one of them is \( \emptyset \)
    is trivial).
    As \(  \ACTION, \ACTIONbis \) are convex sets, we have already seen
    in~\eqref{eq:two_operations_subsets_otimes_alternative_expression} that
    the operation~\eqref{eq:two_operations_subsets_otimes}
    has the alternative expression
    \( \ACTION \otimes \ACTIONbis = \overline{\ACTION + \ACTIONbis} \).
    We are going to prove that \( \overline{\ACTION + \ACTIONbis}= {\ACTION + \ACTIONbis} \).

    Let \( \actionter \in \overline{\ACTION + \ACTIONbis} \). Thus, there exists two sequences
    \( \sequence{\action^{n}}{n\in\NN} \subset \ACTION \) and
    \( \sequence{\actionbis^{n}}{n\in\NN} \subset \ACTIONbis \)
    such that \( \lim_{n\to +\infty} \np{\action^{n}+\actionbis^{n}}=\actionter \).
    We are going to show that the sequence 
    \( \sequence{\np{\action^{n},\actionbis^{n}}}{n\in\NN} \subset \ACTION\times\ACTIONbis \)
    admits a convergent subsequence.

    \begin{subequations}
On the one hand, by assumption, the sequence 
\( \sequence{{\action^{n}+\actionbis^{n}}}{n\in\NN} \subset \ACTION+\ACTIONbis \)
converges, hence is bounded for any of the equivalent norms on~\( \RR^{\NATURE} \).
Using the \( \ell_{\infty} \)-norm, we thus get that there exists \( C\in\RR \) such that
\begin{equation}
  C \leq \action^{n}_{\nature}+\actionbis^{n}_{\nature} \eqsepv \forall n\in\NN
  \eqsepv \forall \nature\in\NATURE
  \eqfinp 
  \label{eq:two_operations_subsets_otimes_REGULARACTIONPAYOFFSET_lemma_a}
\end{equation}
On the other hand, by Definition~\ref{de:4C_action_set} of~\( \REGULARACTIONPAYOFFSET \),
the functions~\( \RestrictedSupportFunction{\ACTION} \) and
\( \RestrictedSupportFunction{\ACTIONbis} \) are bounded on~\( \BELIEF \).
We denote by \( \delta_{\nature} \in \BELIEF \) the Dirac mass at \( \nature\in\NATURE \).
As \( \action^{n}_{\nature}=\scalpro{\delta_{\nature}}{\action^{n}}
\leq \RestrictedSupportFunction{\ACTION}\np{\delta_{\nature}} \)
and
\( \actionbis^{n}_{\nature}=\scalpro{\delta_{\nature}}{\actionbis^{n}}
\leq \RestrictedSupportFunction{\ACTIONbis}\np{\delta_{\nature}} \), for all \( \nature\in\NATURE \),
we thus get that \(  \action^{n}_{\nature} \leq
\max_{\nature\in\NATURE}\RestrictedSupportFunction{\ACTION}\np{\delta_{\nature}} < +\infty \)
and \(  \actionbis^{n}_{\nature} \leq
\max_{\nature\in\NATURE}\RestrictedSupportFunction{\ACTIONbis}\np{\delta_{\nature}} < +\infty \).
As a consequence, there exists \( D\in\RR \) such that
\begin{equation}
\action^{n}_{\nature} \leq D \text{ and }
  \actionbis^{n}_{\nature} \leq D \eqsepv \forall n\in\NN \eqsepv \forall \nature\in\NATURE
  \eqfinp 
  \label{eq:two_operations_subsets_otimes_REGULARACTIONPAYOFFSET_lemma_b}
\end{equation}
From~\eqref{eq:two_operations_subsets_otimes_REGULARACTIONPAYOFFSET_lemma_a}
and~\eqref{eq:two_operations_subsets_otimes_REGULARACTIONPAYOFFSET_lemma_b},
we deduce that 
\begin{equation}
C-D \leq \action^{n}_{\nature} \leq D \text{ and }
C-D \leq \actionbis^{n}_{\nature} \leq D  \eqsepv \forall n\in\NN \eqsepv \forall \nature\in\NATURE
  \eqfinp 
  \label{eq:two_operations_subsets_otimes_REGULARACTIONPAYOFFSET_lemma_c}
\end{equation}
Thus, the sequence 
\( \bpsequence{\np{\action^{n},\actionbis^{n}}}{n\in\NN} \subset \ACTION\times\ACTIONbis \) is bounded
in~\( \RR^{\NATURE} \) and, as a consequence, admits a convergent subsequence:
\( \bpsequence{\np{\action^{n_m},\actionbis^{n_m}}}{m\in\NN} \to_{m\to +\infty}
\np{\action,\actionbis}\in\ACTION\times\ACTIONbis \), as \( \ACTION\) and \(\ACTIONbis \)
are closed sets.
\end{subequations}
As \( \lim_{n\to +\infty} \np{\action^{n}+\actionbis^{n}}=\actionter \), we deduce that
\( \action+\actionbis = \lim_{m\to +\infty}\action^{n_m} + \lim_{m\to +\infty}\actionbis^{n_m}
= \lim_{m\to +\infty} \np{\action^{n_m}+\actionbis^{n_m}}=\actionter \),
hence that \( \actionter \in \ACTION+\ACTIONbis \).

Thus, we have shown that  \( \overline{\ACTION + \ACTIONbis} \subset {\ACTION + \ACTIONbis} \).

  \item
    (Proof of Item~\ref{it:isomorphism_continuous_isomorphism})

    This follows from Proposition~\ref{pr:isomorphism}.
  \end{enumerate}

\end{proof}

\subsection{Proof of Theorem~\ref{th:more_valuable_information}}

\begin{proof}\quad

 \noindent$\bullet$
 That Item~\ref{it:values_more_information_strongly} implies
 Item~\ref{it:values_more_information} has been explained in the discussion
 following Definition~\ref{de:relative_VoI}.
 \bigskip 

  \noindent$\bullet$
  The equivalence between Item~\ref{it:values_more_information}
and Item~\ref{it:difference_is_convex} follows from 
  \begin{align*}
    &
      \text{the decision-maker~$\MORE$ values more information than the
    decision-maker~$\LESS$}
    \\
  \iff  &
      \VoI_{\MORE}\np{\va{\beliefbis}} \geq \VoI_{\LESS}\np{\va{\beliefbis}} 
          \eqsepv \forall \va{\beliefbis}\in\InformationStructures\np{\Omega,\BELIEF}
          \tag{by~\eqref{eq:VoIMore} in Definition~\ref{de:VoIMore}}
    \\
        \iff &
\EE \bc{\RestrictedSupportFunction{\MORE}\np{\va{\beliefbis}}}-\RestrictedSupportFunction{\MORE}\bp{\EE\nc{\va{\beliefbis}}}
      \geq \EE
      \bc{\RestrictedSupportFunction{\LESS}\np{\va{\beliefbis}}}-\RestrictedSupportFunction{\LESS}\bp{\EE\nc{\va{\beliefbis}}}
               \eqsepv \forall \va{\beliefbis}\in\InformationStructures\np{\Omega,\BELIEF}
\intertext{by definition~\eqref{eq:VoI} of the value of information~$\VoI_{\ACTION}\np{\va{\beliefbis}}$} 
   \iff &
 \EE \bc{\RestrictedSupportFunction{\MORE}\np{\va{\beliefbis}}-\RestrictedSupportFunction{\LESS}\np{\va{\beliefbis}}}
      \geq \EE \bc{ \RestrictedSupportFunction{\MORE}\bp{\EE\nc{\va{\beliefbis}}}
          -\RestrictedSupportFunction{\LESS}\bp{\EE\nc{\va{\beliefbis}}} }
      \eqsepv \forall \va{\beliefbis}\in\InformationStructures\np{\Omega,\BELIEF}
          \intertext{because we have commented, right after
          Definition~\ref{de:VoI}, that both \(
          \EE\bc{\RestrictedSupportFunction{\LESS}\np{\va{\beliefbis}}} \)
          and \( \RestrictedSupportFunction{\MORE}\bp{\EE\nc{\va{\beliefbis}}} \) belong to~\( \RR \)}
    \iff &
 \mtext{the function } \RestrictedSupportFunction{\MORE}-\RestrictedSupportFunction{\LESS} \mtext{ is convex on } \BELIEF
  \end{align*}
  by restricting to random variables \(
  \va{\beliefbis}\in\InformationStructures\np{\Omega,\BELIEF} \) taking a finite number of
  values in~$\BELIEF$ to obtain the implication~$\implies$,
  and  a straighforward application of Jensen inequality
\cite[Lemma~2.5]{Kallenberg:2002} 
  to the convex function 
\( \RestrictedSupportFunction{\MORE}-\RestrictedSupportFunction{\LESS} \colon \BELIEF \to\RR \) to obtain the reverse
implication~$\impliedby$.
 \bigskip 

  \noindent$\bullet$
  That Item~\ref{it:difference_is_convex} implies
Item~\ref{it:values_more_information_strongly} 
follows from a straighforward application of Jensen inequality,
but with conditional expectations, to the convex function 
\( \RestrictedSupportFunction{\MORE}-\RestrictedSupportFunction{\LESS} \colon \BELIEF \to\RR \).
We refer the reader to \cite[p.~33]{Doob-1953},
\cite[Chap.~II,41.4]{Dellacherie-Meyer:1975}, 
\cite{Artstein-Wets:1993}, \cite{Artstein1999:gains}, among others.
  \bigskip

\noindent$\bullet$
Thus, we have shown that Item~\ref{it:values_more_information_strongly},
 Item~\ref{it:values_more_information} 
 and Item~\ref{it:difference_is_convex} are equivalent.
  \bigskip

 
\noindent$\bullet$
We prove that Item~\ref{it:fusion} 
implies Item~\ref{it:difference_is_convex}.
Suppose that there exists a \RegularActionPayoffSet~\( \MEDIUM \in \REGULARACTIONPAYOFFSET \) such that
\( \MORE={\LESS\otimes\MEDIUM} \).
We have that 
\begin{align*}
  \RestrictedSupportFunction{\MORE}-\RestrictedSupportFunction{\LESS}
  &=
     \RestrictedSupportFunction{{\LESS\otimes\MEDIUM}} -\RestrictedSupportFunction{\LESS}
       \tag{as \( \MORE={\LESS\otimes\MEDIUM} \) by assumption}
      \\
  &=
\RestrictedSupportFunction{\LESS}\otimes\RestrictedSupportFunction{\MEDIUM}-\RestrictedSupportFunction{\LESS}
 \tag{since the mapping~\( \FromACTIONPAYOFFSETtoConvBELIEF \)
    in~\eqref{eq:isomorphism_a_continuous} is a homomorphism by
    Item~\ref{it:isomorphism_continuous_isomorphism} in Proposition~\ref{pr:isomorphism_continuous}}
  \\
    &=
    \RestrictedSupportFunction{\LESS} \LowPlus \RestrictedSupportFunction{\MEDIUM} -\RestrictedSupportFunction{\LESS}
       \tag{by~\eqref{eq:two_operations_functions_oplus}}
       \\
    &=
    \RestrictedSupportFunction{\MEDIUM}
\end{align*}
as $\RestrictedSupportFunction{\LESS}$ takes values in~$\RR$
since \( \LESS \) is a \RegularActionPayoffSet\ by assumption
(see Definition~\ref{de:4C_action_set}).
      We conclude that the function \(
      \RestrictedSupportFunction{\MORE}-\RestrictedSupportFunction{\LESS} \) is convex, because
      \( \RestrictedSupportFunction{\MEDIUM} \) is a convex function.
\bigskip

\noindent$\bullet$
We show that Item~\ref{it:difference_is_convex} implies Item~\ref{it:fusion}.

Since \( \MORE \) and \( \LESS \) are \RegularActionPayoffSet s by assumption,
the functions~$\RestrictedSupportFunction{\MORE}$ and $\RestrictedSupportFunction{\LESS}$
are continuous bounded (see Definition~\ref{de:4C_action_set}).
 We set 
\( \fonctionun=\RestrictedSupportFunction{\MORE}-\RestrictedSupportFunction{\LESS} \colon \BELIEF \to\RR \).
Thus, the function~\( \fonctionun \) is continuous bounded
as the difference of the two continuous bounded
functions~$\RestrictedSupportFunction{\MORE}$ and $\RestrictedSupportFunction{\LESS}$,
and is convex by assumption.
Thus, \( \fonctionun \in \RegularConv{\BELIEF} \). By~\eqref{eq:one-to-one_isomorphisms_b},
we get that \( \fonctionun = \RestrictedSupportFunction{\MEDIUM} \),
      where the subset \( \MEDIUM=\InverseRestrictedSupportFunction{\fonctionun} \)
      is defined in~\eqref{eq:isomorphism_b_continuous} and is such that
      \( \MEDIUM \in \REGULARACTIONPAYOFFSET \), by Definition~\ref{de:4C_action_set}.
      Summing up, we have obtained that \( \RestrictedSupportFunction{\MORE}-\RestrictedSupportFunction{\LESS}
      = \RestrictedSupportFunction{\MEDIUM} \), hence that
      \begin{align*}
        \RestrictedSupportFunction{\MORE}
        &=
          \RestrictedSupportFunction{\LESS} + \RestrictedSupportFunction{\MEDIUM}
        \intertext{as $\RestrictedSupportFunction{\LESS}$ takes values in~$\RR$
since \( \LESS \) is a \RegularActionPayoffSet\ by assumption
(see Definition~\ref{de:4C_action_set})}
        &=
 \RestrictedSupportFunction{\LESS} \otimes \RestrictedSupportFunction{\MEDIUM}
          \tag{by definition~\eqref{eq:two_operations_functions_oplus_continuous} of~\( \otimes \)}
        \\
  &=
     \RestrictedSupportFunction{{\LESS\otimes\MEDIUM}}
 \tag{since the mapping~\( \FromACTIONPAYOFFSETtoConvBELIEF \)
    in~\eqref{eq:isomorphism_a_continuous} is a homomorphism by
    Item~\ref{it:isomorphism_continuous_isomorphism} in Proposition~\ref{pr:isomorphism_continuous}}    
    \eqfinp
      \end{align*}
      Now, all three sets \( \MORE \), \( \LESS \) and \( \MEDIUM \) belong to~\( \REGULARACTIONPAYOFFSET \).
As \( \bp{ \REGULARACTIONPAYOFFSET \cup \na{\emptyset}, \oplus, \otimes } \) is a dioid
(see Definition~\ref{de:4C_action_set}), we get that 
\( {\LESS\otimes\MEDIUM} \in \REGULARACTIONPAYOFFSET \).
Then, from \( \RestrictedSupportFunction{\MORE}=\RestrictedSupportFunction{{\LESS\otimes\MEDIUM}} \)
with \( \MORE, {\LESS\otimes\MEDIUM} \in \REGULARACTIONPAYOFFSET \),
we deduce that \( \MORE = {\LESS\otimes\MEDIUM} \) as the mapping~\( \FromACTIONPAYOFFSETtoConvBELIEF \)
in~\eqref{eq:isomorphism_a_continuous} is injective.

That \( \MORE = {\LESS\otimes\MEDIUM} 
={\LESS + \MEDIUM } \) 
follows from~\eqref{eq:two_operations_subsets_otimes_continuous}.
\bigskip

\noindent$\bullet$
It is obvious that Item~\ref{it:fusion_star-difference} implies Item~\ref{it:fusion}.
We show that Item~\ref{it:fusion} implies Item~\ref{it:fusion_star-difference}.

  Suppose that Item~\ref{it:fusion} holds true:
  there exists a~\RegularActionPayoffSet~\( \MEDIUM \in \REGULARACTIONPAYOFFSET \) such that
  \( \MORE=\overline{\LESS +\MEDIUM}={\LESS +\MEDIUM} \).

  It is easy to see, by~\eqref{eq:stardifference}, that \( \MORE\stardifference\LESS \) is the largest set
  \( \MEDIUM' \subset \RR^{\NATURE} \) such that  \( \LESS +\MEDIUM' \subset \MORE \).
  As \( \LESS +\MEDIUM  = \MORE \), we deduce that
  \( \MEDIUM \subset \MORE\stardifference\LESS \).
  On the one hand, we get that \( \MORE\stardifference\LESS \neq \emptyset \) as
  \( \MEDIUM \in \REGULARACTIONPAYOFFSET \), hence is not empty.
  On the other hand, we get that 
  \( \MORE=\LESS +\MEDIUM \subset \LESS + \np{\MORE\stardifference\LESS} \subset \MORE \), hence that 
  \( \MORE= \LESS + \np{\MORE\stardifference\LESS} \).
  As \( \MORE \) is closed, we also get that 
  \( \MORE= \overline{\LESS + \np{\MORE\stardifference\LESS}} \).

  There remains to prove that \( \MORE\stardifference\LESS \in \REGULARACTIONPAYOFFSET \).
  First, we prove that \( \MORE\stardifference\LESS  \in \ACTIONPAYOFFSET \) because
  \begin{align*}
    \np{\MORE\stardifference\LESS} \otimes \RR_{-}^{\NATURE}
    &=
      \overline{ \np{\MORE\stardifference\LESS} + \RR_{-}^{\NATURE} }
      \tag{by~\eqref{eq:two_operations_subsets_otimes_alternative_expression}}
      \\
    &=
      \overline{ \bp{ \bigcap_{\primal\in \LESS} \np{\MORE -\primal} } + \RR_{-}^{\NATURE} }
      \tag{by definition~\eqref{eq:stardifference} of \( \MORE\stardifference\LESS \)}
    \\
    & \subset
      \overline{ \bp{ \bigcap_{\primal\in \LESS} \np{\MORE+ \RR_{-}^{\NATURE} -\primal} }  }
      \tag{as easily seen}
    \\
    & \subset
      \overline{ \bp{ \bigcap_{\primal\in \LESS} \np{\overline{\MORE+ \RR_{-}^{\NATURE}} -\primal} }  }
      \tag{obvious}
    \\
    &=
     \overline{ \bp{ \bigcap_{\primal\in \LESS} \np{{\MORE \otimes \RR_{-}^{\NATURE}} -\primal} }  }
      \tag{by~\eqref{eq:two_operations_subsets_otimes_alternative_expression}}
    \\
    &=
      \overline{ \bp{ \bigcap_{\primal\in \LESS} \np{\MORE -\primal} }  }
\tag{as \( {\MORE \otimes \RR_{-}^{\NATURE}} = \MORE \) since \( \MORE \in \REGULARACTIONPAYOFFSET \)}
         \\
    &=
      \overline{\MORE\stardifference\LESS}
 \tag{by definition~\eqref{eq:stardifference} of \( \MORE\stardifference\LESS \)}
               \\
    &=
      \MORE\stardifference\LESS
      \intertext{as \( \MORE\stardifference\LESS = \bigcap_{\primal\in \LESS} \np{\MORE -\primal} \)
      is closed, as the intersection of closed sets by definition~\eqref{eq:stardifference},
 since \( \MORE \in \REGULARACTIONPAYOFFSET \) is closed}
    & \subset
      \overline{ \np{\MORE\stardifference\LESS} + \RR_{-}^{\NATURE} }
      \tag{obvious}
    \\
    &=
      \np{\MORE\stardifference\LESS} \otimes \RR_{-}^{\NATURE}
      \eqfinp
            \tag{by~\eqref{eq:two_operations_subsets_otimes_alternative_expression}}
  \end{align*}
  We deduce that \( \np{\MORE\stardifference\LESS} \otimes \RR_{-}^{\NATURE}
  = \MORE\stardifference\LESS \), that is, we have proven that
  \( \MORE\stardifference\LESS  \in \ACTIONPAYOFFSET \)
    by Item~\ref{it:EffectiveComprehensiveSet_otimes_ACTION} in Proposition~\ref{pr:EffectiveComprehensiveSet}.

    Second, we show that \( \MORE\stardifference\LESS \in \REGULARACTIONPAYOFFSET \).
    From \( \MORE= \LESS + \np{\MORE\stardifference\LESS} \), we deduce that
    \( \RestrictedSupportFunction{\MORE}=\RestrictedSupportFunction{\LESS}
    \LowPlus \RestrictedSupportFunction{\MORE\stardifference\LESS} \) by~\eqref{eq:support_function_Minkowski},
    hence that \( \RestrictedSupportFunction{\MORE\stardifference\LESS} =
    \RestrictedSupportFunction{\MORE}- \RestrictedSupportFunction{\LESS} \) since
    \(  \RestrictedSupportFunction{\LESS} \) takes finite values as 
    \( \LESS \in \REGULARACTIONPAYOFFSET \) by assumption. 
Since \( \MORE \) and \( \LESS \) are \RegularActionPayoffSet s by assumption,
the functions~$\RestrictedSupportFunction{\MORE}$ and $\RestrictedSupportFunction{\LESS}$
are continuous bounded (see Definition~\ref{de:4C_action_set}).
Thus, the function~\(  \RestrictedSupportFunction{\MORE\stardifference\LESS} \) is continuous bounded
as the difference of two continuous bounded functions,
and we conclude that \( \MORE\stardifference\LESS \in \REGULARACTIONPAYOFFSET \) by 
Definition~\ref{de:4C_action_set}.

  \medskip

 This ends the proof.  

\end{proof}


\subsection{Proof of Proposition~\ref{pr:additive_flexibility_more_valuable_information}}

\begin{proof}

  \begin{enumerate}
    \item
      (Proof of Item~\ref{it:additive_flexibility_more_valuable_information_CNS})

      Follows from the equivalence between Item~\ref{it:values_more_information}
--- the decision-maker~\( \LESS \oplus \ACTION \) values more information than the decision-maker~\( \LESS \) ---
      and Item~\ref{it:difference_is_convex} 
      --- the function \( \RestrictedSupportFunction{\LESS \oplus \ACTION}-\RestrictedSupportFunction{\LESS} \)
      is convex.
Now, by the homomorphism \( \FromACTIONPAYOFFSETtoConvBELIEF \colon 
\REGULARACTIONPAYOFFSET \to \RegularConv{\BELIEF} \) in~\eqref{eq:isomorphism_a_continuous},
we get that 
\( \RestrictedSupportFunction{\LESS \oplus \ACTION} =
\RestrictedSupportFunction{\LESS} \oplus \RestrictedSupportFunction{\ACTION} =
\max\ba{\RestrictedSupportFunction{\LESS}, \RestrictedSupportFunction{\ACTION} } \),
and then 
\( \RestrictedSupportFunction{\LESS \oplus \ACTION}-\RestrictedSupportFunction{\LESS}
= \max\ba{0, \RestrictedSupportFunction{\ACTION}-\RestrictedSupportFunction{\LESS}} \). 

    \item
      (Proof of Item~\ref{it:additive_flexibility_more_valuable_information_CN})

      If the decision-maker~\( \LESS \oplus \ACTION \) values more information than the decision-maker~\( \LESS \),
      we have just shown that 
      \( \fonctiondual=\max\ba{0, \RestrictedSupportFunction{\ACTION}-\RestrictedSupportFunction{\LESS}} \)
      is a convex function. Then, necessarily, its level set
      \( \na{ \fonctiondual \leq 0 } = \na{ \RestrictedSupportFunction{\ACTION}-\RestrictedSupportFunction{\LESS}
        \leq 0 } = \na{ \RestrictedSupportFunction{\ACTION} \leq \RestrictedSupportFunction{\LESS} } \)
      is a convex set.

      If the decision-maker~\( \LESS \oplus \ACTION \) values more information than the decision-maker~\( \LESS \),
      then there exists a~\RegularActionPayoffSet~\( \MEDIUM \in \REGULARACTIONPAYOFFSET \) such that
\( \LESS \oplus \ACTION =\LESS +\MEDIUM = \LESS\otimes\MEDIUM \)
by the equivalence between Item~\ref{it:values_more_information} and
Item~\ref{it:fusion} in Theorem~\ref{th:more_valuable_information}.
%
Hence any element 
\( \actionter\in \LESS \oplus \ACTION 
= \LESS+\MEDIUM \)
is of the form \( \actionter=\less+\medium \) where \( \less\in\LESS \) and \( \more\in\MORE \).
By Equation~\eqref{eq:Normal_cone_Minkowski}, we have that
\( \NormalCone\np{ \overline{\LESS +\MEDIUM}, \less+\medium }
= \NormalCone\np{\LESS,\less} \cap \NormalCone\np{\MEDIUM,\medium} \).
Thus, as \( \LESS +\MEDIUM =\overline{\LESS +\MEDIUM} \) 
by~\eqref{eq:two_operations_subsets_otimes_continuous}
and~\eqref{eq:two_operations_subsets_otimes_alternative_expression},
we get that
\( \NormalCone\np{ \LESS\oplus\ACTION, \actionter}
  = \NormalCone\np{ \LESS\oplus\ACTION, \less+\medium }
= \NormalCone\np{ \overline{\LESS +\MEDIUM}, \less+\medium }
= \NormalCone\np{\LESS,\less} \cap \NormalCone\np{\MEDIUM,\medium}
\subset \NormalCone\np{\LESS,\less} \), which is refinement.

Before going on, we show that, for any \( \action\in\ACTION \),
\begin{equation}
  \NORMALCONE\np{\LESS\oplus\ACTION,\action} =
\na{ \SupportFunction{\LESS} \leq \SupportFunction{\ACTION}=
  \scalpro{\cdot}{\action} }
\eqfinp
\label{eq:NORMALCONEnpLESSoplusACTIONaction}
\end{equation}
Indeed, we have that\footnote{%
In the same way, we prove that, for any \( \less\in\LESS \),
  \( \NORMALCONE\np{\LESS\oplus\ACTION,\less}=
  \na{ \SupportFunction{\ACTION} \leq \SupportFunction{\LESS} =\scalpro{\cdot}{\less} }
\subset \NORMALCONE\np{\LESS,\less} \), and that, for any 
\( \alpha\in\OpenIntervalOpen{0}{1} \), \( \less\in\LESS \) and \( \action\in\ACTION \),
\( \NORMALCONE\np{\LESS\oplus\ACTION,\alpha\less+(1-\alpha)\action}
=\na{ \scalpro{\cdot}{\action} = \SupportFunction{\ACTION} = \SupportFunction{\LESS} =\scalpro{\cdot}{\less} }
\subset \NORMALCONE\np{\LESS,\less} \).}
\begin{align*}
  \dual \in \NORMALCONE\np{\LESS\oplus\ACTION,\action}
  &\iff
    \SupportFunction{\LESS\oplus\ACTION}\np{\dual}=
    \scalpro{\dual}{\action}
    \tag{by~\eqref{eq:NormalCone_support_function}}
  \\
  &\iff
    \max\ba{\SupportFunction{\LESS}\np{\dual}, \SupportFunction{\ACTION}\np{\dual}}=
    \scalpro{\dual}{\action}
  \tag{by~\eqref{eq:two_operations_functions_oplus} and~\eqref{eq:two_operations_subsets_oplus}}
      \\
  &\iff
    \SupportFunction{\LESS}\np{\dual} \leq \SupportFunction{\ACTION}\np{\dual}=
    \scalpro{\dual}{\action}
    \tag{as \( \scalpro{\dual}{\action} \leq \SupportFunction{\ACTION}\np{\dual} \)
    by~\eqref{eq:support_function} since \( \action\in\ACTION \)}
    \eqfinp
\end{align*}
Now, suppose that the normal cone lattice~$\NORMALCONE(\LESS\oplus\ACTION)$ refines (is included in)
the normal cone lattice~$\NORMALCONE(\LESS)$. As a consequence,\footnote{%
  The condition that we provide is necessary and ``almost sufficient''.
  Indeed, we have seen that, for any 
\( \alpha\in \na{1}\cup\OpenIntervalOpen{0}{1} \), \( \less\in\LESS \) and \( \action\in\ACTION \),
\( \NORMALCONE\np{\LESS\oplus\ACTION,\alpha\less+(1-\alpha)\action}
=\na{ \scalpro{\cdot}{\action} = \SupportFunction{\ACTION} = \SupportFunction{\LESS} =\scalpro{\cdot}{\less} }
\subset \NORMALCONE\np{\LESS,\less} \).
  Thus, if \( \convexhull\np{\LESS\cup\ACTION} \) is closed, it is equal to
\( {\LESS\oplus\ACTION} = \closedconvexhull\np{\LESS\cup\ACTION} \), and then
  the condition~\eqref{eq:additive_flexibility_more_valuable_information_CN_bis} is equivalent to refinement.
  Else, we need to study \( \NORMALCONE\np{\LESS\oplus\ACTION,\actionter} \) where
  \( \actionter \in {\LESS\oplus\ACTION}\setminus\convexhull\np{\LESS\cup\ACTION} =
  \closedconvexhull\np{\LESS\cup\ACTION}\setminus\convexhull\np{\LESS\cup\ACTION} \).
  \label{ft:almost}
}
we get that, for any \( \action\in\ACTION \), 
there exists \( \less\in\LESS \) such that
\( \NORMALCONE\np{\LESS\oplus\ACTION,\action} \subset \NORMALCONE\np{\LESS,\less} \), that is,
by~\eqref{eq:NORMALCONEnpLESSoplusACTIONaction}, 
\begin{equation}
  \forall\action\in\ACTION \eqsepv \exists \less\in\LESS \eqsepv
\na{ \SupportFunction{\LESS} \leq \SupportFunction{\ACTION}=
  \scalpro{\cdot}{\action} }
\subset
\na{ \SupportFunction{\LESS}=\scalpro{\cdot}{\less} }
\eqfinv
\label{eq:additive_flexibility_more_valuable_information_CN_bis}
\end{equation}
from which we obtain~\eqref{eq:additive_flexibility_more_valuable_information_CN} by
the restriction~\eqref{eq:support_function_restriction}.

    \item
      (Proof of Item~\ref{it:additive_flexibility_more_valuable_information_CS})      

  If there exists \( \ACTIONbis  \in \REGULARACTIONPAYOFFSET \) such that
  \( \ACTION = \LESS \otimes \ACTIONbis \),
  then \( \LESS \oplus \ACTION = \np{\LESS \otimes \RR_{-}^{\NATURE}} \oplus \np{\LESS \otimes \ACTIONbis}
  = \LESS \otimes \np{ \RR_{-}^{\NATURE} \oplus \ACTIONbis } \).
  We conclude thanks to the equivalence between Item~\ref{it:fusion} and Item~\ref{it:values_more_information} 
  in Theorem~\ref{th:more_valuable_information}.

Suppose that there exists \( \less\in\LESS \) such that~\eqref{eq:additive_flexibility_more_valuable_information_CS} holds true, that is, 
\( \na{ \RestrictedSupportFunction{\ACTION} \geq \scalpro{\cdot}{\less} }
\subset
\na{ \RestrictedSupportFunction{\ACTION} \geq \RestrictedSupportFunction{\LESS} }
\subset 
\na{ \RestrictedSupportFunction{\LESS}=\scalpro{\cdot}{\less} }
\).
We have that
\begin{align*}
  \RestrictedSupportFunction{\LESS \oplus \ACTION}-\RestrictedSupportFunction{\LESS}
  &=
  \max\ba{0,\RestrictedSupportFunction{\ACTION}-\RestrictedSupportFunction{\LESS}}
  \\
  &=
  \begin{cases}
    \RestrictedSupportFunction{\ACTION}-\RestrictedSupportFunction{\LESS}
    & \text{on }
    \na{ \RestrictedSupportFunction{\ACTION} \geq \RestrictedSupportFunction{\LESS} }
  \\
  0
    & \text{on }
    \na{ \RestrictedSupportFunction{\ACTION} < \RestrictedSupportFunction{\LESS} }
\end{cases}
  \\
  &=
  \begin{cases}
    \RestrictedSupportFunction{\ACTION}-\scalpro{\cdot}{\less} 
    & \text{on }
    \na{ \RestrictedSupportFunction{\ACTION} \geq \RestrictedSupportFunction{\LESS} }
  \\
  0
    & \text{on }
    \na{ \RestrictedSupportFunction{\ACTION} < \RestrictedSupportFunction{\LESS} }
  \end{cases}
      \intertext{by assumption that \( \na{ \RestrictedSupportFunction{\ACTION} \geq \RestrictedSupportFunction{\LESS} }
\subset 
\na{ \RestrictedSupportFunction{\LESS}=\scalpro{\cdot}{\less} }
      \) in~\eqref{eq:additive_flexibility_more_valuable_information_CS}}
  &=
  \begin{cases}
    \RestrictedSupportFunction{\ACTION}-\scalpro{\cdot}{\less} 
    & \text{on }
    \na{ \RestrictedSupportFunction{\ACTION} \geq \scalpro{\cdot}{\less} }
  \\
  0
    & \text{on }
    \na{ \RestrictedSupportFunction{\ACTION} < \scalpro{\cdot}{\less} } 
  \end{cases}
      \intertext{because \( \na{ \RestrictedSupportFunction{\ACTION} \geq \RestrictedSupportFunction{\LESS} }
      = \na{ \RestrictedSupportFunction{\ACTION} \geq \scalpro{\cdot}{\less} } \) as, on the one hand,
\( \na{ \RestrictedSupportFunction{\ACTION} \geq \RestrictedSupportFunction{\LESS} }
      \subset \na{ \RestrictedSupportFunction{\ACTION} \geq \scalpro{\cdot}{\less} } \)
      by definition~\eqref{eq:support_function} of the support function~\( \SupportFunction{\LESS} \)
      since \( \less\in\LESS \) by assumption,
      and as, on the other hand, \( \na{ \RestrictedSupportFunction{\ACTION} \geq \scalpro{\cdot}{\less} }
\subset
\na{ \RestrictedSupportFunction{\ACTION} \geq \RestrictedSupportFunction{\LESS} } \)
by assumption~\eqref{eq:additive_flexibility_more_valuable_information_CS}
      }
    &=
      \max\ba{0,\RestrictedSupportFunction{\ACTION}-\scalpro{\cdot}{\less} }
  \\
    &=
      \max\ba{\RestrictedSupportFunction{\RR_{-}^{\NATURE}},
      \RestrictedSupportFunction{\ACTION-\less} }
      \tag{as \( 0=\RestrictedSupportFunction{\RR_{-}^{\NATURE}} \) and
      \( \RestrictedSupportFunction{\ACTION}-\scalpro{\cdot}{\less}=
      \RestrictedSupportFunction{\ACTION}+ \RestrictedSupportFunction{-\na{\less}}=
      \RestrictedSupportFunction{\ACTION-\less} \)  by~\eqref{eq:support_function_Minkowski}}
  \\
    &=
      \RestrictedSupportFunction{\RR_{-}^{\NATURE}} \oplus
      \RestrictedSupportFunction{\ACTION-\less}
            \tag{by definition~\eqref{eq:two_operations_functions_oplus} of~$\oplus$}
  \\
    &=
      \RestrictedSupportFunction{\RR_{-}^{\NATURE} \oplus \np{\ACTION-\less}}
       \tag{by~\eqref{eq:RestrictedSupportFunction_homomorphism_oplus}}
      \eqfinp
\end{align*}
As this last function is 
convex, 
we conclude that 
the decision-maker~\( \LESS \oplus \ACTION \) values more information than the decision-maker~\( \LESS \)
thanks to Item~\ref{it:additive_flexibility_more_valuable_information_CNS} proven above.

We have just established that \( \RestrictedSupportFunction{\LESS \oplus \ACTION}-\RestrictedSupportFunction{\LESS}
= \RestrictedSupportFunction{\RR_{-}^{\NATURE} \oplus \np{\ACTION-\less}} \).
As \( \LESS \in \REGULARACTIONPAYOFFSET \), the function~\( \RestrictedSupportFunction{\LESS} \)
takes finite values, and we deduce that
\begin{equation*}
\RestrictedSupportFunction{\LESS \oplus \ACTION} = \RestrictedSupportFunction{\LESS} +
\RestrictedSupportFunction{\RR_{-}^{\NATURE} \oplus \np{\ACTION-\less}}
= \RestrictedSupportFunction{\LESS} \otimes
\RestrictedSupportFunction{\RR_{-}^{\NATURE} \oplus \np{\ACTION-\less}}
= \RestrictedSupportFunction{\LESS \otimes \np{\RR_{-}^{\NATURE} \oplus \np{\ACTION-\less} } }
\eqfinv
\end{equation*}
by~\eqref{eq:two_operations_functions_otimes}
and~\eqref{eq:RestrictedSupportFunction_homomorphism_otimes}.
Now, all sets \( \LESS \), \( \ACTION \), \( \RR_{-}^{\NATURE} \) and \( \ACTION-\less \) belong to~\( \REGULARACTIONPAYOFFSET \).
As \( \bp{ \REGULARACTIONPAYOFFSET \cup \na{\emptyset}, \oplus, \otimes } \) is a dioid
by Item~\ref{it:isomorphism_continuous_subsets} in Proposition~\ref{pr:isomorphism_continuous}, we get that
\( {\LESS \oplus \ACTION} = {\LESS \otimes \np{\RR_{-}^{\NATURE} \oplus \np{\ACTION-\less} } } \)
since the mapping~\( \FromACTIONPAYOFFSETtoConvBELIEF \)
in~\eqref{eq:isomorphism_a} is injective.
    \end{enumerate}
      
\end{proof}

\subsection{Proof of Proposition~\ref{pr:additive_flexibility_more_valuable_information_little}}
\label{Proof_of_Proposition_ref_pr:additive_flexibility_more_valuable_information_little}

\begin{proof}

  \begin{enumerate}
    \item
      (Proof of Item~\ref{it:additive_flexibility_more_valuable_information_little_CNS})

      We apply Item~\ref{it:additive_flexibility_more_valuable_information_CNS}
      in Proposition~\ref{pr:additive_flexibility_more_valuable_information}
      to the case where \( \ACTION = \RR_{-}^{\NATURE} \oplus \na{\hat{\neutral}} \),
      hence \( \RestrictedSupportFunction{\ACTION}= \scalpro{\cdot}{\hat{\neutral}} \),
      and obtain that the decision-maker~\( \LESS \oplus \np{ \RR_{-}^{\NATURE} \otimes \na{\hat{\neutral}} } \)
      values more information than the decision-maker~\( \LESS \)
if and only if 
\( \max\ba{0, \scalpro{\cdot}{\hat{\neutral}}-\RestrictedSupportFunction{\LESS}} \) is a convex function
on~\( \BELIEF \). 

    \item
      (Proof of Item~\ref{it:additive_flexibility_more_valuable_information_little_CN})

      Suppose that the decision-maker~\( \LESS \oplus \np{ \RR_{-}^{\NATURE}
        \otimes \na{\hat{\neutral}} } \)
      values more information than the decision-maker~\( \LESS \).
      We apply Item~\ref{it:additive_flexibility_more_valuable_information_CN}
      in Proposition~\ref{pr:additive_flexibility_more_valuable_information}
      to the case where \( \ACTION = \RR_{-}^{\NATURE} \oplus \na{\hat{\neutral}} \),
      and obtain that
      \begin{itemize}
      \item
 \( \na{ \scalpro{\cdot}{\hat{\neutral}} \leq\RestrictedSupportFunction{\LESS} } \) is a convex subset
      of~\( \BELIEF \), using that         
      \( \RestrictedSupportFunction{\ACTION}= \scalpro{\cdot}{\hat{\neutral}} \),
    \item 
the normal cone lattice~$\NORMALCONE(\LESS\oplus\na{\hat{\neutral}})$ refines (is included in)
the normal cone lattice~$\NORMALCONE(\LESS)$.
\end{itemize}
Now, by~\eqref{eq:NORMALCONEnpLESSoplusACTIONaction}, we get that
\( \NORMALCONE\bp{\LESS\oplus\na{\hat{\neutral}},\hat{\neutral}} =
\ba{ \max\na{\SupportFunction{\LESS},\scalpro{\cdot}{\hat{\neutral}}}
  = \scalpro{\cdot}{\hat{\neutral}} } = \na{ \scalpro{\cdot}{\hat{\neutral}} \geq \SupportFunction{\LESS} } \) and,
 by~\eqref{eq:Normal_cone_Minkowski}, that
\( \NORMALCONE\np{\LESS,\less} = \na{ \scalpro{\cdot}{\less} = \SupportFunction{\LESS} } \).
As a consequence, there exists \( \less\in\LESS \) such that 
\( \NORMALCONE\bp{\LESS\oplus\na{\hat{\neutral}},\hat{\neutral}} \subset
\NORMALCONE\np{\LESS,\less} \) if and only if 
there exists \( \less\in\LESS \) such that 
\( \na{ \scalpro{\cdot}{\hat{\neutral}} \geq \SupportFunction{\LESS} }
\subset \na{ \scalpro{\cdot}{\less} = \SupportFunction{\LESS} } \).
We obtain~\eqref{eq:additive_flexibility_more_valuable_information_little_CN} by intersecting the
above equality with~\( \BELIEF \) and then using the restriction~\eqref{eq:support_function_restriction}
of the support function~\eqref{eq:support_function}.

    \item
      (Proof of Item~\ref{it:additive_flexibility_more_valuable_information_little_CS})

Suppose that there exists \( \less\in\LESS \) such
that~\eqref{eq:additive_flexibility_more_valuable_information_little_CS}
holds true.
By appling Item~\ref{it:additive_flexibility_more_valuable_information_CS}
      in Proposition~\ref{pr:additive_flexibility_more_valuable_information}
      to the case where \( \ACTION = \RR_{-}^{\NATURE} \otimes \na{\hat{\neutral}} \),
      hence \( \RestrictedSupportFunction{\ACTION}= \scalpro{\cdot}{\hat{\neutral}} \),
      we obtain that the decision-maker~\( \LESS \oplus \np{ \RR_{-}^{\NATURE} \otimes \na{\hat{\neutral}} } \)
      values more information than the decision-maker~\( \LESS \),
      and that Equations~\eqref{eq:additive_flexibility_more_valuable_information_little_CS_convex_function}
      and~\eqref{eq:additive_flexibility_more_valuable_information_little_CS} hold true,
      using that
      \begin{align*}
        \RR_{-}^{\NATURE} \oplus \np{\ACTION-\less}
        &=
          \RR_{-}^{\NATURE} \oplus \bp{ \np{\RR_{-}^{\NATURE} \otimes \na{\hat{\neutral}} }-\less}
          \tag{as \( \ACTION = \RR_{-}^{\NATURE} \otimes \na{\hat{\neutral}} \)}
        \\
        &=
         \np{ \RR_{-}^{\NATURE}\otimes \na{0} } \oplus \np{ \RR_{-}^{\NATURE} \otimes \na{\hat{\neutral}-\less} }
\tag{easy from~\eqref{eq:two_operations_subsets_otimes}}
        \\
        &=
          \RR_{-}^{\NATURE}\otimes \np{ \na{0} \oplus \na{\hat{\neutral}-\less} }
          \tag{by distributivity}
        \\
        &=
          \RR_{-}^{\NATURE}\otimes \ClosedIntervalClosed{0}{\hat{\neutral}-\less}
          \eqfinv
          \tag{by definition~\eqref{eq:two_operations_subsets_oplus} of~\( \oplus \)}
      \end{align*}
 where \( \ClosedIntervalClosed{0}{\hat{\neutral}-\less} \subset \RR^{\NATURE} \)
 denotes the segment between~$0$ and~\( \hat{\neutral}-\less \).

  \end{enumerate}
      
\end{proof}

\bibliographystyle{alpha}
\bibliography{DeLara,InformationMoreValuable}

\end{document}